\documentclass[preprint]{elsarticle}

\usepackage{amssymb}

\usepackage{amsmath}
\usepackage{amsthm}
\usepackage{tikz}
\usepackage{multicol}
\usetikzlibrary{arrows}
\usetikzlibrary{shapes}
\usetikzlibrary{plotmarks}
\usepackage{tikz-cd} 

\usepackage{mathtools}
\usepackage{scalerel}

\newtheorem{theorem}{Theorem}[section]
\newtheorem{lemma}[theorem]{Lemma}
\newtheorem{proposition}[theorem]{Proposition}
\newtheorem{corollary}[theorem]{Corollary}

\theoremstyle{definition}

\newtheorem{example}[theorem]{Example}

\newtheorem{remark}[theorem]{Remark}

\newcommand{\meet}{\wedge}
\newcommand{\join}{\vee}
\newcommand{\under}{\backslash}
\newcommand{\ovr}{\slash}

\newcommand\upset{\mathord{\uparrow}}
\newcommand{\m}{\mathbf}

\newcommand{\downset}{\downarrow}

\newcommand\g[1]{g_{\bf{#1}}}
\newcommand\f[1]{f_{\bf{#1}}}
\newcommand{\gb}{\g{B}}
\newcommand{\gc}{\g{C}}
\newcommand{\fb}{\f{B}}
\newcommand{\fc}{\f{C}}

\newcommand{\rd}{{/}}
\newcommand{\ld}{{\backslash}}
\newcommand{\ra}{\mathbin{\rightarrow}}
\newcommand{\jn}{\vee}
\newcommand{\mt}{\wedge}
\newcommand{\Lra}{\Leftrightarrow}
\newcommand{\lra}{\leftrightarrow}
\newcommand{\fm}{\mathsf{Fm}}

\newcommand{\var}{\mathsf{Var}}

\newcommand{\CIL}{{\bf sCI}}
\newcommand{\SIL}{{\bf sLI}}

\journal{Annals of Pure and Applied Logic}

\begin{document}

\begin{frontmatter}



\title{Semiconic Idempotent Logic II:\\ Beth Definability and Deductive Interpolation}

 \author[label1]{Wesley Fussner}
 \affiliation[label1]{organization={Mathematical Institute, University of Bern},
             country={Switzerland}}

 \author[label2]{Nikolaos Galatos}
 \affiliation[label2]{organization={Department of Mathematics, University of Denver},
             country={USA}}




\begin{abstract}
Semiconic idempotent logic $\CIL$ is a common generalization of intuitionistic logic, semilinear idempotent logic $\SIL$, and in particular relevance logic with mingle. We establish the projective Beth definability property and the deductive interpolation property for many extensions of $\m{\CIL}$, and identify extensions where these properties fail. We achieve these results by studying the (strong) amalgamation property and the epimorphism-surjectivity property for the corresponding algebraic semantics, viz. semiconic idempotent residuated lattices. Our study is made possible by the structural decomposition of conic idempotent models achieved in the prequel, as well as a detailed analysis of the structure of idempotent residuated chains serving as index sets in this decomposition. Here we study the latter on two levels: as certain enriched Galois connections and as enhanced monoidal preorders. Using this, we show that although conic idempotent residuated lattices do not have the amalgamation property, the natural class of rigid and conjunctive conic idempotent residuated lattices has the strong amalgamation property, and thus has surjective epimorphisms. This extends to the variety generated by rigid and conjunctive conic idempotent residuated lattices, and we establish the (strong) amalgamation and epimorphism-surjectivity properties for several important subvarieties. Using the algebraizability of $\m{\CIL}$, this yields the deductive interpolation property and the projective Beth definability property for the corresponding substructural logics extending $\m{\CIL}$.
\end{abstract}




\end{frontmatter}



\section{Introduction}

This is the second in our series of papers on semiconic idempotent logic $\CIL$, a substructural logic comprising a common framework for the study of superintuitionistic logics and semilinear idempotent logics. In the prequel to this paper \cite{FG1}, we provide a detailed structural analysis of the algebraic models of $\CIL$ and deploy this to obtain local deduction theorems for all extensions of $\CIL$. In the present paper, we apply the results of \cite{FG1} to study the Beth definability and deductive interpolation property in $\CIL$ and its extensions. Our study advances the literature on interpolation and Beth definability in substructural logics on several different fronts, but two aspects of this investigation are especially notable. First, our work adds to Maksimova's celebrated classification of superintuitionistic logics with interpolation \cite{Mak77}, and in particular extends Maksimova's work to certain logics without the weakening rule. Second, in this study we identify several logics without the exchange rule that enjoy the deductive interpolation property. Until recently it was unknown whether there are any logics without exchange that enjoy deductive interpolation (cf. \cite{GJM2020}), and here we catalog several natural examples of this phenomenon. Our investigation thus contributes to the study both of logics without weakening and logics without exchange.

We refer the reader to \cite[Section~1]{FG1} for a detailed discussion of substructural logics and their motivation in relation to this study, but recall here that substructural logics are resource-sensitive systems of reasoning that generalize both classical and intuitionistic logic. Typically, they are formulated as extensions of the full Lambek calculus $\m {FL}$, which is often presented by a Gentzen-style sequent calculus. Substructural logics may lack some of the main structural rules appearing in proof theory (contraction, weakening, and exchange), and this drives their widespread application in areas spanning computer science \cite{girard87,OHearnPym1999a}, philosophy \cite{AB75}, physics \cite{FStJ2021,FLP21}, and other areas. $\m {FL}$ and its axiomatic extensions are algebraizable logics in the Blok and Pigozzi's sense, and their equivalent algebraic semantics are given by varieties of residuated lattices; see \cite{GJKO2007}.

Idempotent logic, axiomatized relative to $\m {FL}$ by the contraction and mingle rules, is one of the most prominent substructural logics. It includes intuitionistic logic and relevance logic with mingle; see \cite{FG2019}. On the other hand, semilinear logic is the extension of $\m {FL}$ modeled by the variety generated by totally ordered residuated lattices, and generalizes G\"{o}del-Dummett logic \cite{met:god}, \L{}ukasiewicz logic \cite{CDM2000,FGGM2021}, H\'{a}jek's basic fuzzy logic \cite{Haj98,F2021,FZ2021}, and more generally monoidal t-norm based logics \cite{estgodo:mtl,FU2019}. Semilinear idempotent logic $\SIL$ is the intersection of idempotent logic and semilinear logic, and has been studied quite extensively via research on its algebraic models \cite{R2007,GJM2020,Ga2005,CZ2009}.

Our focus in this paper and in its prequel is semiconic idempotent logic $\CIL$, a generalization of $\SIL$ that also encompasses idempotent logics with the weakening rule, such as intuitionistic logic. In \cite{FG1} we give an analytical hypersequent axiomatization for $\m {\CIL}$ and study its algebraic semantics, the variety of semiconic idempotent residuated lattices. In particular, we prove a structural decomposition theorem for the finitely subdirectly irreducible members of the variety (conic idempotent residuated lattices) and use it to show that $\m {\CIL}$ has a local deduction theorem, among other results. In the present paper, we prove that the deductive interpolation property fails in both $\CIL$ and $\SIL$, and pinpoint the reason for these failures on a semantic level. Using this semantic insight, we identify several natural extensions of $\CIL$ and $\SIL$ that have both the deductive interpolation property and the Beth definability property. Our results rest on the structural description of conic idempotent residuated lattices obtained in \cite{FG1}, but require a further sharpening of this analysis. In particular, our work demands a thorough investigation of idempotent residuated chains appearing as `skeletons' in the decomposition of conic residuated lattices in \cite{FG1}. Therefore, the present study particularly contributes to the study of $\SIL$.

\paragraph{Outline of the paper} In Section~\ref{sec:connections}, we define the logic $\m {\CIL}$ and its algebraic semantics, viz.  semiconic idempotent residuated lattices. Further, we review various logical properties and corresponding bridge theorems that establish the connections between logical and algebraic properties. These include bridge theorems between local deduction theorems and the congruence extension property, between the deductive interpolation and amalgamation properties, and between the Beth definability property and epimorphism surjectivity.

In Section~\ref{sec:basicstructure}, we state the results from the prequel \cite{FG1} that will be used in this paper. In particular, we explain the importance of the inverse operations $^\ell$ and $^r$ in conic idempotent residuated lattices and how they allow us to get a structural decomposition via a nucleus. Moreover, we present a full description of congruence generation for the idempotent semilinear case and show how it leads to the proof of a local deduction theorem and a characterization of the finitely subdirectly irreducible semiconic idempotent residuated lattices.

In Section~\ref{sec:chains}, we describe totally ordered idempotent residuated chains, aiming in particular to describe the structure of those appearing as skeletons in the decomposition of Section~\ref{sec:basicstructure}, viz. quasi-involutive idempotent residuated chains. In Section~\ref{s:idempotent Galois connections}, we characterize the $\{\meet,\join,{}^r,{}^\ell,1\}$-reducts of idempotent residuated chains as a universal class consisting of certain enriched Galois connections, which we call \emph{idempotent Galois connections}. We show that idempotent Galois connections are definitionally equivalent to idempotent residuated chains. Then, in Section~\ref{s:flow diagrams}, we introduce \emph{flow diagrams} as a diagrammatic tool for understanding the action of the inversion operations in idempotent residuated chains. Flow diagrams provide a bridge between idempotent Galois connections and \emph{enhanced monoidal preorders}, introduced in Sections~\ref{s:ircs to empos} and \ref{s:empos to ircs} as another definitionally equivalent presentation of idempotent residuated chains. Enhanced monoidal preorders improve on the monoidal preorder representations considered in \cite{GJM2020} by expanding the latter by information about which elements are positive and negative, as well as with a minimality and maximality condition that is crucial in the infinite setting. In conjunction with idempotent Galois connections and flow diagrams, enhanced monoidal preorders provide a pictorial and combinatorial description of arbitrary idempotent residuated chains. In Section~\ref{sec:subalgebras of chains}, we describe subalgebra generation in terms of enhanced monoidal preorders in a transparent way. Then, in Section~\ref{s:weaklyinvolutive}, we identify a fundamental problem with amalgamating quasi-involutive idempotent residuated chains, and introduce the class of $^\star$-involutive idempotent residuated chains in order to rectify this problem; algebras that are not $^\star$-involutive always give rise to V-formations that do not possess an amalgam. In Section~\ref{s:one-generated}, we completely characterize one-generated $^\star$-involutive idempotent residuated chains in terms of their enhanced monoidal preorders and in Section~\ref{nestedsums} characterize arbitrary $^\star$-involutive idempotent residuated chains as nested sums of one-generated ones.

In Section~\ref{s:amalgamation}, we combine the ingredients assembled in the previous sections in order to study amalgamation in conic idempotent residuated lattices. We illustrate in Subsection~\ref{s:APfailschains} that the amalgamation property fails for idempotent residuated chains (via non $^\star$-involutive examples). Actually, we prove that there are V-formations of idempotent residuated chains that do not have an amalgam in the whole variety of semilinear idempotent residuated lattices, thus the latter lacks the amalgamation property. Consequently, $\SIL$ does not have the deductive interpolation property. Further, in Section~\ref{s:failure for rigid conic}, we show that the amalgamation property fails even for \emph{rigid} conic idempotent residuated lattices (i.e., those whose skeleton is $^\star$-involutive); in fact, the amalgamation property fails even for the variety of rigid semiconic idempotent residuated lattices and its commutative subvariety (Theorem~\ref{thm:failures rigid variety}). The problem causing the failure of amalgamation is that some blocks $\gamma^{-1}(a)$ may not be lattices. This problem is avoided if and only if the conic idempotent residuated lattice is \emph{conjunctive} (i.e., satisfies $\gamma(x \mt y)=\gamma(x) \mt \gamma(y)$). In Theorem~\ref{t:sAPqichains} we show that the class of $^\star$-involutive idempotent residuated chains has the strong amalgamation property, and in Theorem~\ref{t:sAPconic} that the class of rigid conjunctive conic idempotent residuated lattices has the strong amalgamation property. Using some new tools on extending the strong amalgamation property, we show that the varieties of $^\star$-involutive semilinear idempotent residuated lattices and rigid conjunctive semiconic idempotent residuated lattices both have the strong amalgamation property. In Section~\ref{s:amalg subvarieties}, we further establish the amalgamation and strong amalgamation property for several subvarieties, notably studying this property for the commutative subvarieties. As a consequence of the strong amalgamation property, we also obtain the epimorphism-surjectivity property for a number of the varieties we consider. We conclude by deducing from these results the deductive interpolation property and the projective Beth definability property for the corresponding extensions of $\m {\CIL}$.

\section{Semiconic idempotent logic and logical preliminaries}\label{sec:connections}

\subsection{Semiconic idempotent logic and its semantics} Semiconic idempotent logic $\m {\CIL}$ is an axiomatic extension of the basic substructural logic $\m {FL}$ (full Lambek calculus) and it admits an analytic hypersequent calculus; see \cite{FG1}. The logic  $\m {FL}$ is algebraizable and its algebraic semantics is the variety of residuated lattices. Here we introduce $\CIL$ in terms of its equivalent algebraic semantics, directing the reader to \cite{FG1} for further discussion and a proof-theoretic formulation.

A \emph{residuated lattice} is an algebra ${\bf A} = (A,\meet,\join,\cdot,\under,\ovr,1)$ such that $(A,\meet,\join)$ is a lattice, $(A,\cdot,1)$ is a monoid, and $\cdot,\under,\ovr$ are binary operations on $A$ with
$$y\leq x\under z \iff x\cdot y\leq z \iff x\leq z\ovr y$$
for all $x,y,z\in A$.
When $\cdot$ is commutative, we have $ x\under y = y\ovr x$ for all $x,y$ and their common value is denoted by $x\to y$.

In any residuated lattice ${\m A}$, we may define two \emph{inverse} operations by $x \mapsto x^r =x\under 1$ and $x \mapsto x^\ell = 1\ovr x$. These play a very important role in the classes of residuated lattices studied in this paper. We denote by $A^i$ the set of all inverse elements of $A$, i.e., $A^i=\{a^\ell : a \in A\} \cup \{a^r : a \in A\}$. Given a residuated lattice ${\m A}$, a  \emph{nucleus} on ${\m A}$ is a closure operator $\gamma$ on the lattice reduct of ${\m A}$ that satisfies $\gamma (x)\cdot\gamma (y)\leq\gamma (x\cdot y)$. If ${\m A}$ is a residuated lattice and $\gamma$ is a nucleus on ${\m A}$, then the image of $A$ under $\gamma$ forms a residuated lattice ${\m A}_\gamma$, where the operations $\meet,\under,\ovr$ are inherited from ${\m A}$, $\gamma(1)$ is the new unit, and the new product and join are defined by $(x,y)\mapsto \gamma(x\join y)$ and $(x,y)\mapsto \gamma(x\cdot y)$, respectively. For the residuated lattices implicated in our study, we obtain the set of inverses as a nuclear image of an appropriately chosen nucleus; see Section~\ref{subsec:skeleton}.

A residuated lattice ${\bf A}$ is called \emph{conic} if $a\leq 1$ or $1\leq a$ every $a\in A$. Our terminology derives from the fact that a residuated lattice $\m A$ is conic precisely when $A=A^- \cup A^+$, where $A^-=\{x \in A: x \leq 1\}$ is the  \emph{negative cone} and $A^+=\{x \in A: x \geq 1\}$ is the \emph{positive cone} of $\m A$. Elements of $A^+$ are said to be \emph{positive} or to have \emph{positive sign} and elements of $A^-$ are  \emph{negative} or have  \emph{negative sign}; note that the element $1$ has both signs, so it has simultaneously the same sign and the opposite sign of any other element. A residuated lattice ${\bf A}$ is \emph{idempotent} if it satisfies the identity $x^2=x$. It is called \emph{integral} if it satisfies $x \leq 1$, \emph{linearly ordered} if every element is comparable to $1$, and \emph{semilinear} if it is a subdirect product of linearly ordered residuated lattices. Thus, conic residuated lattices simultaneously generalize integral and linearly ordered residuated lattices.

\emph{Semiconic residuated lattices} are members of the variety generated by the conic residuated lattices. This variety is axiomatized relative to all residuated lattices by the identities
$$1 = \gamma_1(x\meet 1)\join \gamma_2((x\ld 1)\meet 1),$$
where $\gamma_1$ and $\gamma_2$ range over all iterated conjugates.\footnote{Recall that the left and right \emph{conjugate} of $a$ by $x$ are the elements $\lambda_x(a)=x \ld ax \mt 1 $ and $\rho_x(a)=xa \rd x \mt 1$. An \emph{iterated conjugate} is an arbitrary compositions of left and right conjugates with no restriction to the conjugating elements used.}

The logic  $\m {\CIL}$ is algebraizable and its algebraic semantics are semiconic idempotent residuated lattices; see \cite{GO2006}.
We denote by $\vdash_{\m {\CIL}}$ the deductive system (finitary consequence relation) related to $\m {\CIL}$. Since this system is algebraizable, in particular it is equivalential; see~\cite{FG1,BP1989,F2016}.

\subsection{Interpolation and Beth definability} This paper's chief goal is to establish the deductive interpolation property and the projective Beth definability property for various extensions of $\m {\CIL}$, and we now briefly discuss these properties. See \cite{FG1,BP1989,F2016} for further information.

Recall that in general a deductive system $\vdash$ over a language $\mathcal{L}$ has the \emph{(deductive) interpolation property} if 
\begin{center}
	whenever $\Gamma\vdash\varphi$, there exists a set of formulas $\Gamma'\subseteq\fm_\mathcal{L}$ with $\var(\Gamma')\subseteq\var(\Gamma)\cap\var(\varphi)$ and $\Gamma\vdash\Gamma'$, $\Gamma'\vdash\varphi$.\footnote{Here $\fm_\mathcal{L}$ denotes the collection of $\mathcal{L}$-formulas as usual, and $\var(\varphi)$ denotes the collection of propositional variables appearing in the formulas $\varphi$. The latter notation is extended sets of formulas in the obvious way: $\var(\Gamma) = \bigcup \{\var(\varphi) : \varphi\in\Gamma\}$.}
\end{center}
Intuitively, the \emph{interpolant} $\Gamma'$ gives an `explanation' for why $\varphi$ follows from $\Gamma$ within the logic $\vdash$.

Now suppose that $\vdash$ is an equivalential deductive system with the set of equivalence formulas $\Delta$, and let $X,Y,Z$ be pairwise disjoint sets of variables with $X\neq\emptyset$. Let $\Gamma\subseteq\fm_\mathcal{L}$ with $\var(\Gamma)\subseteq X\cup Y\cup Z$. The set of formulas $\Gamma$ is said to \emph{implicitly define $Z$ in terms of $X$ via $Y$} if 
\begin{center}
	for every $z\in Z$ and every substitution $\sigma$ with $\sigma(x)=x$ for all $x\in X$,\\ we have $\Gamma\cup\sigma[\Gamma]\vdash\Delta(z,\sigma(z))$.
\end{center}
 When $Y=\emptyset$, we simply say that $\Gamma$ \emph{implicitly defines $Z$ in terms of $X$}. We also say that $\Gamma$ \emph{explicitly defines $Z$ in terms of $X$ via $Y$} if 
\begin{center}
	for every $z\in Z$ there exists a formula $t_z$ with $\var(t_z)\subseteq X$ such that $\Gamma\vdash \Delta(z,t_z)$.
\end{center}
 When $Y=\emptyset$, we just say that $\Gamma$ \emph{explicitly defines $Z$ in terms of $X$}.

We say that an equivalential deductive system $\vdash$ has the \emph{infinite Beth definability property} if for any $\Gamma\subseteq\fm_\mathcal{L}$ with $\var(\Gamma)\subseteq X\cup Z$, we have that 
\begin{center}
	if $\Gamma$ implicitly defines $Z$ in terms of $X$,\\ then $\Gamma$ also explicitly defines $Z$ in terms of $X$.
\end{center}
 On the other hand, $\vdash$ is said to have the \emph{finite Beth definability property} if we add that $Z$ is finite in the previous definition. Furthermore, $\vdash$ has the \emph{projective Beth definability property} if for any $\Gamma\subseteq\fm_\mathcal{L}$ with $\var(\Gamma)\subseteq X\cup Y\cup Z$, if $\Gamma$ implicitly defines $Z$ in terms of $X$ via $Y$, then $\Gamma$ also explicitly defines $Z$ in terms of $X$ via $Y$.


\subsection{The bridge theorems} The algebraic study of deductive systems is motivated, in part, by correspondences between properties like those just introduced and various associated algebraic properties. `Bridge theorems' announcing connections of this sort may be found throughout the literature. We will use several of these bridge theorems.

To articulate the first of these, recall that an algebra ${\m B}$ has the \emph{congruence extension property} (or \emph{CEP}) if for any subalgebra ${\m A}$ of ${\m B}$, if $\Theta$ is a congruence of ${\m A}$, then there exists a congruence $\Psi$ of ${\m B}$ such that $\Psi\cap A^2=\Theta$. A variety is said to have the CEP if each algebra contained in it does. Moreover, we say that a deductive system $\vdash$ has a \emph{local deduction theorem} if there exists a family $\Lambda$ of sets of binary formulas such that, for all $\Gamma\cup\{\varphi,\psi\}$, 
\begin{center}
	$\Gamma,\varphi\vdash\psi$ iff there exists $\lambda\in\Lambda$ such that for all $l\in\lambda$, $\Gamma\vdash l(\varphi,\psi)$.
\end{center}
The following bridge theorem may be found in \cite{BP1988}:
\begin{theorem}\label{t:local deduction}
Let $\vdash$ be an algebraizable deductive system whose equivalent algebraic semantics is the variety $\mathcal{V}$. Then $\vdash$ has a local deduction theorem iff $\mathcal{V}$ has the congruence extension property.
\end{theorem}
In \cite{FG1}, we apply Theorem~\ref{t:local deduction} to conclude the existence of a local deduction theorem from a proof of the congruence extension property for the algebraic models:
\begin{theorem} \cite{FG1}
Every extension of the logic $\m {\CIL}$ has a local deduction theorem.
\end{theorem}
The presence of a local deduction theorem simplifies a number of other algebra-to-logic links, such as the following result proven in \cite{CP1999}:
\begin{theorem}\label{t:deductive interpolation}
Let $\vdash$ be an algebraizable deductive system with a local deduction theorem, and suppose that the variety $\mathcal{V}$ is its equivalent algebraic semantics. Then $\vdash$ has the deductive interpolation property iff $\mathcal{V}$ has the amalgamation property.
\end{theorem}
From the previous theorem, the study of deductive interpolation for an appropriately chosen deductive system boils down to the study of amalgamation for its equivalent algebraic semantics. We thus turn to a discussion of the amalgamation property. Let $\mathcal{K}$ be a class of similar algebras. A \emph{V-formation in $\mathcal{K}$} is an ordered quintuple $({\m A},{\m B},{\m C},\fb,\fc)$, where $\m A,\m B,\m C\in \mathcal{K}$ and $\fb\colon\m A\to\m B$ and $\fc\colon\m A\to\m C$ are embeddings. Given a V-formation $V=({\m A},{\m B},{\m C},\fb,\fc)$ in $\mathcal{K}$ and a class $\mathcal{M}$ of algebras in the type of $\mathcal{K}$, an \emph{amalgam} of $V$ in $\mathcal{M}$ is an ordered triple $({\m D},\gb,\gc)$, where ${\m D}\in\mathcal{M}$ and $\gb\colon\m B\to\m D$ and $\gc\colon\m C\to\m D$ are embeddings such that $\gb\circ\fb=\gc\circ\fc$; see Figure~\ref{fig:amalgamation}. A class $\mathcal{K}$ of similar algebras is said to \emph{have the amalgamation property in $\mathcal{M}$} if every V-formation in $\mathcal{K}$ has an amalgam in $\mathcal{M}$. The class $\mathcal{K}$ is said to \emph{have the amalgamation property} if $\mathcal{K}$ has the amalgamation property in $\mathcal{K}$. 

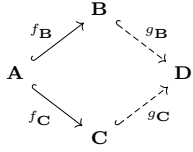
\begin{figure}[t]
\centering
\begin{scriptsize}
\begin{tikzcd}
& {\m B} \ar[dashed, hook, rd, "g_{\m B}"] &       \\
{\m A} \ar[hook, ru, "f_{\m B}"] \ar[hook, rd,"f_{\m C}"'] &      & \m D  \\
 & {\m C} \ar[dashed, hook, ru, "g_{\m C}"'] &     
\end{tikzcd}
\end{scriptsize}
\caption{The amalgamation property rendered as a commutative diagram}
\label{fig:amalgamation}
\end{figure}

We say that a V-formation $({\m A},{\m B},{\m C},\fb,\fc)$ is \emph{reduced} if ${\m A}$ is a subalgebra of each of ${\m B}$ and ${\m C}$, and $\fb$ and $\fc$ are the inclusion maps.
 Note that for classes closed under isomorphisms, we can assume without loss of generality that all V-formations are reduced. The amalgamation property ensures that the maps $\gb$ and $\gc$ do not send two different elements of $A$ to the same element of $D$, but an element of $B$ and an element of $C$ may be sent to the same element of $D$ since the intersection of the images of these maps may be properly larger than the image of $A$. In other words, $\m B$ and $\m C$  may be identified with subalgebras of $\m D$ simultaneously, but only if we are allowed to rename their elements that have the same image in $\m D$. Therefore, for classes closed under isomorphisms the amalgamation property can be stated as: If the algebras $\m B$ and $\m C$ have a common subalgebra $\m A$, then there exists algebras $\m B'$ and $\m C'$ isomorphic to $\m B$ and $\m C$, respectively, also having $\m A$ as a common subalgebra, and there exists an algebra $\m D$ that contains $\m B'$ and $\m C'$ as subalgebras.

 The situation is much more transparent if a stronger version of the amalgamation holds. 
  A class of similar algebras $\mathcal{K}$ has the \emph{strong amalgamation property} iff it has the amalgamation property and amalgams $({\m D}, g_{\m B}, g_{\m C})$ may be selected so that $(g_{\m B}\circ f_{\m B})[A]=g_{\m B}[B]\cap g_{\m C}[C]$. In this event, we say that $({\m D}, g_{\m B}, g_{\m C})$ is a \emph{strong amalgam}. 
  It is easy to see that the strong amalgamation property has a much easier formulation if the class $\mathcal{K}$ of similar algebras is closed under isomorphisms: If the algebras $\m B$ and $\m C$ in $\mathcal{K}$ intersect at a common subalgebra $\m A$, there exists an algebra $\m D$ in $\mathcal{K}$ having $\m B$ and $\m C$ as subalgebras. Note that no renaming of elements (as in $\m B'$ and $\m C'$) is necessary. 

The strong amalgamation property is closely linked to the surjectivity of epimorphisms. A variety $\mathcal{V}$ has the \emph{strong epimorphism-surjectivity property} if whenever ${\m B}\in\mathcal{V}$, ${\m A}$ is a subalgebra of ${\m B}$, and $b\in B - A$, then there exists ${\m C}\in\mathcal{V}$ and homomorphisms $f,g\colon {\m B}\to {\m C}$ such that $f\restriction_A = g\restriction_A$ but $f(b)\neq g(b)$.\footnote{Clearly, if $\mathsf{V}$ has the strong epimorphism-surjectivity property then every epimorphism between members of $\mathsf{V}$ is a surjection.} The strong amalgamation property entails the strong epimorphism-surjectivity property (see the discussion in \cite{GR2015}). Further, the following bridge theorem from \cite{H2000} links strong epimorphism surjectivity to Beth definability:
\begin{theorem}\label{t:Beth def}
Let $\vdash$ be an algebraizable deductive system and suppose that the variety $\mathcal{V}$ is its equivalent algebraic semantics. Then $\vdash$ has the projective Beth definability property iff $\mathcal{V}$ has the strong epimorphism-surjectivity property.
\end{theorem}
We will prove that several varieties have the epimorphism-surjectivity property by proving that they have the strong amalgamation property, further establishing the projective Beth definability property for the corresponding deductive systems. 

\section{The structure of conic idempotent residuated lattices}\label{sec:basicstructure}

In order to establish the deductive interpolation and projective Beth definability properties for the deductive systems we consider, we will prove that the corresponding varieties of semiconic idempotent residuated lattices have the strong amalgamation property and apply Theorems~\ref{t:deductive interpolation} and \ref{t:Beth def}. To establish the strong amalgamation property for the specified varieties, we prove in Theorem~\ref{thm:strongMMT2014} that we can focus on amalgamating within the class of finitely subdirectly irreducible algebras from the variety. In our case, these will be in particular idempotent and conic.

To prove Theorem~\ref{thm:strongMMT2014}, the ideas supporting our proof of the local deduction theorem from \cite{FG1} will be crucial. Moreover, according to \cite{FG1}, the structure of idempotent conic residuated lattices is given by an ordinal sum decomposition of blocks along a linear skeleton. Our strategy is to amalgamate blocks and skeletons separately, and then put the amalgams together by another use of the decomposition result. Skeletons are characterized as idempotent quasi-involutive residuated chains in \cite{FG1}. To amalgamate idempotent quasi-involutive residuated chains, we must develop a deep understanding of their structure in the present paper. We accomplish this by studying the behavior of the two inversion operations, and by making use of general results of \cite{FG1} on semiconic idempotent residuated lattices specialized to the linear case.

Consequently, we recall some pertinent facts from \cite{FG1}, where all proofs may be found. We start with two examples that will play an important role in the proof of amalgamation.

\begin{example}\label{ex:sugihara} A commutative, idempotent, semilinear residuated lattice is called an \emph{odd Sugihara monoid} provided the map $x\mapsto x\to 1$ is an involution. Linearly ordered odd Sugihara monoids are of the form $A=\{a_i: i \in I\} \cup \{1\} \cup \{b_i: i \in I\}$, where $I$ is a chain with $b_i<b_j<1<a_m<a_n$, for $i<j$ and $m>n$; see Figure~\ref{f:Sugihara}. Multiplication satisfies $a_ib_j=a_i$ if $i<j$,  $a_ib_j=b_j$ if $i>j$, and $a_ib_i=b_i$. The product is \emph{conservative} in the sense that $xy\in\{x,y\}$. In fact, the product of the two elements is always whichever of them is the furthest away from $1$; when two elements have the same distance from $1$, the product is the meet. Here, distance from $1$ is understood from the index set $I$ (cf. Corollary~\ref{c:idempotentchainsoperations}). In odd Sugihara monoids, we define $x^*:=x \ra 1$. Observe that the signs of $x$ and $x^*$ are opposite and that whether $y$ is closer to $1$ than $x$ is depends on whether $y$ is in the interval between $x$ and $x^*$ or not. These comments show that multiplication can be defined in terms of ${}^*$ and the order.
\end{example}

  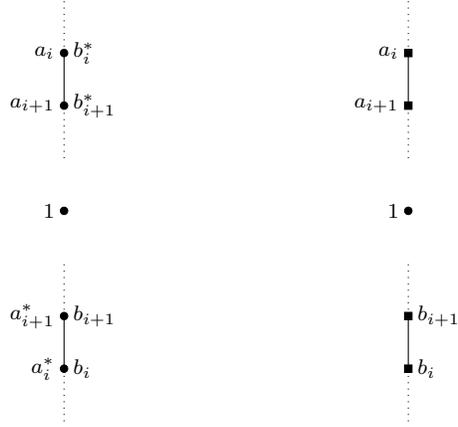
\begin{figure}[htbp]
\begin{center}
\begin{tikzpicture}[scale=0.7]
\draw[dotted] (3,6) --(3,7); 
\node[left] at (3,6) {\footnotesize  $a_i$};
\node[right] at (3,6) {\footnotesize  $b_i^*$};
\draw[mark=*]  plot coordinates {(3,6)};
\draw (3,5) --(3,6); 
\node[left] at (3,5) {\footnotesize  $a_{i+1}$};
\node[right] at (3,5) {\footnotesize  $b_{i+1}^*$};
\draw[mark=*]  plot coordinates {(3,5)};
\draw[dotted] (3,4) --(3,5); 
\node[left] at (3,3) {\footnotesize  $1$};
\draw[mark=*]  plot coordinates {(3,3)};
\draw[dotted] (3,1) --(3,2); 
\node[right] at (3,1) {\footnotesize  $b_{i+1}$};
\node[left] at (3,1) {\footnotesize  $a_{i+1}^*$};
\draw[mark=*]  plot coordinates {(3,1)};
\draw (3,0) --(3,1); 
\node[right] at (3,0) {\footnotesize  $b_i$};
\node[left] at (3,0) {\footnotesize  $a_i^*$};
\draw[mark=*]  plot coordinates {(3,0)};
\draw[dotted] (3,-1) --(3,0); 
\end{tikzpicture}
\qquad \qquad \qquad \qquad
\begin{tikzpicture}[scale=0.7]
\draw[dotted] (3,6) --(3,7); 
\node[left] at (3,6) {\footnotesize  $a_i$};
\draw[mark=square*]  plot coordinates {(3,6)};
\draw (3,5) --(3,6); 
\node[left] at (3,5) {\footnotesize  $a_{i+1}$};
\draw[mark=square*]  plot coordinates {(3,5)};
\draw[dotted] (3,4) --(3,5); 
\node[left] at (3,3) {\footnotesize  $1$};
\draw[mark=*]  plot coordinates {(3,3)};
\draw[dotted] (3,1) --(3,2); 
\node[right] at (3,1) {\footnotesize  $b_{i+1}$};
\draw[mark=square*]  plot coordinates {(3,1)};
\draw (3,0) --(3,1); 
\node[right] at (3,0) {\footnotesize  $b_i$};
\draw[mark=square*]  plot coordinates {(3,0)};
\draw[dotted] (3,-1) --(3,0); 
\end{tikzpicture}
  \end{center}
\caption{Sugihara chains and some non-commutative analogues.\label{f:Sugihara}}
\end{figure}

\begin{example}\label{e:noncom Sug}
 Non-commutative analogues of linearly ordered odd Sugihara monoids are given in \cite{Ga2005}; see Figure~\ref{f:Sugihara} on the right. The underlying order of these algebras is as in Example~\ref{ex:sugihara}, but non-commutativity introduces many possible variations. In these algebras, multiplication is encoded in a subset $J$ of the index set $I$. The product is the same as in Example~\ref{ex:sugihara} for elements with different index, but the products where there is a tie are determined by $J$: 
\begin{enumerate}
\item if $i \in J$, then $a_i b_i=a_i$ and $b_ia_i=b_i$ (Left), and
\item if $i \not \in J$, then $a_i b_i=b_i$ and $b_ia_i=a_i$ (Right).
\end{enumerate}
\end{example}
Observe that for no index $i$ do we have $a_i b_i=a_i$ and $b_ia_i=a_i$, or $a_i b_i=b_i$ and $b_ia_i=b_i$ (cf. Lemma~\ref{l:centralizer}(5)).

Note that although the algebras in Examples~\ref{ex:sugihara} and \ref{e:noncom Sug} are defined very similarly, the two examples vary considerably in terms of subalgebras. In Example~\ref{ex:sugihara}, for any $i$ the set $\{b_i, 1, a_i\}$ gives a subalgebra. Indeed, every linearly ordered odd Sugihara monoid is a \emph{nested sum} (to be defined formally in Section~\ref{nestedsums}) of copies of this 3-element algebra. In contrast, the algebras in Example~\ref{e:noncom Sug} have no proper, nontrivial subalgebras and each of them is generated by any non-identity element. Due to these comments, the two examples differ markedly when considering amalgamation.

\subsection{The skeleton of inverse elements}\label{subsec:skeleton}

In any residuated lattice ${\m A}$, we say $x\in A$ is \emph{conical} if $x\leq y$ or $y\leq x$ for all $y\in A$. When ${\m A}$ is a conic idempotent residuated lattice, the inverse elements of ${\m A}$ are conical and ${\m A}$ may be decomposed along these. Because inverse elements are conical in any conic idempotent residuated lattice, $\{a^\ell\meet a^r, a^\ell\join a^r\}=\{a^r, a^\ell\}$, and we define $a^*:= a^\ell \jn a^r$ and $a^\star=a^\ell \mt a^r$.\footnote{Observe that if $\m A$ is commutative, then $a^*=(a \ra 1) \jn (a \ra 1)=a \ra 1$ and the definition of $^*$ is the same as the one given in Example~\ref{ex:sugihara}.} We further define $\gamma(x):=x^{\ell r} \mt x^{r \ell}$ in any conic idempotent residuated lattice ${\m A}$.

The following lemma---which summarizes parts of Lemmas~3.1, 3.5, 3.9 and 3.11, Corollaries 3.6 and 3.7, and Theorem 3.11 of \cite{FG1}---indicates how the inverse elements of a conic idempotent residuated lattice form a skeleton of the whole algebra. This skeleton is a totally ordered subalgebra and it is \emph{quasi-involutive}, that is, it satisfies $x^{\ell r} \mt x^{r \ell}=x$.

To properly state the lemma, we recall a couple of definitions. A semigroup ${\m S}$ is \emph{left-zero} if it satisfies $xy=x$ and \emph{right-zero} if it satisfies $xy=y$. Further, a join semilattice $A$ with a greatest element is a \emph{prelattice} if $A$ can be made into a lattice by adjoining a new bottom element, i.e., $A^\bot :=A\cup\{\bot\}$ is a lattice if the order $\leq$ on $A$ is extended by setting $\bot\leq a$ for all $a\in A$. An element is called \emph{central} if it commutes with all elements.

\begin{lemma}\label{l:A^i_has_conical_elements} \label{l:centralizer} \label{l:uniquenoncommuting} \label{l:covering} \label{c:non-commuting} \label{c:starand*}
Let ${\m A}$ be a conic idempotent residuated lattice and let $x,y\in A$.
\begin{enumerate}
\item Each of $x^\ell$ and $x^r$ are conical elements and both have the same sign, which is opposite to the sign of $x$. In particular, $A^i$ is a chain.
\item $x$ is central iff $x^\ell = x^r$.
 \item If $x$ is central, then $\{x, x^*\}$ forms a semilattice with multiplication equal to the inherited meet.
 \item If $x$ is not central, then $xy=yx$ for all $y \not = x^*$ and $\{x, x^*\}$ forms a left-zero or a right-zero semigroup.
\item If $x$ and $y$ do not commute, then $x^*=y$, $y^*=x$ and $\{x,y\}$ forms a left-zero or a right-zero semigroup.
\item There is no element between the conical elements $x^\ell$ and $x^r$ and these elements form a covering pair.
\item $x \leq y^\star$ iff $y \leq x^\star$, so $((-)^\star, (-)^\star)$ is a Galois connection. Moreover, $(x^\star, x^*)$ forms a splitting pair: For all $c$, $c \leq x^\star$ or $x^* \leq c$.
\item  $\gamma$ is a nucleus on ${\m A}$.
\item  $A^i$ is the universe of a totally ordered quasi-involutive subalgebra $\m A^i$ of $\m A$, and ${\m A}_\gamma={\m A}^i$.
\item  For each $a \in A^i$, $\gamma^{-1}(a)$ is a prelattice. Furthermore:
\begin{enumerate}
\item If $a$ has no lower cover in $A^i$, then  $\gamma^{-1}(a)$ is a lattice.
\item For negative $a \in A^i$, the block $\gamma^{-1}(a)$ is a Brouwerian lattice and the implication in this Brouwerian lattice is given by $x\Rightarrow_a y = (x \under y)\meet a = (y\ovr x)\meet a$.
\item If $a \in A^i$ is not central then  $\gamma^{-1}(a)$ is trivial.
  \end{enumerate}
\end{enumerate}
\end{lemma}

The prelattices $\gamma^{-1}(a)$ for $a\in A^i$ are called \emph{blocks}. The fundamental operations in any conic idempotent residuated lattice ${\m A}$ are determined by the subalgebra ${\m A}^i$ together with the order-theoretic information coming from its blocks. The following lemma depicts this in numerous different ways, all of which will be used in proofs later; this summarizes portions of Lemma 3.3, Theorem 3.11, and Corollary 3.12 from \cite{FG1}.

\begin{lemma}\label{l:idcon_mult} \label{c:simple mult} 
Let ${\m A}$ be a conic idempotent residuated lattice and let $x,y\in A$.
\begin{enumerate}
	\item Then
\[ x y = \begin{cases} 
      x\meet y & x,y\leq 1 \\
      x\join y & x,y\geq 1 \\
      x & y\leq 1\leq y^\ell<x\text{ or }x \leq 1\leq y\leq x^r \\
      &   \Lra   \;\; x^r< y\leq 1\leq x\text{ or }x \leq 1\leq y\leq x^r \\
      &   \Lra   \;\; y\in  (x^r, 1] \cup [1, x^r]    \Lra   \;\; y^\ell \in  [x, 1] \cup [1, x) \\		
      y & x\leq 1\leq x^r<y\text{ or }y \leq 1\leq x\leq y^\ell \\
      & \Lra   \;\; y^ \ell < x\leq 1\leq y\text{ or }y \leq 1\leq x\leq y^\ell  \\
      & \Lra   \;\; x\in  (y^\ell, 1] \cup [1, y^\ell] 
        \Lra   \;\; x^r\in  [y, 1] \cup [1, y) 
   \end{cases}
\]
\item Also,
\[ x y = \begin{cases} 
      x\meet y & x,y\leq 1 \text{ and }  x,y \text{ are incomparable} \\
      x\join y & x,y\geq 1  \text{ and }  x,y\text{ are incomparable} \\
      x &  y\in  (x^r, x] \cup  [x, x^r] \\
      y &  x\in  (y^\ell, y] \cup [y, y^\ell]
   \end{cases}
\]
\item Moreover,
\[ x y = \begin{cases} 
      x\meet y & xy\leq 1  \;\; (\Lra x \leq y^\ell \Lra y \leq x^r)\\
      x\join y & 1< xy  \;\; (\Lra y^\ell < x \Lra x^r <y)
   \end{cases}
\]
\item Suppose $s=\gamma(x)$ and $t=\gamma(y)$. Then
\[ xy = \begin{cases} 
      x\meet y & s,t\leq 1 \\
      x\join y & s, t>1\\
      x & st=s\text{ and }s,t\text{ have opposite sign} \\
      y & st=t\text{ and }s,t\text{ have opposite sign}  \\
   \end{cases}
	= \begin{cases} 
      x\meet y & s=t\leq 1 \\
      x\join y & s=t>1\\
      x & st=s\text{ and }s\neq t \\
      y & st=t\text{ and }s\neq t \\
   \end{cases}
\]
\item We also have 
$$x\ld y = \begin{cases} 
      x^r \jn y=\gamma(x)^r \jn y & x\leq y \\
      x^r \mt y=\gamma(x)^r \mt y & \gamma(y) < \gamma(x),\text{ or }1<\gamma(x)=\gamma(y)\text{ and }x \not \leq y\\
       x \Rightarrow_a y & \gamma(x)=\gamma(y)=a \leq 1\text{ and }x \not \leq y.
   \end{cases}
$$
$$y\rd x = \begin{cases} 
      x^\ell \jn y=\gamma(x)^\ell \jn y & x\leq y \\
      x^\ell \mt y=\gamma(x)^\ell \mt y & \gamma(y) < \gamma(x),\text{ or }1<\gamma(x)=\gamma(y)\text{ and }x \not \leq y\\
       x \Rightarrow_a y & \gamma(x)=\gamma(y) = a \leq 1\text{ and }x \not \leq y.
   \end{cases}
$$
\end{enumerate}
\end{lemma}

\subsection{Decomposing of conic idempotent residuated lattices}

The information presented in Section~\ref{subsec:skeleton} underwrites a decomposition theorem for conic idempotent residuated lattices. A \emph{decomposition system} is a structure $(\m S, \{\m A_s: s \in S\})$, where $\m S$ is an idempotent residuated chain and, for every $s \in S$, $\m A_s$ is a prelattice with top element $s$ such that:
\begin{enumerate}
\item If $s$ has no lower cover in $\m S$, then $A_s$ is a lattice.
\item For negative $s \in \m S$, the block $A_s$ is a Brouwerian lattice. 
\item If $s$ is not central, then $A_s$ is trivial.
  \end{enumerate}
Whenever $\m A$ is a conic idempotent residuated lattice, $(\m A^i, \{\gamma^{-1}(s): s \in A^i\})$ is a decomposition system and in this situation, $\m A^i$ is quasi-involutive. 

Given a decomposition system $D=(\m S, \{\m A_s: s \in S\})$, we consider the ordinal sum $A_D:=\bigoplus_{s \in S} A_s$. The order $\leq$ on $A_D$ is the order given by the usual ordinal sum of posets. Given $s\in S$, $s^r$ and $s^\ell$ give the usual elements of $S$. For any $x\in A_s$ and $y\in A_t$, we further define

$$ x\cdot y = \begin{cases} 
      x\meet y & s=t\leq 1 \\
      x\join y & s=t>1\\
      x & st=s\text{ and }s\neq t \\
      y & st=t\text{ and }s\neq t \\
   \end{cases}
	$$

$$x\ld y = \begin{cases} 
      s^r \jn y \\
      s^r \mt y \\
      x \Rightarrow_s y \\
   \end{cases}
	\quad
y\rd x = \begin{cases} 
      s^\ell \jn y  & \qquad \text{ if }x\leq y \\
      s^\ell \mt y  & \qquad \text{ if }t < s,\text{ or }1<s=t\text{ and }x \not \leq y\\
       x \Rightarrow_s y  & \qquad \text{ if }s=t \leq 1\text{ and }x \not \leq y.
   \end{cases}
$$
  Let ${\m A}_D$ denote the algebra defined with the operations above, the lattice operations of the ordinal sum, and unit $1$ of $\m S$.
 
  \begin{theorem}[{\cite[Theorem 3.13]{FG1}}]\label{t:decomposition}
  Given a decomposition system $D$, the algebra $\m A_D$ is a conic idempotent residuated lattice.
  \end{theorem}
  
A decomposition system $(\m S, \{\m A_s: s \in S\})$ is called \emph{quasi-involutive} if $\m S$ is quasi-involutive. 
	
	 \begin{remark} For a commutative conic idempotent residuated, $x^*=x^\ell=x^r$. Thus, in the presence of commutativity, quasi-involutivity becomes equivalent to involutivity $x^{**}=x$ and skeletons are odd Sugihara monoids. 
	\end{remark}
	
For subalgebras, it is useful to think of decomposition systems as (multisorted) partial algebras. For $(\m S, \{\m A_s: s \in S\})$ a decomposition system:
	\begin{enumerate}
		\item If $s \in S$ is negative, $\m A_s$ is a total algebra in the language of Brouwerian algebras.
		\item If $s \in S$, $\m A_s$ is a topped prelattice, i.e. 
		a total algebra with respect to 
		join and with respect to a constant operation that produces the top element $s$ as an element of $\m S$, and a partial operation with respect to meet. 
		\item 	$\m S$ is a total algebra in the language of residuated lattices, and a partial algebra with respect to an operation $s \mapsto s^\downarrow$, where if $\m A_s$ is a proper prelattice (i.e., a prelattice that is not a lattice), then $s^\downarrow$ is a lower cover of $s$ in $\m S$.
	\end{enumerate}
We say $(\m S_A, \{\m A_s: s \in S_A\})$ is a \emph{subsystem} of $(\m S_B, \{\m B_s: s \in S_B\})$ if it is a sub-partial algebra, understood as above, i.e., the subsystem is closed under all operations whenever they are defined in the bigger system. Spelled out,  $(\m S_A, \{\m A_s: s \in S_A\})$ is  a subsystem of  $(\m S_B, \{\m B_s: s \in S_B\})$ if:
	\begin{enumerate}
		\item $\m S_A$ is  subalgebra of $\m S_B$.
		\item For every $s \in S^-_A$, $\m A_s$ is a Brouwerian subalgebra of $\m B_s$.
		\item For every $s \in S^+_A$, $\m A_s$ is a topped subprelattice of $\m B_s$ (in particular if $\m B_s$ is a lattice, then $\m A_s$ is a lattice).
		\item If $\m A_s$ is a proper prelattice (hence also $\m B_s$ is a proper prelattice) then $S_A$ contains the lower cover $s^\downarrow$ of $s$ in $S_B$. 
	\end{enumerate}
	
    \begin{lemma}[{\cite[Lemma 3.16]{FG1}}]\label{l:subsystems}
Let $(\m S_A, \{\m A_s: s \in S_A\})$ be the decomposition system corresponding to the algebra ${\m A}$ and $(\m S_B, \{\m B_s: s \in S_B\})$ be the decomposition system corresponding to the algebra ${\m B}$. Then $(\m S_A, \{\m A_s: s \in S_A\})$ is a subsystem of $(\m S_B, \{\m B_s: s \in S_B\})$ iff ${\m A}$ is a subalgebra of ${\m B}$. 
  \end{lemma}

	\subsection{Congruence generation and the local deduction theorem}
	
	A subset $F$ of a residuated lattice ${\m A}$ is a \emph{congruence filter} of ${\m A}$ if $F$ is a lattice filter, $F$ is closed under multiplication, and $F$ is closed under conjugation, i.e., whenever $y \in F$ and $x \in A$ then $x \ld yx, \; xy\rd x \in F$. Congruences of ${\m A}$ are in bijective correspondence with congruence filters of ${\m A}$ via the mutually inverse maps $\theta\mapsto F_\theta=\upset [1]_\theta$ and $F\mapsto\theta_F$, where $a \mathrel{\theta_F} b$ iff $a \ld b, \; b \ld a \in F$; see \cite[Section 3.6]{GJKO2007}.
	
We write $y^{u^n}:= y^{uu \cdots u}$ when $u\in\{r,\ell\}$ appears $n$-many times. Also define:

 $$s(y)=s_1(y)=y \mt y^{\ell \ell} \mt  y^{rr},$$
 $$s_2(y)= y \mt    y^{\ell \ell \ell \ell} \mt  y^{\ell \ell r r } \mt  y^{r r \ell \ell} \mt y^{r r rr},$$ 
 and for all $n\in\mathbb{N}$, 

      $$s_n(y):= y \mt \bigwedge  \{y^{c_1c_1c_2c_2\cdots c_n c_n}:  c_1, c_2, \ldots, c_n \in \{\ell, r\}\}.$$ Moreover, write $s^n$ for the $n$-fold composition of $s$ with itself, define $t_n(y):=1 \mt s_n(y)$ and $t(y):=1\mt s(y)$, and use $t^n$ for the $n$-fold composition of $t$ with itself. Finally, if ${\m A}$ is a residuated lattice and $Y\subseteq A$, write $Y\meet 1 := \{y\meet 1 : y\in Y\}$.

    \begin{lemma}[{\cite[Lemma 4.2]{FG1}}]\label{l:congruencefilters2} 
    Let $\m A$ be a semiconic idempotent residuated lattice and $Y \subseteq A$.
    \begin{enumerate}
    \item $\m A$ satisfies the identities $s^n(y)=s_n(y)$ and $t^n(y)=t_n(y)$. 
    \item The congruence filter  of $A$ generated by $Y$ is given by\\ $\langle Y \rangle={\uparrow} \{t_{n_1}(y_1) \mt \cdots \mt t_{n_k}(y_k): k, n_1, \ldots, n_k  \in \mathbb{N}, y_1, \ldots, y_k \in Y \mt 1\}$. 
		\item  $\langle Y \rangle={\uparrow} \{t_{n}(y): n  \in \mathbb{N}, y \in Y^F\}$, where $Y^F$ is the $\mt$-subalgebra generated by $Y \cup \{1\}$. 
\item   If $Y$ is a closed under meets, then $\langle Y \rangle={\uparrow} \{t_{n}(y):  n  \in \mathbb{N}, y \in Y \mt 1\}$.
\item If $Y$ is a congruence filter and $a \in A$, then\\
 $\langle Y \cup \{a\} \rangle ={\uparrow} \{ y \mt s_{n}(a):  n  \in \mathbb{N}, y \in Y\}$.
\end{enumerate}
  \end{lemma}
The previous lemma is used in the proof of the following, which summarizes Theorems 4.3, Lemmas 4.6 and 4.8, and Corollaries 4.4 and 4.9 of \cite{FG1}.

\begin{lemma} \label{t:CEP} \label{c:CEP} \label{l:adding_conjugates} \label{l:FSI} \label{c:FSI}
\begin{enumerate}
\item[]
\item The variety of semiconic idempotent residuated lattices has the congruence extension property.
\item The variety of semilinear idempotent residuated lattices has the congruence extension property. 
\item Semiconic idempotent residuated lattices satisfy the implication $x\join y = 1 \Rightarrow s_n(x)\join s_m(x) = 1$ for all $n,m\in\mathbb{N}$.
\item A semiconic idempotent residuated lattice is finitely subdirectly irreducible iff $1$ is join irreducible. Also, each one of these assumptions imply that the algebra is conic.
\item A semilinear idempotent residuated lattice is finitely subdirectly irreducible iff $1$ is join irreducible. Also, each one of these assumptions imply that the algebra is linear
\end{enumerate}
\end{lemma}

\section{The structure of (quasi-involutive) idempotent residuated chains}\label{sec:chains}

We have seen that each conic idempotent residuated lattice can be decomposed as an ordinal sum based on a quasi-involutive idempotent residuated chain skeleton, where the blocks are arbitrary prelattices, lattices, or Brouwerian algebras. Since the blocks are fairly well understood, we now describe the structure of quasi-involutive idempotent residuated chains. More generally, we study the structure of all idempotent residuated chains by using three tools: An equivalent formulation of idempotent residuated chains in terms of their $\{\mt, \jn, 1, {}^r, {}^\ell\}$-reducts (see Section~\ref{s:idempotent Galois connections}), a pictorial presentation of these called flow diagrams (see Section~\ref{s:flow diagrams}), and another equivalent formulation in terms of certain enhancements of their monoidal preorders (see Sections~\ref{s:ircs to empos} and \ref{s:empos to ircs}).

With an eye on amalgamation, in Section~\ref{s:weaklyinvolutive} we use the tools developed so far to identify a crucial property, rigidity, that is necessary for amalgamating quasi-involutive idempotent residuated chains. Rigidity states that the quasi-involutive skeleton is $^\star$-involutive. Then, in Section~\ref{s:one-generated}, we study one-generated rigid quasi-involutive idempotent residuated chains, and in Section~\ref{nestedsums} we show that all rigid quasi-involutive idempotent residuated chains are nested sums of one-generated ones.

As previously mentioned, the decomposition of Section~\ref{sec:basicstructure} applies to idempotent residuated chains since they are examples of conic idempotent residuated lattices. Combined with the decomposition results given before, the results of this section provide a fine-grained analysis of the structure of idempotent residuated chains. Since residuated chains are the subject of a considerable literature, we expect the results of this section to therefore be of broad interest outside of its applications to the conic case.

Lemma~\ref{l:idcon_mult} specializes to chains as follows:

\begin{corollary} \label{c:idempotentchainsoperations}
In every idempotent residuated chain we have
\[ x y = \begin{cases} 
      x &  y\in  (x^r, x] \text{ or } y\in  [x, x^r] \\
      y &  x\in  (y^\ell, y] \text{ or } x\in [y, y^\ell]
   \end{cases}
	 = \begin{cases} 
      x\meet y & xy\leq 1  \;\; (\Lra x \leq y^\ell \Lra y \leq x^r)\\
      x\join y & 1< xy  \;\; (\Lra y^\ell < x \Lra x^r <y)
   \end{cases}
\]
	\[
x\ld y = \begin{cases} 
      x^r \jn y & x\leq y \\
      x^r \mt y & y < x
   \end{cases} \qquad \qquad
y\rd x = \begin{cases} 
      x^\ell \jn y & x\leq y \\
      x^\ell \mt y & y < x
   \end{cases}
\]
\end{corollary}

 \subsection{The reducts of idempotent residuated chains}\label{s:idempotent Galois connections} In view of the above corollary, subalgebra generation is performed by closure under the unary operations $ {}^r$ and  ${}^\ell$ and closure under $1$. Also, every idempotent residuated chain is definitionally-equivalent to a structure over the language $\{\mt, \jn, {}^r, {}^\ell, 1\}$. We will characterize these $\{\mt, \jn, {}^r, {}^\ell, 1\}$-reducts of idempotent residuated chains. This characterization will be useful in connecting idempotent residuated chains to enhanced monoidal preorders (see Sections~\ref{s:ircs to empos} and \ref{s:empos to ircs}).

By Lemma~\ref{l:centralizer}(6,7), if $\m A$ is an idempotent residuated chain, then the reduct $(A, \mt, \jn, {}^\ell, {}^r, 1)$ is a chain such that $({}^\ell, {}^r)$ forms a Galois connection,  $1^\ell=1^r=1$, and for all $x$ there is no element between $x^\ell$ and $x^r$. We call such algebras \emph{idempotent Galois connections}.
Note that $(A, \mt, \jn, {}^\ell, {}^r, 1)$ is an idempotent Galois connection iff it satisfies the following definition in terms of positive universal sentences:
\begin{enumerate}
\item $(A, \mt, \jn)$ is a lattice,
\item $1^\ell=1^r=1$,
\item  $x \leq x^{\ell r}$, $x \leq x^{r \ell}$,  $(x \jn y)^\ell = x^\ell \mt x^\ell$, $(x \jn y)^r = x^r \mt x^r$, $x^{\ell r \ell}=x^\ell$ and $x^{r \ell r}=x^r$ 
(Galois connection on a lattice; see Lemma~7.26 of \cite{DavPrie2002}),
\item $x \leq y$ or $y \leq x$ (totally ordered),
\item $y \leq x^r \mt x^\ell$ or $x^r \jn x^\ell \leq y$ (there is no element between $x^\ell$ and $x^r$).
\end{enumerate}

Note that condition 5 can be written as ($y \leq x^\star$ or $x^* \leq y$), i.e. $(x^\star, x^*)$ forms a splitting pair.

\medskip

Conversely, given $\m C=(C, \mt, \jn, {}^\ell, {}^r, 1)$, the algebra $\m R(\m C)=(C, \mt, \jn, \cdot, {\ld}, {\rd}, 1)$ is defined by

$$ x y 
	 = \begin{cases} 
      x\meet y &  x \leq y^\ell \; (\Lra y \leq x^r)\\
      x\join y & y^\ell < x \; (\Lra x^r <y)
   \end{cases} \;\;\;
x\ld y = \begin{cases} 
      x^r \jn y & x\leq y \\
      x^r \mt y & y < x
   \end{cases} \;\;\; 
y\rd x = \begin{cases} 
      x^\ell \jn y & x\leq y \\
      x^\ell \mt y & y < x
   \end{cases}
$$

\begin{lemma} \label{l:idGC}
If $\m C=(C, \mt, \jn, {}^\ell, ^r, 1)$ is an idempotent Galois connection, then $\m R(\m C)$ is an idempotent residuated chain.
\end{lemma}

\begin{proof} 
 First we show that multiplication is order-preserving; assume that $x \leq z$. If $z \leq y^\ell$, then also $x \leq y^\ell$. Therefore, $xy=x\mt y \leq z \mt y=zy$. If $y^\ell <z$, then $xy \leq x \jn y \leq z \jn y=zy$. Since $C$ is a chain, this implies that multiplication distributes over join and meet. 

Note that for $x \leq 1=1^\ell=1^r$, we have $x1=x \mt 1=x$ and $1x=1 \mt x=x$;
 for $1^r=1^\ell=1<x$, we have $x1=x \jn 1=x$ and $1x=1 \jn x=x$. Therefore, $1$ is the multiplicative unit. 
Clearly, multiplication is conservative, hence idempotent. To show that it is associative, we consider $x,y,z$ and show that $(xy)z=x(yz)$ by distinguishing cases. 

If $x\leq y^\ell$ and $y^r<z$, then 
since there are no elements between $y^\ell$ and $y^r$, we get $x \leq z$. So, 
$(xy)z=(x \mt y)z=xz \mt yz =xz \mt (y \jn z)=xz$, because $xz \leq zz=z \leq y \jn z$, and
$x(yz)=x(y\jn z)=xy \jn xz=(x \mt y) \jn xz=xz$, because $x \mt y \leq x=xx \leq xz$.

If $y^\ell<x$ and $z \leq y^r$, then 
since there are no elements between $y^\ell$ and $y^r$, we get $z \leq x$. So, 
$(xy)z=(x \jn y)z=xz \jn yz =xz \jn (y \mt z)=xz$, because $y \mt z \leq z=zz \leq xz$, and
$x(yz)=x(y\mt z)=xy \mt xz=(x \jn y) \mt xz=xz$, because $xz \leq xx=x \leq x \jn y$.

If $x\leq y^\ell$ and $z \leq y^r$, then $y \leq x^r$ and $y \leq z^\ell$, so $(xy)z=(x \mt y)z=xz \mt yz =xz \mt y \mt z$ and
$x(yz)=x(y\mt z)=xy \mt xz=x \mt y \mt xz$. If $x \leq z^\ell$, then $xz=x \mt z$ and the two sides are equal to $x \mt y \mt z$. If $z^\ell <x$, then $x^r <z$ and $xz=x \jn z$, so one side is equal to $y \mt z$ and the other to $y \mt x$. Moreover, we obtain $y \leq z^\ell <x$ and $y \leq x^r<z$, so both sides equal to $y$.

If $y^\ell<x$ and $y^r<z$, then $x^r<y$ and $z^\ell < y$, so $(xy)z=(x \jn y)z=xz \jn yz =xz \jn y \jn z$ and
$x(yz)=x(y\jn z)=xy \jn xz=x \jn y \jn xz$. If $z^\ell <x$, then $xz=x \jn z$ and the two sides are equal to $x \jn y \jn z$. If $x \leq z^\ell$, then $z \leq x^r$ and $xz=x \mt z$, so one side is equal to $y \jn z$ and the other to $y \jn x$. Moreover, we obtain $x \leq z^\ell<y$ and $z \leq x^r<y$, so both sides equal to $y$.

Finally, we will show $xy \leq z$ iff $y \leq x \ld z$, for all $x,y,z \in A$, by considering cases. The equivalence to $x \leq z \rd y$ is analogous.

Assume $z<x$. Then $xy \leq z$ implies $xy\not = x$, so $y=xy\leq z<x$. 
Note that if we had $x^r<y$ then $xy=x\jn y=x$, a contradiction; so $y \leq x^r$.

Therefore, $xy \leq z$ iff $x^r \geq xy=y\leq z$.
On the other hand,  $y \leq x \ld z$ iff $y \leq x^r \mt z$ iff $y \leq x^r$ and $y \leq z$. 
So, $xy=x \mt y=y$, since $y \leq z<x$.

Therefore,  $y \leq x \ld z$ iff $x^r \geq xy=y\leq z$. 
 Consequently,  $xy \leq z$ iff $y \leq x \ld z$.

Assume $x \leq z$ and $xy=x$. Then  $xy \leq z$ is true. By way of contradiction, if  $y \leq x \ld z$ were false we would have $x^r \jn z <y$ so $x^r <y$ and $z<y$; hence $x \leq z < y$  and so $x=xy=x \jn y=y>x$, a contradiction.

Assume $x \leq z$, $xy=y$ and $y \leq z$. Then $xy \leq z$ is true. Also $y \leq z \leq x^r \jn z$, so also $y \leq x \ld z$ is true.

Assume $x \leq z$, $xy=y$ and $z<y$. So $z < xy$; hence $xy \leq z$ is false. Also note that we cannot have $y \leq x^r$, as that would imply the contradiction $y=xy=x \mt y=x<y$; hence $x^r <y$. Since $z<y$ we have $x^r \jn z <y$, so $y \leq x \ld z$ is also false. 
\end{proof}

The condition that  for all $x$ there is no element between $x^\ell$ and $x^r$ does not follow from the other conditions, as a 5-element example shows, and it is precisely what guarantees associativity of multiplication. 

\begin{corollary}
Idempotent residuated chains are definitionally equivalent to idempotent Galois connections. 
\end{corollary}

 \subsection{Flow diagrams}\label{s:flow diagrams}
We now provide a visual way to move between the action of the inverse operations ${}^\ell$, ${}^r$, and the ordering of the chain. This will eventually lead to the connection with the enhanced monoidal preorders.
 
It follows from Lemma~\ref{c:non-commuting}(5) that if  $a,b$ are non-commuting elements of an idempotent residuated chain, then  $a$ and $b$ have different sign, $a^*=b$,  $b^*=a$ and  $\{a, b\}$ forms either a left-zero or a right-zero semigroup. If $\{a, b\}$ forms a right-zero semigroup and $a$ is positive (and $b$ negative), we write $a \mathrel{R} b$. 
If $\{a, b\}$ forms a left-zero semigroup  and $a$ is positive (and $b$ negative), we write $a \mathrel{L} b$. We write $a \mathrel{C} b$ if  $a$ is positive (and $b$ negative) and $a^\ell=a^r=b$ (i.e., if $ab=ba$). Note that these relations are not symmetric because they encode the signs of the elements. We follow the naming convention of \cite{Ga2004} of \textbf{a}bove and \textbf{b}elow the identity element $1$, for $a$ and $b$. Also, recall that $x \prec y$ means that $x$ is a lower cover of $y$.

\begin{lemma}\label{l:RandL}
Let $a,b$ be elements of an idempotent residuated chain $\m A$ such that $a$ is positive and $b$ is negative.
\begin{enumerate}
\begin{multicols}{2}
\item The following are equivalent:
\begin{enumerate}
\item $a \mathrel{R} b$.
\item$a^\ell \prec a^r=b$.
\item  $b^r \prec b^\ell=a$.
\item  $a^{rr} \prec a$ and $b=a^r$.
\item $b^{\ell \ell} \prec b$ and $a=b^\ell$.
\item  $a=a^{r \ell}$ is not central and $b=a^r$.
\item  $b = b^{\ell r}$ is not central and $a=b^\ell$.
\end{enumerate}
\columnbreak
\item The following are equivalent:
\begin{enumerate}
\item $a \mathrel{L} b$.
\item$a^r \prec a^\ell=b$.
\item  $b^\ell \prec b^r=a$.
\item  $a^{\ell \ell} \prec a$ and $b=a^\ell$.
\item $b^{rr} \prec b$ and $a=b^r$.
\item  $a=a^{\ell r}$ is not central and $b=a^\ell$.
\item  $b = b^{r \ell}$ is not central  and $a=b^r$.
\end{enumerate}
\end{multicols}
\end{enumerate}
\end{lemma}  
 
 \begin{proof}
1(a-e). Given that $a$ is positive and $b$ is negative, $a \mathrel{R} b$ is equivalent to the conjunction of $ab=b$ and $ba=a$.  By Lemma~\ref{l:idcon_mult}(1), $ab=b$ is equivalent to $a \leq b^\ell$, which is also equivalent to  $ab \leq 1$ and $b \leq a^r$. 
Also, by Lemma~\ref{l:idcon_mult}(1), $ba=a$ is equivalent to $b^r < a$, which is also equivalent to  $a \not \leq b^r$, $ba \not \leq 1$, $b \not \leq a^\ell$, and $a^\ell < b$. 

Therefore,   $a \mathrel{R} b$ is equivalent to $b^r < a \leq b^\ell$, and since $b^r$ and $b^\ell$ form a covering pair (by Lemma~\ref{l:centralizer}(6)), this is further equivalent to  $b^r \prec a= b^\ell$.  

 Also,  $a \mathrel{R} b$ is equivalent to $a^\ell < b \leq a^r$, and since  $a^\ell$ and $a^r$ form a covering pair (by Lemma~\ref{l:centralizer}(6)), this is further equivalent to  $a^\ell \prec b=a^r$. 
 
  All these establish that (a), (b) and (c) are equivalent. Clearly, (d) follows from (b) and (c). Conversely, if (d) holds, then in particular $b^r < a$ and $b \leq a^r$, which is equivalent to (a), by the calculations above. The equivalence of (e) is established in  similar way.
  2(a-e) are established in a similar way. 
	
	Now to show that 1(f) is equivalent to 1(a), we note that (f) follows from (b) and (c).  Conversely, if (f) holds, then $a$ is not central, so one of 1(a) and 2(a) hold. It cannot be the case that 2(a) holds because that would imply 2(b), which contradicts 1(f); so 1(a) holds. Likewise we get the equivalence of 1(g) to the rest of 1, as well as the equivalence of 2(f) and 2(g) to the rest of 2.
 \end{proof}

 The following result is an immediate consequence of Lemma~\ref{l:RandL}.
 
 \begin{corollary}\label{c:RandL}
 Let  $\m A$ be an  idempotent residuated chain. 
  \begin{enumerate}
\item If $a$ is a positive non-central element of $\m A$, then exactly one of the following situations happen.
 \begin{enumerate}
\item $a^{rr}\prec a^{r \ell}=a \mathrel{R} a^r=a^* \succ a^\star = a^\ell$.
\item $a^{\ell \ell}\prec a^{\ell r}=a \mathrel{L} a^\ell=a^* \succ a^\star =a^r$.
\end{enumerate}
\item If $b$ is a negative non-central element of $\m A$, then exactly one of the following situations happen.
 \begin{enumerate}
\item $b^{r}=b^\star \prec b^*= b^{\ell} \mathrel{R} b=b^{\ell r} \succ b^{\ell \ell}$.
\item $b^{\ell}=b^\star\prec b^*=b^r \mathrel{L} b=b^{r \ell} \succ b^{rr}$.
\end{enumerate}
\item If $x$ is a central element of $\m A$, then $x^\star=x^*=x^\ell=x^r$.
\end{enumerate}
 \end{corollary}

 Figure~\ref{flowdiagrams} demonstrates the two situations and provides an explanation for the notation used.

\begin{figure}[htbp]
\begin{center}
\begin{tikzpicture}[scale=1]
\node[left]  at (-.5,5) {\footnotesize  $b^\star=b^\ell$};
\node[left] at (-.5,4) {\footnotesize  $b^*=b^r=a$};
\draw[mark=square*]  plot coordinates {(-.5,4)};
\node at (0,4) {\footnotesize  $L$};
\node[right] at (.5,4) {\footnotesize  $b=a^\ell=a^*$};
\draw[mark=square*]  plot coordinates {(.5,4)};
\node[right] at (.5,3) {\footnotesize  $a^r=a^\star$};
\draw[very thick] (-.5,5) -- (-.5,4);
\draw[very thick] (.5,4) -- (.5,3);
\draw  [->>,dashed]  (.5,4) to [out=150,in=-60] (-.5,5);
\draw  [->>]  (.5,4) to [out=-150,in=-30] (-.5,4);
\draw  [->>, dashed]  (-.5,4) to [out=30,in=150] (.5,4);
\draw [->>]  (-.5,4) to [out=-30,in=120] (.5,3);
\end{tikzpicture}
\quad
\begin{tikzpicture}[scale=1]
\node[left]  at (-.5,5) {\footnotesize  $b^\star=b^r$};
\node[left] (a) at (-.5,4) {\footnotesize  $b^*=b^\ell=a$};
\draw[mark=square*]  plot coordinates {(-.5,4)};
\node at (0,4) {\footnotesize  $R$};
\node[right] (b) at (.5,4) {\footnotesize  $b=a^r=a^*$};
\draw[mark=square*]  plot coordinates {(.5,4)};
\node[right] at (.5,3) {\footnotesize  $a^\ell=a^\star$};
\draw[very thick] (-.5,5) -- (-.5,4);
\draw[very thick] (.5,4) -- (.5,3);
\draw  [->>]  (.5,4) to [out=120,in=-30] (-.5,5);
\draw  [->>, dashed]  (.5,4) to [out=150,in=30] (-.5,4);
\draw  [->>]  (-.5,4) to [out=-30,in=-150] (.5,4);
\draw [->>, dashed]  (-.5,4) to [out=-60,in=150] (.5,3);
\end{tikzpicture}
\end{center}
\vskip 5pt
\begin{center}
\begin{tikzpicture}[scale=1]
\draw[very thick] (-3.5,0)--(-3.5,1);
\node[left] at (-3.5,1) {\footnotesize  $a_1$};
\node[left] at (-3.5,0) {\footnotesize  $a_2$};
\node at (-4.2, 0.5) {\footnotesize  $a_1 \prec a_2$};
\draw[very thick] (-2.5,0)--(-2.5,1);
\node[right] at (-2.5,1) {\footnotesize  $b_2$};
\node[right] at (-2.5,0) {\footnotesize  $b_1$};
\node at (-1.8, 0.5) {\footnotesize  $b_1 \prec b_2$};
\draw[->, dashed] (0,1)--(1,1);
\node[left] at (0,1) {\footnotesize  $x$};
\node[right] at (1,1) {\footnotesize  $x^\ell$};
\draw[->] (0,0.5)--(1,0.5);
\node[left] at (0,0.5) {\footnotesize  $x$};
\node[right] at (1,0.5) {\footnotesize  $x^r$};
\node[right] at (3.1,1) {\footnotesize  non-central};
\draw[mark=square*]  plot coordinates {(3,1)};
\node[right] at (3, 0.5) {\footnotesize  $a$'s: positive};
\node[right] at (3, 0) {\footnotesize  $b$'s: negative};
\end{tikzpicture}
\end{center}
\caption{The flow diagrams for non-central elements\label{flowdiagrams}}
\end{figure}

\subsection{From idempotent residuated chains to enhanced monoidal preorders}\label{s:ircs to empos}
In this section and the next, we are finally ready to introduce the main tool for amalgamating idempotent residuated chains, enhanced monoidal preorders. These are definitionally equivalent to both idempotent Galois connections and to idempotent residuated chains, and will be key in the arguments in the sequel.

 Given an idempotent residuated chain $\m A$, the \emph{natural order}, also considered in \cite{CZ2009}, is defined by: $x \leq_n y$ iff $xy=yx=x$.

   The \emph{monoidal preorder}, also considered in \cite{GJM2020}, is defined by: $x \sqsubseteq y$ iff $xy=x$. We use the convention that Hasse diagrams for preordered sets are similar to ones for ordered sets with the only difference being that mutually-comparable elements are placed on the same level and are connected by horizontal line segments. 
   For a preorder $\sqsubseteq$ we write $x \sqsubset y$ if $x \sqsubseteq y$ and $y \not \sqsubseteq x$; this is a stronger demand than simply asking that  $x \sqsubseteq y$ and $y \not = x$. Observe that distinct $x,y$ are mutually-comparable in $\sqsubseteq$ iff $x \mathrel{L} y$ or $y \mathrel{L} x$. Likewise, distinct $x,y$ are mutually-incomparable iff $x \mathrel{R} y$ or $y \mathrel{R} x$.
   
  \begin{lemma}\label{l:naturalmonoidal1} 
The following hold in idempotent residuated chains.
	\begin{enumerate}
	\item The relation $\leq_n$ is an order, and the relation $\sqsubseteq$ is a preorder. 
	\item $x <_n y$ iff $x \sqsubset y$.
	\item $xy=y$ iff $y \not \sqsubseteq x$.
	\item $x \sqsubset y$ iff $y \not = x=xy=yx$.
   \end{enumerate}
    \end{lemma}
		
 \begin{proof}
For 1, observe that if $xy=x$ and $yz=y$, then $xz=xyz=xy=x$. Thus $\sqsubseteq$ is transitive.
Also, if $yx=x$ and $zy=y$, then $zx=zyx=yx=x$, so also $\leq_n$ is transitive.
Since $xx=x$,  $\sqsubseteq$  is also reflexive, hence it is a preorder. Likewise, $\leq_n$ is reflexive. Finally, if $x \leq_n y$ and $y \leq_n x$, then $x=xy=y$, so $\leq_n$ is an order.

 Note that $x \sqsubset y$ iff ($x \sqsubseteq y$ and $y \not \sqsubseteq x$) iff ($xy=x$ and $yx \not=y$) iff ($xy=yx=x\not = y$) (due to conservativity) iff $x <_n y$. This gives 2; 3 follows by definition. 4 follows from 2.
 \end{proof}
            
 Lemma~\ref{l:naturalmonoidal1} shows that the monoidal preorder completely encodes the multiplication operation and contains more information than the natural order. Still, there are different idempotent residuated structures on the same set that have the same monoidal preorder, as can be seen from Figure~\ref{f:samemon}.

\begin{figure}[htbp]
\begin{center}
\begin{tikzpicture}[scale=1]
\node[left] at (3,6) {\footnotesize  $a_2$};
\draw[mark=square*]  plot coordinates {(3,6)};
\node[left] at (3,5) {\footnotesize  $\mathbf{c_3}$};
\draw[mark=*]  plot coordinates {(3,5)};
\node[left] at (3,4) {\footnotesize  $a_5$};
\draw[mark=square*]  plot coordinates {(3,4)};
\node[left] at (3,3) {\footnotesize  $1$};
\draw[mark=*]  plot coordinates {(3,3)};
\node[right] at (3,2) {\footnotesize  $b_5$};
\draw[mark=square*]  plot coordinates {(3,2)};
\node[right] at (3,1) {\footnotesize  $b_4$};
\draw[mark=*]  plot coordinates {(3,1)};
\node[right] at (3,0) {\footnotesize  $b_2$};
\draw[mark=square*]  plot coordinates {(3,0)};
\node[right] at (3,-1) {\footnotesize  $b_1$};
\draw[mark=*]  plot coordinates {(3,-1)};
\draw (3,-1) --(3,0)--(3,1)--(3,2)--(3,3)--(3,4)--(3,5)--(3,6);
\end{tikzpicture}
\qquad \qquad
\begin{tikzpicture}[scale=1]
\node[left] at (0,5) {\footnotesize  $1$};
\draw[mark=*]  plot coordinates {(0,5)};
\node[left] at (-.5,4) {\footnotesize  $a_5$};
\draw[mark=square*]  plot coordinates {(-.5,4)};
\node[right] at (.5,4) {\footnotesize  $b_5$};
\draw[mark=square*]  plot coordinates {(.5,4)};
\node[right] at (0,3) {\footnotesize  $b_4$};
\draw[mark=*]  plot coordinates {(0,3)};
\draw  (0,5) --(-.5,4) --(0,3);
\draw (0,5) --(.5,4) -- (0,3)--(0,2);
\draw  (-.5,4)-- (.5,4);
\node[right] at (0,2) {\footnotesize  $c_3$};
\draw[mark=*]  plot coordinates {(0,2)};
\node[left] at (-.5,1) {\footnotesize  $a_2$};
\draw[mark=square*]  plot coordinates {(-.5,1)};
\node[right] at (.5,1) {\footnotesize  $b_2$};
\draw[mark=square*]  plot coordinates {(.5,1)};
\node[right] at (0,0) {\footnotesize  $b_1$};
\draw[mark=*]  plot coordinates {(0,0)};
\draw  (0,2) --(-.5,1) --(0,0);
\draw (0,2) --(.5,1) -- (0,0);

\qquad
\node[left] at (3,6) {\footnotesize  $a_2$};
\draw[mark=square*]  plot coordinates {(3,6)};
\node[left] at (3,5) {\footnotesize  $a_5$};
\draw[mark=square*]  plot coordinates {(3,5)};
\node[left] at (3,4) {\footnotesize  $1$};
\draw[mark=*]  plot coordinates {(3,4)};
\node[right] at (3,3) {\footnotesize  $b_5$};
\draw[mark=square*]  plot coordinates {(3,3)};
\node[right] at (3,2) {\footnotesize  $b_4$};
\draw[mark=*]  plot coordinates {(3,2)};
\node[right] at (3,1) {\footnotesize  $\mathbf{c_3}$};
\draw[mark=*]  plot coordinates {(3,1)};
\node[right] at (3,0) {\footnotesize  $b_2$};
\draw[mark=square*]  plot coordinates {(3,0)};
\node[right] at (3,-1) {\footnotesize  $b_1$};
\draw[mark=*]  plot coordinates {(3,-1)};
\draw (3,-1) --(3,0)--(3,1)--(3,2)--(3,3)--(3,4)--(3,5)--(3,6);
\end{tikzpicture}
\end{center}
\caption{Two algebras (sides) with the same monoidal order (middle).\label{f:samemon}}
\end{figure}
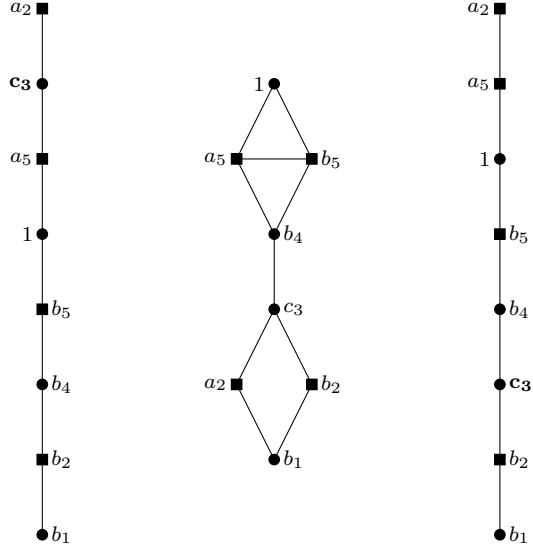

 However, as we show in Corollary~\ref{c:irc-mp}, if we further add the information of which elements are positive and which are negative (and impose a maximality and a minimality condition), then we can fully recover the whole residuated lattice structure. We will work toward this direction now, by establishing some useful properties of the monoidal preorder.

   \begin{lemma} \label{l:naturalmonoidal}
	Let $x,y$ be elements of an idempotent residuated chain.
   \begin{enumerate}
		\item $x \sqsubset y$ or $y \sqsubset x$ or $x=y$ or $xy\not = yx$.
		\item If $x,y \leq 1$, then $x \sqsubset y$ iff $x<y$. If $x,y \geq 1$, then $x \sqsubset y$ iff $x>y$.
		\item If $x$ and $y$ do not commute, then \emph{they are on the same layer}: For all $z$ we have $z \sqsubset x$ iff $z \sqsubset y$,  and $x \sqsubset z$ iff $y \sqsubset z$.
   \end{enumerate}
    \end{lemma}
    
    \begin{proof}
   1. If $x \not \sqsubset y$,  $y \not \sqsubset x$, $x \not=y$ and $xy=yx$, then by Lemma~\ref{l:naturalmonoidal1}(4), we get $x \not = xy=yx \not = y$, which contradicts conservativity.
	
 2. If $x,y \leq 1$, then by Lemma~\ref{l:idcon_mult}(1) $xy=x \mt y=yx$, so $x \sqsubset y$ iff $xy=x$ and $yx \not=y$ iff $x\mt y=x$ and $x \mt y \not=y$ iff $x<y$. If  $x,y \geq 1$, then by Lemma~\ref{l:idcon_mult}(1) $xy=x \jn y=yx$, so $x \sqsubset y$ iff $xy=x$ and $yx \not=y$ iff $x\jn y=x$ and $x \jn y \not=y$ iff $y<x$.

3. If $a$ and $b$ do not commute then they have opposite signs, say $a$ is the positive and $b$ is the negative, and either $a \mathrel{L} b$ or $a \mathrel{R} b$. We will show that for all $c\in A$ we have $c \sqsubset b$ iff $c \sqsubset a$, i.e. that $cb=bc=c\not = b$ iff $ca=ac=c \not = a$. 
	Assume first that  $a \mathrel{R} b$ so $ab=b$ and $ba=a$. If $cb=bc=c\not = b$, then $c \not = a$ (because $c$ commutes with $b$ and $a$ does not) so $c$ commutes with $a$ (because $b$ is the unique element that does not commute with $a$) and $cb=cba=ca$, hence $ac=ca=cb=c$. Conversely, if $ca=ac=c\not = a$, then $c \not = b$, so $c$ commutes with $b$ and $bc=bac=ac$, hence $bc=cb=ca=c$. The proof for the case where $a \mathrel{L} b$ is similar. 
	
	We will show that for all $c\in A$ we have $b \sqsubset c$ iff $a \sqsubset c$, i.e. that $cb=bc=b\not = c$ iff $ca=ac=a \not = c$. 
	Assume first that  $a \mathrel{L} b$ so $ab=a$ and $ba=b$. If $cb=bc=b\not = c$, then $c \not = a$ and $ac=abc=ab=a$. Also, $c$ commutes with $a$ (since $c \not = b$), hence $ac=ca=a$. The converse and also the proof for the case where $a \mathrel{R} b$ are similar.  
				\end{proof}

\begin{lemma}\label{l: 3 conditions}
Let  $\m A$ be an idempotent residuated chain. 
  \begin{enumerate}
  \item For $1\neq b \in A^-$, $b^\star$ is the smallest element of $A^+$ such that $b\sqsubset b^\star$ 
  \item For $1\neq a \in A^+$, $a^\star$ is the largest element  of $A^-$ such that $a^\star\sqsubset a$.
	\item $\sqsubseteq$ is \emph{layered}: If two distinct elements are not related by $\sqsubset$ nor $\sqsupset$, then they have different signs and they are in the same layer (their $\sqsubset$-upsets and downsets coincide). 
  \end{enumerate}
	\end{lemma}

 \begin{proof}
To verify the first two conditions, 
we refer to Corollary~\ref{c:RandL} and Figure~\ref{flowdiagrams}. 

 1. Assume that $b$ is strictly negative. If $a \mathrel{L} b$ or $a \mathrel{R} b$, for some (necessarily positive) $a$, then $b^\star$ is positive and $b^\star \prec a$, so $b^\star$ is the largest element of $(A^+, \leq)$ such that $b^\star < a$ and, by Lemma~\ref{l:naturalmonoidal}(2), $b^\star$ is the smallest element of $(A^+,\sqsubseteq)$ such that $a \sqsubset b^\star$. By Lemma~\ref{l:naturalmonoidal}(3), $b^\star$ is the smallest element of $(A^+,\sqsubseteq)$ such that $b\sqsubset b^\star $.
 Now, if $b$ is central, then $b^\ell=b^r=b^\star$ is a positive element,  $b b^\star=b^\star b$ and 	$b^\star b=b^\ell b=b^\ell \mt b=b$; also $b^\star \not=b$. Therefore,
	$b <_n b^\star$ which is equivalent to  $b\sqsubset b^\star $ by Lemma~\ref{l:naturalmonoidal1}(2). To show that $b^\star$ is the smallest element of $(A^+,\sqsubseteq)$ with this property, or equivalently that it is the largest element of $(A^+,\leq)$ with this property, let $b \sqsubset c$ for some $c \in A^+$, so $cb=bc=b \not = c$. By Lemma~\ref{l:idcon_mult}(3) we have $c \leq b^\ell=b^\star$.
 
2. Assume that $a$ is strictly positive. If $a \mathrel{L} b$ or $a \mathrel{R} b$, for some (necessarily negative) $b$, then $a^\star$ is negative and $a^\star \prec b$, so $a^\star$ is the largest element of $(A^-, \leq)$ such that $a^\star < b$ and by Lemma~\ref{l:naturalmonoidal}(2) the largest element of $(A^-,\sqsubseteq)$ such that $a^\star \sqsubset b$. Since $a$ and $b$ do not commute, by Lemma~\ref{l:naturalmonoidal}(3) their $\sqsubset$-downsets are equal, so $a^\star$ is the largest element of $(A^-,\sqsubseteq)$ such that $a^\star \sqsubset a$.
  Now, if $a$ is central so $a^\ell=a^r=a^\star$ is a negative element,  $a a^\star=a^\star a$ and 
$a^\star a=a^\ell a=a^\ell \mt a=a^\ell=a^\star$; also $a^\star \not=a$. Therefore,
	$a^\star <_n a$ which is equivalent to  $a^\star \sqsubset a$ by Lemma~\ref{l:naturalmonoidal1}(2). To show that $a^\star$ is the largest element of $(A^-,\sqsubseteq)$ with this property, or equivalently (by Lemma~\ref{l:naturalmonoidal}(2)) that it is the largest element of $(A^-,\leq)$ with this property, let $c \sqsubset a$ for some $c \in A^-$, so $ca=ac=c \not = a$. By Lemma~\ref{l:idcon_mult}(3) we have $c \leq a^\ell=a^\star$.
	
 3. To show that $\sqsubseteq$ is layered, we note that if two distinct elements are not comparable by $\sqsubset$ nor $\sqsupset$, then by Lemma~\ref{l:naturalmonoidal}(1) they cannot commute, so they have different signs and by Lemma~\ref{l:naturalmonoidal}(3) they are on the same layer. 
  \end{proof}

Given an idempotent residuated chain $\m A$, we associate to it the structure $(A, \sqsubseteq, A^+, A^-, 1,{}^\star)$.  Note that by Lemma~\ref{l:naturalmonoidal}(2) the positive cone appears inverted in the monoidal preorder.

 We now characterize these structures abstractly;  conditions 1 and 2 below are extracted from \cite{CZ2009}, where they are stipulated for the the natural order, but we formulate them in terms of the the richer monoidal preorder of \cite{GJM2020}.   \medskip

 $(P,  \sqsubseteq, P^+, P^-, 1, {}^\star)$ is an \emph{enhanced monoidal preorder}, if
 $(P, \sqsubseteq)$ is a pre-ordered set with sole maximum element  $1$ ($x \sqsubset 1$, for all $x\not=1$),  $P^+$ and $P^-$ are totally-ordered subsets of $P$ (i.e., the restriction of $\sqsubseteq$ to each of $P^+$, $P^-$ antisymmetric and total) such that $P^+ \cup P^-=P$ and $P^+ \cap P^-=\{1\}$ and ${}^\star$ is a unary operation on $P$ such that $1^\star =1$ and for all other elements:
  \begin{enumerate}
  \item For $1\neq b \in P^-$, $b^\star$ is the smallest element of $P^+$ such that $b\sqsubset b^\star$ 
  \item For $1\neq a \in P^+$, $a^\star$ is the largest element  of $P^-$ such that $a^\star\sqsubset a$.
	\item The preordered is \emph{layered}: If two distinct elements are not related by $\sqsubset$ nor $\sqsupset$, then they have different signs and their $\sqsubset$-upsets and downsets coincide. 
  \end{enumerate}

\begin{lemma}\label{l:emp to IdRC}
If $\m A$ is an idempotent residuated chain, then $(A, \sqsubseteq, A^+, A^-, 1,{}^\star)$ is an enhanced monoidal preorder.
\end{lemma}

 \begin{proof} By Lemma~\ref{l:naturalmonoidal1}, $\sqsubseteq$ is a preorder and for all $x \not =1$, we have $x1=x$ and $1x \not=1$, so $x \sqsubset 1$. Also, by Lemma~\ref{l:naturalmonoidal}(2), $A^+$ and $A^-$ are chains of $\sqsubseteq$ that intersect at $\{1\}$ and union to $A$. The remaining three conditions follow from Lemma~\ref{l: 3 conditions}.
\end{proof}

\subsection{From enhanced monoidal preorders to idempotent residuated chains}\label{s:empos to ircs}
We now prove that the newly introduced enhanced monoidal preorders are definitionally equivalent to idempotent residuated chains.

First, we extend a previous definition to an arbitrary preorder $\sqsubseteq$. If the $\sqsubset$-upsets and $\sqsubset$-downsets of two elements coincide, we say that \emph{they are in the same layer},  we write $\equiv_*$ for this equivalence relation and define a \emph{layer} to be an equivalence class of $\equiv_*$. (Note that this is exactly how we defined the notion of being in the same layer in Lemma~\ref{l:naturalmonoidal}(3).) We will denote by $\overline{x}$ the layer of $x$. 

Note that the layers form a chain with top element $1$, under the ordering: for all $x,y$, $\overline{x} \leq \overline{y}$ iff they are in the same layer or $x \sqsubseteq y$. Clearly, this definition is independent of the representatives (as elements of the same layer have the same $\sqsubset$-upsets and $\sqsubset$-downsets), reflexivity and antisymmetry follow from the properties of $\sqsubseteq$, and antisymmetry follows from Lemma 3.10(2), below.

\begin{lemma}\label{l:layers} Let  $(P,  \sqsubseteq, P^+, P^-, 1, {}^\star)$ be an enhanced monoidal preorder.
\begin{enumerate}
	\item There are at most two elements in the same layer. In this case, they have different sign.
	\item No set of pairwise mutually-comparable elements has more than two elements. In this case, these elements are in the same layer.
	\item No set of pairwise mutually-incomparable elements has more than two elements. In this case, they are in the same layer.	
	\item Each layer contains at most two elements, which are thus either mutually comparable or mutually incomparable.
	\item If the layer of an element is a singleton, then the element is comparable with any element of different sign.
\end{enumerate}
\end{lemma}

\begin{proof}
1. First note that only $1$ is in the same layer as $1$. Also, note that if $x,y \in P^-$ and $x \equiv_* y$, then in particular $\{z \in P^- \mid z \sqsubset x\}=\{z \in P^- \mid z \sqsubset y\}$, so since $P^-$ is a chain we get $x=y$. Likewise, $x,y \in P^+$ and $x \equiv_* y$, then $x=y$. Therefore, if $x \equiv_* y$ for distinct $x,y$, then they have different sign. Also, it is impossible to have three distinct elements in the same layer, as two of them would have the same sign and thus they will be equal.

2. If $x,y \in P^-$, then by antisymmetry they cannot be mutually comparable unless they are equal; likewise for $P^+$. So, elements that are mutually comparable have to have different sign. Thus, no three distinct elements are pairwise mutually comparable. Now let $x,y$ be distinct  mutually comparable and let $z \sqsubset x$. Then by transitivity $z \sqsubseteq y$. If we also had $y \sqsubseteq z$, then we would have $y \sqsubseteq z$, a contradiction. Therefore, $x$ and $y$ have the same $\sqsubset$-downsets. Likewise, they have the same upsets.

Property 3 follows from Property 3 of the definition of enhanced monoidal preorders. Property 4 follows from the others.

5. If two elements are not comparable, then in particular they are not related by $\sqsubset$ nor $\sqsupset$, so by condition 3 in the definition of enhanced monoidal preorders, they are in the same layer. So, if the layer of $x$ is a singleton and $y$ is incomparable to $x$, then $y=x$, a contradiction. So, $y$ has to be comparable to $x$.
\end{proof}

	We will draw the Hasse diagrams for enhanced monoidal preorders on two columns, following the convention of putting the elements of the chain $P^+$ on the left column and the elements of the chain $P^-$ on the right column. We place $1$ on the left. 

Figure~\ref{f:samemondiff} shows how the two chains of Figure~\ref{f:samemon} can be distinguished via their enhanced monoidal preorders.

\begin{figure}[htbp]
\begin{center}
\begin{tikzpicture}[scale=1]
\node[left] at (-.5,5) {\footnotesize  $1$};
\draw[mark=*]  plot coordinates {(-.5,5)};
\node[left] at (-.5,4) {\footnotesize  $a_5$};
\draw[mark=square*]  plot coordinates {(-.5,4)};
\node[right] at (.5,4) {\footnotesize  $b_5$};
\draw[mark=square*]  plot coordinates {(.5,4)};
\node[right] at (.5,3) {\footnotesize  $b_4$};
\draw[mark=*]  plot coordinates {(.5,3)};
\draw  (-.5,5) --(-.5,4) --(.5,3);
\draw (-.5,5) --(.5,4) -- (.5,3)--(-.5,2);
\draw  (-.5,4)-- (.5,4);
\node[left] at (-.5,2) {\footnotesize  $a_3$};
\draw[mark=*]  plot coordinates {(-.5,2)};
\node[left] at (-.5,1) {\footnotesize  $a_2$};
\draw[mark=square*]  plot coordinates {(-.5,1)};
\node[right] at (.5,1) {\footnotesize  $b_2$};
\draw[mark=square*]  plot coordinates {(.5,1)};
\node[right] at (.5,0) {\footnotesize  $b_1$};
\draw[mark=*]  plot coordinates {(.5,0)};
\draw  (-.5,2) --(-.5,1) --(.5,0);
\draw (-.5,2) --(.5,1) -- (.5,0);
\qquad
\node[left] at (3,6) {\footnotesize  $a_2$};
\draw[mark=square*]  plot coordinates {(3,6)};
\node[left] at (3,5) {\footnotesize  $\mathbf{a_3}$};
\draw[mark=*]  plot coordinates {(3,5)};
\node[left] at (3,4) {\footnotesize  $a_5$};
\draw[mark=square*]  plot coordinates {(3,4)};
\node[left] at (3,3) {\footnotesize  $1$};
\draw[mark=*]  plot coordinates {(3,3)};
\node[right] at (3,2) {\footnotesize  $b_5$};
\draw[mark=square*]  plot coordinates {(3,2)};
\node[right] at (3,1) {\footnotesize  $b_4$};
\draw[mark=*]  plot coordinates {(3,1)};
\node[right] at (3,0) {\footnotesize  $b_2$};
\draw[mark=square*]  plot coordinates {(3,0)};
\node[right] at (3,-1) {\footnotesize  $b_1$};
\draw[mark=*]  plot coordinates {(3,-1)};
\draw (3,-1) --(3,0)--(3,1)--(3,2)--(3,3)--(3,4)--(3,5)--(3,6);
\end{tikzpicture}
\qquad
\begin{tikzpicture}[scale=1]
\node[left] at (-.5,5) {\footnotesize  $1$};
\draw[mark=*]  plot coordinates {(-.5,5)};
\node[left] at (-.5,4) {\footnotesize  $a_5$};
\draw[mark=square*]  plot coordinates {(-.5,4)};
\node[right] at (.5,4) {\footnotesize  $b_5$};
\draw[mark=square*]  plot coordinates {(.5,4)};
\node[right] at (.5,3) {\footnotesize  $b_4$};
\draw[mark=*]  plot coordinates {(.5,3)};
\draw  (-.5,5) --(-.5,4) --(.5,3);
\draw (-.5,5) --(.5,4) -- (.5,3)--(.5,2);
\draw  (-.5,4)-- (.5,4);
\node[right] at (.5,2) {\footnotesize  $b_3$};
\draw[mark=*]  plot coordinates {(.5,2)};
\node[left] at (-.5,1) {\footnotesize  $a_2$};
\draw[mark=square*]  plot coordinates {(-.5,1)};
\node[right] at (.5,1) {\footnotesize  $b_2$};
\draw[mark=square*]  plot coordinates {(.5,1)};
\node[right] at (.5,0) {\footnotesize  $b_1$};
\draw[mark=*]  plot coordinates {(.5,0)};
\draw  (.5,2) --(-.5,1) --(.5,0);
\draw (.5,2) --(.5,1) -- (.5,0);

\qquad
\node[left] at (3,6) {\footnotesize  $a_2$};
\draw[mark=square*]  plot coordinates {(3,6)};
\node[left] at (3,5) {\footnotesize  $a_5$};
\draw[mark=square*]  plot coordinates {(3,5)};
\node[left] at (3,4) {\footnotesize  $1$};
\draw[mark=*]  plot coordinates {(3,4)};
\node[right] at (3,3) {\footnotesize  $b_5$};
\draw[mark=square*]  plot coordinates {(3,3)};
\node[right] at (3,2) {\footnotesize  $b_4$};
\draw[mark=*]  plot coordinates {(3,2)};
\node[right] at (3,1) {\footnotesize  $\mathbf{b_3}$};
\draw[mark=*]  plot coordinates {(3,1)};
\node[right] at (3,0) {\footnotesize  $b_2$};
\draw[mark=square*]  plot coordinates {(3,0)};
\node[right] at (3,-1) {\footnotesize  $b_1$};
\draw[mark=*]  plot coordinates {(3,-1)};
\draw (3,-1) --(3,0)--(3,1)--(3,2)--(3,3)--(3,4)--(3,5)--(3,6);
\end{tikzpicture}
\end{center}
\caption{Distinguishing two algebras with the same monoidal order\label{f:samemondiff}}
\end{figure}
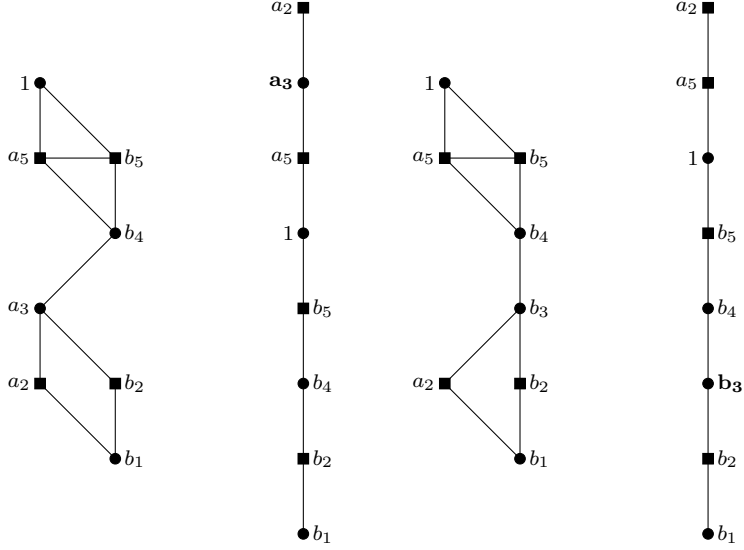

\begin{remark}
Note that in the second algebra/preorder in Figure~\ref{f:samemondiff} we have $b_3^\star=b_3^\ell=b_3^r=a_5$. This can be seen from Figure~\ref{flowdiagrams}/Figure~\ref{f:configurations} and also by the fact that $a_5$ is the $\leq$-biggest, i.e. $\sqsubseteq$-smallest, positive element such that $b_3 \sqsubset a_5$. In other words, $a_5$ is the $\sqsubseteq$-smallest positive element in a layer above $b_3$. But it is not in the smallest layer above $b_3$, which is $\{b_4\}$, as that does not contain any positive elements.  
Therefore, it is delicate to argue about the position of $b^\star$, for some negative $b$, as for some negative $b$ it may not be in the layer above $b$. The correct way to argue is to consider elements of $P^+$ closest to $b$ from above, and this is how the arguments were structured in the proof of Lemma~\ref{l: 3 conditions}.
\end{remark}

  Figure~\ref{f:configurations} shows diagrammatically how to move seamlessly between 
	enhanced monoidal preorders and flow diagrams.
    
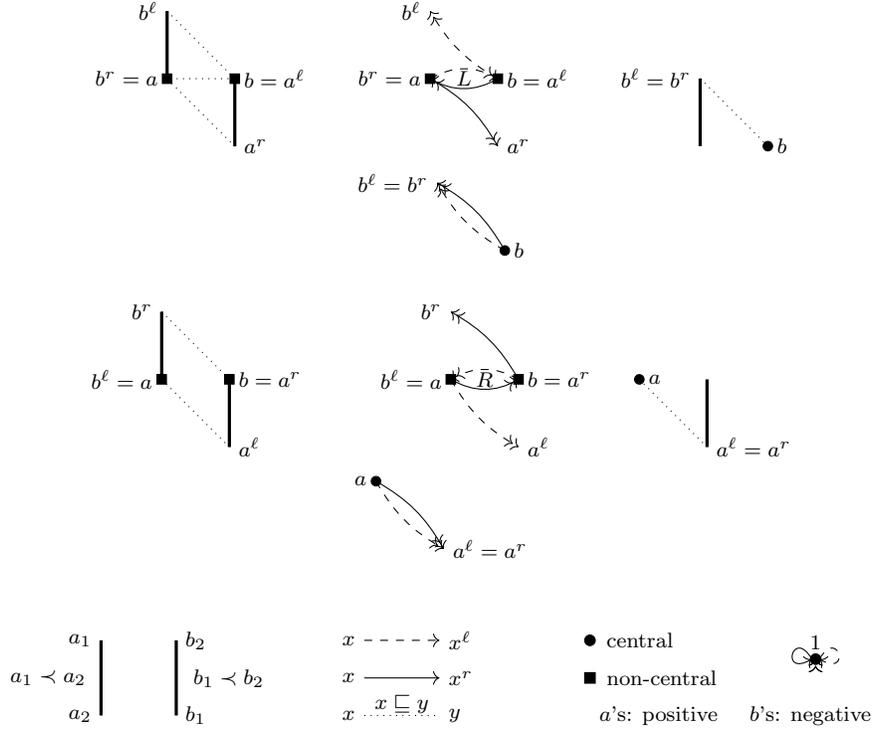
\begin{figure}[htb]
\begin{center}
\begin{tikzpicture}[scale=0.9]
\node[left]  at (-.5,5) {\footnotesize  $b^\ell$};

\node[left] (a) at (-.5,4) {\footnotesize  $b^r=a$};
\draw[mark=square*]  plot coordinates {(-.5,4)};
\node[right] (b) at (.5,4) {\footnotesize  $b=a^\ell$};
\draw[mark=square*]  plot coordinates {(.5,4)};
\node[right] at (.5,3) {\footnotesize  $a^r$};

\draw[very thick] (-.5,5) --(-.5,4);
\draw[dotted] (-.5,5) --(.5,4);
\draw[dotted]   (.5,4) -- (-.5,4);
\draw[dotted] (-.5,4) --(.5,3);
\draw[very thick] (.5,4) -- (.5,3);
\end{tikzpicture}
\quad 
\begin{tikzpicture}[scale=0.9]
\node[left]  at (-.5,5) {\footnotesize  $b^\ell$};

\node[left] (a) at (-.5,4) {\footnotesize  $b^r=a$};
\draw[mark=square*]  plot coordinates {(-.5,4)};
\node at (0,4) {\footnotesize  $L$};
\node[right] (b) at (.5,4) {\footnotesize  $b=a^\ell$};
\draw[mark=square*]  plot coordinates {(.5,4)};
\node[right] at (.5,3) {\footnotesize  $a^r$};

\draw  [->>,dashed]  (.5,4) to [out=150,in=-60] (-.5,5);
\draw  [->>]  (.5,4) to [out=-150,in=-30] (-.5,4);
\draw  [->>, dashed]  (-.5,4) to [out=30,in=150] (.5,4);
\draw [->>]  (-.5,4) to [out=-30,in=120] (.5,3);
\end{tikzpicture}
\quad
\begin{tikzpicture}[scale=0.9]
\node[left]  at (-.5,5) {\footnotesize  $b^\ell=b^r$};

\node[right] (b) at (.5,4) {\footnotesize  $b$};
\draw[mark=*]  plot coordinates {(.5,4)};
\draw[very thick] (-.5,5) --(-.5,4);
\draw[dotted] (-.5,5) --(.5,4);
\end{tikzpicture}
\quad
\begin{tikzpicture}[scale=0.9]
\node[left]  at (-.5,5) {\footnotesize  $b^\ell=b^r$};

\node[right] (b) at (.5,4) {\footnotesize  $b$};
\draw[mark=*]  plot coordinates {(.5,4)};

\draw  [->>,dashed]  (.5,4) to [out=150,in=-60] (-.5,5);
\draw  [->>]  (.5,4) to [out=120,in=-30] (-.5,5);
\end{tikzpicture}
\end{center}
\begin{center}
\begin{tikzpicture}[scale=0.9]
\node[left]  at (-.5,5) {\footnotesize  $b^r$};

\node[left] (a) at (-.5,4) {\footnotesize  $b^\ell=a$};
\draw[mark=square*]  plot coordinates {(-.5,4)};
\node[right] (b) at (.5,4) {\footnotesize  $b=a^r$};
\draw[mark=square*]  plot coordinates {(.5,4)};
\node[right] at (.5,3) {\footnotesize  $a^\ell$};

\draw[very thick] (-.5,5) --(-.5,4);
\draw[dotted] (-.5,5) --(.5,4);
\draw[dotted] (-.5,4) --(.5,3);
\draw[very thick] (.5,4) -- (.5,3);
\end{tikzpicture}
\qquad
\begin{tikzpicture}[scale=0.9]
\node[left]  at (-.5,5) {\footnotesize  $b^r$};

\node[left] (a) at (-.5,4) {\footnotesize  $b^\ell=a$};
\draw[mark=square*]  plot coordinates {(-.5,4)};
\node at (0,4) {\footnotesize  $R$};
\node[right] (b) at (.5,4) {\footnotesize  $b=a^r$};
\draw[mark=square*]  plot coordinates {(.5,4)};
\node[right] at (.5,3) {\footnotesize  $a^\ell$};

\draw  [->>]  (.5,4) to [out=120,in=-30] (-.5,5);
\draw  [->>, dashed]  (.5,4) to [out=150,in=30] (-.5,4);
\draw  [->>]  (-.5,4) to [out=-30,in=-150] (.5,4);
\draw [->>, dashed]  (-.5,4) to [out=-60,in=150] (.5,3);
\end{tikzpicture}
\quad
\begin{tikzpicture}[scale=0.9]
\node[right] (b) at (-.5,4) {\footnotesize  $a$};
\draw[mark=*]  plot coordinates {(-.5,4)};
\node[right] at (.5,3) {\footnotesize  $a^\ell=a^r$};

\draw[dotted] (-.5,4) --(.5,3);
\draw[very thick] (.5,4) -- (.5,3);
\end{tikzpicture}
\quad
\begin{tikzpicture}[scale=0.9]
\node[left] (a) at (-.5,4) {\footnotesize  $a$};
\draw[mark=*]  plot coordinates {(-.5,4)};
\node[right] at (.5,3) {\footnotesize  $a^\ell=a^r$};

\%draw (-.5,4) -- (-.5,5);

\draw [->>]  (-.5,4) to [out=-30,in=120] (.5,3);
\draw [->>, dashed]  (-.5,4) to [out=-60,in=150] (.5,3);
\end{tikzpicture}
\end{center}
\vskip 5pt
\begin{center}
\begin{tikzpicture}[scale=1]
\draw[very thick] (-3.5,0)--(-3.5,1);
\node[left] at (-3.5,1) {\footnotesize  $a_1$};
\node[left] at (-3.5,0) {\footnotesize  $a_2$};
\node at (-4.2, 0.5) {\footnotesize  $a_1 \prec a_2$};
\draw[very thick] (-2.5,0)--(-2.5,1);
\node[right] at (-2.5,1) {\footnotesize  $b_2$};
\node[right] at (-2.5,0) {\footnotesize  $b_1$};
\node at (-1.8, 0.5) {\footnotesize  $b_1 \prec b_2$};
\draw[->, dashed] (0,1)--(1,1);
\node[left] at (0,1) {\footnotesize  $x$};
\node[right] at (1,1) {\footnotesize  $x^\ell$};
\draw[->] (0,0.5)--(1,0.5);
\node[left] at (0,0.5) {\footnotesize  $x$};
\node[right] at (1,0.5) {\footnotesize  $x^r$};
\node[right] at (3.1,1) {\footnotesize  central};
\draw[mark=*]  plot coordinates {(3,1)};
\draw[dotted] (0,0)--(1,0);
\node[left] at (0,0) {\footnotesize  $x$};
\node[right] at (1,0) {\footnotesize  $y$};
\node at (0.5, 0.15) {\footnotesize  $x \sqsubseteq y$};

\node[right] at (3.1,0.5) {\footnotesize  non-central};
\draw[mark=square*]  plot coordinates {(3,0.5)};
\node[right] at (3, 0)  {\footnotesize  $a$'s: positive};
\node[right] at (5, 0) {\footnotesize  $b$'s: negative};
\node[above]  at (6,0.75) {\footnotesize  $1$};
\draw[->, dashed] (6.05,0.75) to [in=-30, out=60, loop] ();
\draw[->] (5.95,0.75) to [in=210, out=120, loop] ();
\draw[mark=*]  plot coordinates {(6,0.75)};
\draw [->>]  (6,0.75) to [out=0,in=120] (6,0.75);
\draw [->>, dashed]  (6,0.75) to [out=-60,in=150] (6,0.75);
\end{tikzpicture}
\end{center}
\caption{Translating between enhanced monoidal preorders and flow diagrams for each of the four cases.\label{f:configurations}}
\end{figure}

  Conversely, given an enhanced monoidal preorder $(P,  \sqsubseteq, P^+, P^-, 1, {}^\star)$, we define the ordered algebra $\m A$ with underlying set $A=P$, with order given by
	\begin{center}
		$x \leq y$ iff ($x,y \in P^-$ and $x \sqsubseteq y$) or  ($x,y \in P^+$ and $y \sqsubseteq x$) or ($x \in P^-$ and $y \in P^+$).
	\end{center}
The inverses are given by: $x^\ell=x^r=x^\star$, if $x$ is a $\sqsubseteq$-conical element ($\sqsubseteq$-comparable to every element) and, for 
$a \in P^+$ and $b \in P^-$,
 \begin{enumerate}
\item $a^\ell= b$, $a^r=a^\star$, $b^\ell=b^\star$, $b^r=a$, if $a,b$ are mutually comparable and  
\item $a^\ell=a^\star$, $a^r=b$, $b^\ell=a$, $b^r=b^\star$, if $a,b$ are incomparable
\end{enumerate}
Observe that $^\ell$ and $^r$ are well-defined by Lemma~\ref{l:layers}. Multiplication and the divisions are given by:

$$x y 
	 = \begin{cases} 
      x\meet y &  x \leq y^\ell \\
      x\join y & y^\ell < x 
   \end{cases} 
\qquad 
x\ld y = \begin{cases} 
      x^r \jn y & x\leq y \\
      x^r \mt y & y < x
   \end{cases} \qquad 
y\rd x = \begin{cases} 
      x^\ell \jn y & x\leq y \\
      x^\ell \mt y & y < x
   \end{cases}
$$

   \begin{lemma}
 The algebra associated to an enhanced monoidal preorder is an idempotent residuated chain.
  \end{lemma}
  
  \begin{proof}
 Let $\m A$ be the algebra associated to the enhanced monoidal preorder $(P,  \sqsubseteq, P^+, P^-, 1, {}^\star)$. Since $P^-$ and $P^+$ are sub-chains of $P$ and since every element of $P^-$ is less than every element of $P^+$, we get that $(A, {\leq})$ is a chain.

 Note that the inverse $b^i\in \{b^\ell, b^r\}$ of a negative element $b$ is a positive element either at the same layer as $b$ or it is the largest element of $(A^+, \leq)$ such that $b\sqsubset b^i$. If $b_1 < b_2$ are two negative elements then $b_1 \sqsubset b_2$. In case  $b_2 \sqsubset b_2^i$ we have $b_1 \sqsubset b_2^i$; hence  
$b_2^i \leq b_1^i$. In case $b_2^i$ is at the same layer as $b_2$, since $b_1 \sqsubset b_2$ we have $b_1 \sqsubset b_2^i$; hence $b_2^i \leq b_1^i$. A similar argument shows that if $a_1 < a_2$ are two positive elements, then  $a_2^i \leq a_1^i$. The remaining way to have $x \leq y$ in $A$ is to have $x \in P^-$ and $y \in P^+$, in which case $x^\ell, x^r \in P^+$ and $y^\ell, y^r \in P^-$, so $y^\ell \leq x^\ell$ and $y^r \leq x^r$. Therefore, both inverse operations are order reversing on $\m A$. 

Assume that $b$ is a strictly negative element. In case that $a$ is on the same layer as $b$ and they are incomparable we have $b^\ell=a$, and in case  $a$ is on the same layer as $b$ and they are mutually comparable, or $b$ is $\sqsubseteq$-conical, we have that $b^\ell$
it is the largest element of $(A^+, \leq)$ such that $b\sqsubset b^\ell$. In the first case, we have $(b^\ell)^r=b$. In the second case note that $(b^\ell)^r$ is a negative element either at the same layer as $b^\ell$ or it is a the largest element of $(A^-, \leq)$ such that $(b^\ell)^r \sqsubset b^\ell$ or $b^\ell=1=(b^\ell)^r$. Therefore, in all cases $b \leq b^{\ell r}$ and likewise $b \leq b^{r \ell}$. Also, we can similarly show that $a \leq a^{\ell r}$ and  $a \leq a^{r \ell}$ for every strictly positive element, and clearly $1^{\ell r}=1^{r \ell}=1$. Therefore, $({}^\ell, {}^r)$ forms a Galois connection on $\m A$. 

Also, it is clear from the definition that there is no element $\leq$-between $x^\ell$ and $x^r$. Finally, since $1$ is $\sqsubseteq$-conical and both positive and negative, the definitions yield $1^\ell=1^r=1$. Therefore, $\m A':=(A, \mt, \jn, {}^\ell, {}^r, 1)$ is an idempotent Galois connection. By Lemma~\ref{l:idGC}, the expansion $\m R({\m A'})$ of $\m A'$ is an idempotent residuated lattice. Since the multiplication and divisions of $\m A$ are defined in the same way as $\m R({\m A'})$, we get that $\m A=\m R({\m A'})$.
  \end{proof}
	
	   \begin{corollary}\label{c:irc-mp}
		The defined correspondences between idempotent residuated chains and enhanced monoidal preorders are inverses of each other. 
  \end{corollary}
	
	\begin{proof}
	Let $\m P=(P,  \sqsubseteq, P^+, P^-, 1, {}^\star)$ be an enhanced monoidal preorder, $\m A$ its associated idempotent residuated chain, and  $\m P_A=(A, \sqsubseteq_A, A^+, A^-, 1_A,{}^\star)$ the enhanced monoidal preorder associated to $\m A$. By following the definitions we see that $A=P$, $A^-=P^-$, $A^+=P^+$ and $1_A=1$. We have $x \sqsubseteq_A y$ iff $xy=x$ in $\m A$ iff ($x^r<y\leq x$ or $x\leq y \leq x^r$) iff ($x^r<y\leq x>1$ or $1 \geq x\leq y \leq x^r$) iff 
	($1< x \sqsubseteq y$ or $1\geq x \sqsubseteq y$) iff $x \sqsubseteq y$. Therefore, $\m P_A=\m P$.
	We used two equivalences, which we now prove. 
	
	First we show that $x^r<y\leq x>1$ is equivalent to $1< x \sqsubseteq y$; in other words that under the assumption $x>1$, the statements  $x^r<y\leq x$ and $x \sqsubseteq y$ are equivalent.
	We have $x^r<y\leq x$ iff ($1<y\leq x$ or $x^r<y\leq 1$) iff ($1\sqsupset_{P^+} y\sqsupseteq x$ or $x^r \sqsubset y \sqsubseteq_{P^-} 1$) iff   ($1\sqsupset_{P^+} y\sqsupseteq x$ or $x \sqsubseteq y \sqsubseteq_{P^-} 1$) iff $x \sqsubseteq y$. We used that $x^r \sqsubset y \sqsubseteq_{P^-} 1$ is equivalent to $x \sqsubseteq y \sqsubseteq_{P^-} 1$, which we now prove. Note from the definition of $x^r$ in $\m A$ that if there is $z\not = x$ in the layer of $x$ with $z$ incomparable to $x$, then $x^r=z$, so in that case $x^r \sqsubset y \sqsubseteq_{P^-} 1$ is equivalent to $x \sqsubseteq y \sqsubseteq_{P^-} 1$ (since $x^r, x$ have the same $\sqsubset$-upsets and they are incomparable). On the other hand, if there is $z\not = x$ in the layer of $x$ with $z$ comparable to $x$, then $x^r=x^\star$, which is the largest element of $P^-$ with $x^\star \sqsubset x$, or equivalently with  $x^\star \sqsubset z$ since 
	$x$ and $z$ are in the same layer. Therefore, in this case  $x^r \sqsubset y \sqsubseteq_{P^-} 1$ is equivalent to $z \sqsubseteq y \sqsubseteq_{P^-} 1$ and to $x \sqsubseteq y \sqsubseteq_{P^-} 1$, since $x,z$ are in the same layer and comparable.
	 Finally, if $x$ is the only element of its layer, then $x^r=x^\star$, which is the largest element of $P^-$ with $x^\star \sqsubset x$. So, in that case $x^r \sqsubset y \sqsubseteq_{P^-} 1$ is equivalent to $x \sqsubseteq y \sqsubseteq_{P^-} 1$, by Lemma~\ref{l:layers}(5). 

	Now we assume $x \leq 1$ and prove that $x\leq y \leq x^r$ is equivalent to $x \sqsubseteq y$. We have  $x\leq y \leq x^r$ iff ($x\leq y \leq 1$ or  $1\leq y \leq x^r$) iff ($x \sqsubseteq y \sqsubseteq_{P^-} 1$ or  $1\sqsupseteq_{P^+} y \sqsupseteq x^r$) iff ($x \sqsubseteq y \sqsubseteq_{P^-} 1$ or  $1\sqsupseteq_{P^+} y \sqsupseteq x$) iff $x \sqsubseteq y$. We used the equivalence of  $1\sqsupseteq_{P^+} y \sqsupseteq x^r$ and $1\sqsupseteq_{P^+} y \sqsupseteq x$, which we now prove.
	Note from the definition of $x^r$ in $\m A$ that if there is $z\not = x$ in the layer of $x$ with $z$ comparable to $x$, then $x^r=z$, so in that case $x^r \sqsubseteq y \sqsubseteq_{P^+} 1$ is equivalent to $x \sqsubseteq y \sqsubseteq_{P^+} 1$ (since $x^r, x$ have the same $\sqsubset$-upsets and they are comparable). 
	On the other hand, if there is $z\not = x$ in the layer of $x$ with $z$ incomparable to $x$, then $x^r=x^\star$, which is the smallest element of $P^+$ with $x \sqsubset x^\star$, or equivalently with  $z \sqsubset x^\star$ since 
	$x$ and $z$ are in the same layer. Therefore, in this case  $x^r \sqsubseteq y \sqsubseteq_{P^+} 1$ is equivalent to $z \sqsubset y \sqsubseteq_{P^+} 1$ and to $x \sqsubseteq y \sqsubseteq_{P^+} 1$, since $x,z$ are in the same layer and incomparable.
	 Finally, if $x$ is the only element of its layer, then $x^r=x^\star$, which is the smallest element of $P^+$ with $x \sqsubset x^\star$. So, in that case $x^r \sqsubseteq y \sqsubseteq_{P^+} 1$ is equivalent to $x \sqsubseteq y \sqsubseteq_{P^+} 1$, by Lemma~\ref{l:layers}(5).

Now let $\m A$ be an idempotent residuated chain, $(A, \sqsubseteq, A^+, A^-, 1,{}^\star)$ its enhanced monoidal preorder, and let $\m B=(B, \mt, \jn, \cdot, \ld, \rd, 1_B)$ be the idempotent residuated chain associated to that. By the definitions it follows that $B=A$ and $1_B=1$. Also, $x\leq_B y$ iff  ($x,y \in A^-$ and $x \sqsubseteq y$) or  ($x,y \in A^+$ and $y \sqsubseteq x$) or ($x \in A^-$ and $y \in A^+$) iff  ($x,y \in A^-$ and $xy=x$) or  ($x,y \in A^+$ and $yx=y$) or ($x \in A^-$ and $y \in A^+$) iff  ($x,y \in A^-$ and $x\mt y=x$) or  ($x,y \in A^+$ and $y \jn x=y$) or ($x \in A^-$ and $y \in A^+$) iff $x \leq y$.
Now, let $x$ be an element of $B$ with a singleton layer. Then for all $y$ we have $x \sqsubset y$ or $y \sqsubseteq x$; so, by Lemma~\ref{l:naturalmonoidal1}(4), $x$ and $y$ commute. Therefore, $x$ is central in $\m A$ and $x^\ell=x^r=x^\star$. Since this is exactly how the inverses of $x$ in $\m B$ are defined, we get that the inverses of $x$ in $\m A$ and in $\m B$ coincide. If $x$ is an element of $\m B$ in the same layer as another element $y$, incomparable to $x$, then $x \not \sqsubseteq y$ and $y  \not \sqsubseteq x$, so $xy \not =x$ and $yx \not = y$. By conservativity, we get $xy=y$ and $yx=x$. Then the values of $x^\ell$ and $x^r$ given by Corollary~\ref{c:RandL} coincide with the definition of  $x^\ell$ and $x^r$ in $\m B$. Likewise, the values of $x^\ell$ and $x^r$ in $\m A$ and in $\m B$ coincide, if there is a  $y$  incomparable to $x$ and in the same layer as $x$. Consequently, the $\{\mt, \jn, {}^\ell, {}^r, 1\}$-reducts of $\m A$ and $\m B$ coincide. By Corollary~\ref{c:idempotentchainsoperations}, we get that $\m A=\m B$.
	\end{proof}
  
    Therefore, the common notation between idempotent residuated chains and enhanced monoidal preorders (such as $^{\star}$, $\sqsubseteq$, etc) is consistent and we will use it when talking about either one of these structures.

	\begin{remark}
	The motivation for considering the natural order in \cite{CZ2009} is the connection to semigroup theory, and the motivation for the monoidal preorder in \cite{GJM2020} seems to be the desire to capture the behavior of multiplication. However, our reason for considering the enhanced monoidal preorder is very different. After all, by Lemma~\ref{l:idGC} we already know that multiplication can be captured in the language $\{\mt, \jn, {}^\ell, {}^r, 1\}$, i.e. by the idempotent Galois connection reduct. Hence subalgebra generation, which is crucial for proving amalgamation, occurs solely via ${}^\ell$ and ${}^r$. Unfortunately, this generation takes place remotely: A positive element $a$ will generate negative elements $a^\ell$ and $a^r$, which are both far away from $a$ in the chain ordering. The benefit of the enhanced monoidal preorder is that the action of  ${}^\ell$ and ${}^r$ happens locally: $a^\ell$ and $a^r$ are adjacent to $a$ in the enhanced monoidal preorder. Therefore, the reason that we consider enhanced monoidal preorders is that, in that setting, subalgebras are convex subsets. The insight that convex subsets should be kept intact is key to how to amalgamate such structures in a transparent way (see Section~\ref{s:amalgamation}).
	\end{remark}

\subsection{Subalgebras}\label{sec:subalgebras of chains}
		Note that if $x$ is not central, then $x \not = x^*$ and its layer is exactly $\{x,x^*\}$. Also, if $x$ is central, then its layer is $\{x\}$. 
		For all $x$, we define $x^\leftrightarrow$ to be $x$, if $x$ is central, and $x^*$, if $x$ is not central. Therefore, $x^\leftrightarrow$ is the only other element in the layer of $x$, if the layer has two elements, and it is equal to $x$ if the layer has only one element. The next lemma follows directly from Corollary~\ref{c:RandL} and the definition of $x^\leftrightarrow$.
		
		   \begin{lemma}\label{l:generation} Let $\m A$ be idempotent residuated chain and $x \in A$. We have $\{x^\ell, x^r\}=\{x^\star, x^*\}=\{x^\star, x^\leftrightarrow\}$.
      \end{lemma}

   \begin{lemma}\label{l:subemp}
Let $\m A$ and $\m B$ be idempotent residuated chains and let $\m P_{\m A}$ and $\m P_{\m B}$ be the corresponding enhanced monoidal preorders. Then $\m A$ is a subalgebra of $\m B$ iff  $\m P_{\m A}$ is closed under ${}^\leftrightarrow$, $^\star$ and $1$.
      \end{lemma}

  \begin{proof} Recall that subalgebra generation is done via $^\ell$, $^r$ and $1$. 
	Also, we know that $\{x^\ell, x^r\}=\{x^*, x^\star\}$, for all $x$, due to the comparability of the two inverses. From	Figure~\ref{f:configurations} we see that for non-central elements, closure under $^*$ is the same as closure under same-layer elements. 
   \end{proof}

Figure~\ref{f:connectedcomponents} shows three enhanced monoidal preorders. We call the corresponding idempotent residuated chains $\m A$, $\m B$ and $\m C$. The subalgebras of $\m A$ are $\{1\}$, $\{b_4, b_5,a_5, 1\}=\upset_* b_4$, $\{b_1, a_2, b_2, a_3,1\}={\downset}_* a_3$ and $\m A$ itself (see Section~\ref{s:weaklyinvolutive} for the definitions of $\upset_*$ and $\downset_*$). The subalgebras of $\m B$ are $\{1\}$, $\{\ldots a_4, b_4, b_5,a_5, 1\}=\upset_* \{b_5, b_4, ....\}$, $\{b_1, a_2, b_2, a_3, 1\}$ and $\m B$ itself. The subalgebras of $\m C$ are only $\{1\}$, $\{b_4, b_5,a_5, 1\}$ and $\m C$ itself. Note that $\{b_1, a_2, b_2, b_3,1\}$ is not a subalgebra of $\m C$, since $b_2^\star=b_3^\star=a_5$.

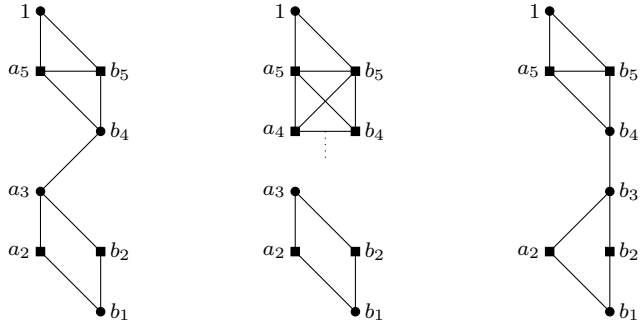
\begin{figure}[htbp]
    \begin{center}
\begin{tikzpicture}[scale=0.8]
\node[left] at (-.5,5) {\footnotesize  $1$};
\draw[mark=*]  plot coordinates {(-.5,5)};
\node[left] at (-.5,4) {\footnotesize  $a_5$};
\draw[mark=square*]  plot coordinates {(-.5,4)};
\node[right] at (.5,4) {\footnotesize  $b_5$};
\draw[mark=square*]  plot coordinates {(.5,4)};
\node[right] at (.5,3) {\footnotesize  $b_4$};
\draw[mark=*]  plot coordinates {(.5,3)};
\draw  (-.5,5) --(-.5,4) --(.5,3);
\draw (-.5,5) --(.5,4) -- (.5,3);
\draw  (-.5,4)-- (.5,4);
\draw (.5,3)--(-.5,2);
\node[left] at (-.5,2) {\footnotesize  $a_3$};
\draw[mark=*]  plot coordinates {(-.5,2)};
\node[left] at (-.5,1) {\footnotesize  $a_2$};
\draw[mark=square*]  plot coordinates {(-.5,1)};
\node[right] at (.5,1) {\footnotesize  $b_2$};
\draw[mark=square*]  plot coordinates {(.5,1)};
\node[right] at (.5,0) {\footnotesize  $b_1$};
\draw[mark=*]  plot coordinates {(.5,0)};
\draw  (-.5,2) --(-.5,1) --(.5,0);
\draw (-.5,2) --(.5,1) -- (.5,0);
\end{tikzpicture}
\qquad \qquad 
\begin{tikzpicture}[scale=0.8]
\node[left] at (-.5,5) {\footnotesize  $1$};
\draw[mark=*]  plot coordinates {(-.5,5)};
\node[left] at (-.5,4) {\footnotesize  $a_5$};
\draw[mark=square*]  plot coordinates {(-.5,4)};
\node[right] at (.5,4) {\footnotesize  $b_5$};
\draw[mark=square*]  plot coordinates {(.5,4)};
\node[left] at (-.5,3) {\footnotesize  $a_4$};
\draw[mark=square*]  plot coordinates {(-.5,3)};
\node[right] at (.5,3) {\footnotesize  $b_4$};
\draw[mark=square*]  plot coordinates {(.5,3)};
\draw[dotted] (0,3)--(0,2.5);
\draw  (-.5,5) --(-.5,4) --(.5,3);
\draw (-.5,5) --(.5,4) -- (.5,3);
\draw  (-.5,4)-- (.5,4);
\draw (-.5,4)--(-.5,3)--(.5,3);
\draw (-.5,3)--(.5,4);
\node[left] at (-.5,2) {\footnotesize  $a_3$};
\draw[mark=*]  plot coordinates {(-.5,2)};
\node[left] at (-.5,1) {\footnotesize  $a_2$};
\draw[mark=square*]  plot coordinates {(-.5,1)};
\node[right] at (.5,1) {\footnotesize  $b_2$};
\draw[mark=square*]  plot coordinates {(.5,1)};
\node[right] at (.5,0) {\footnotesize  $b_1$};
\draw[mark=*]  plot coordinates {(.5,0)};
\draw  (-.5,2) --(-.5,1) --(.5,0);
\draw (-.5,2) --(.5,1) -- (.5,0);
\end{tikzpicture}
\qquad \qquad 
\begin{tikzpicture}[scale=0.8]
\node[left] at (-.5,5) {\footnotesize  $1$};
\draw[mark=*]  plot coordinates {(-.5,5)};
\node[left] at (-.5,4) {\footnotesize  $a_5$};
\draw[mark=square*]  plot coordinates {(-.5,4)};
\node[right] at (.5,4) {\footnotesize  $b_5$};
\draw[mark=square*]  plot coordinates {(.5,4)};
\node[right] at (.5,3) {\footnotesize  $b_4$};
\draw[mark=*]  plot coordinates {(.5,3)};
\draw  (-.5,5) --(-.5,4) --(.5,3);
\draw (-.5,5) --(.5,4) -- (.5,3)--(.5,2);
\draw  (-.5,4)-- (.5,4);
\node[right] at (.5,2) {\footnotesize  $b_3$};
\draw[mark=*]  plot coordinates {(.5,2)};
\node[left] at (-.5,1) {\footnotesize  $a_2$};
\draw[mark=square*]  plot coordinates {(-.5,1)};
\node[right] at (.5,1) {\footnotesize  $b_2$};
\draw[mark=square*]  plot coordinates {(.5,1)};
\node[right] at (.5,0) {\footnotesize  $b_1$};
\draw[mark=*]  plot coordinates {(.5,0)};
\draw  (.5,2) --(-.5,1) --(.5,0);
\draw (.5,2) --(.5,1) -- (.5,0);
\end{tikzpicture}
\end{center}
\caption{The enhanced monoidal preorders of three algebras: $\m A$, $\m B$ and $\m C$. \label{f:connectedcomponents}}
\end{figure}

\subsection{Rigidity}\label{s:weaklyinvolutive}
In Section~\ref{s:amalgamation}, we investigate the the amalgamation property for idempotent residuated chains. However, we show that amalgamation fails in this class, and unfortunately even if the V-formation consists of quasi-involutive ones; see Section~\ref{s:APfailschains}. There we identify the problem as: $x^\star$ belonging to a subalgebra without $x$ also belonging to this subalgebra.

 We will show below in Lemma~\ref{l:star-involutive} that, in the context of quasi-involutive idempotent residuated chains, this problem does not occur iff the the residuated chain is \emph{${}^\star$-involutive}: That is, if it satisfies $x^{\star\star}=x$. An enhanced monoidal preorder is called \emph{$^\star$-involutive} if its associated idempotent residuated chain is.

Recall that $\overline{x}$ denotes the layer of $x$, and that the layers form a chain. Given an enhanced monoidal preorder, we consider the preorder $\sqsubseteq_*$ given by $x \sqsubseteq_* y$ iff $x \sqsubseteq y$ or $x,y$ are in the same layer; we write $\downset_*$ and $\upset_*$ for the downward and upward closures of this relation. In the following we will denote by $\langle x \rangle_1$ the subalgebra generated by $x$ and by $\langle x \rangle$ the 1-free subalgebra (subalgebra with respect to all the operations except possibly for $1$) generated by $x$. By Corollary~\ref{c:idempotentchainsoperations} and Lemma~\ref{l:generation}, we get that $\langle x \rangle=\{x^{c_1 \cdots c_n}: n \in \mathbb{N}, c_i \in \{\star, {\leftrightarrow}\}\}$ and $\langle x \rangle_1 =\langle x \rangle \cup \{1\}$.  

\begin{lemma}\label{l: cyclic subalgebras}
Let $\m A$ be an idempotent residuated chain.
\begin{enumerate}
	\item If $b$ is a negative central element of $A$, then $\langle b \rangle \subseteq {\uparrow}_* b$.
	\item If $a$ is a positive central element of $A$, then $\langle a \rangle \subseteq {\downarrow}_* a$.
	\item If $b$ is a negative element, $b < b^{\star \star}$ and  $b <c \leq b^{\star \star}$, then $c$ is central.
	\item If  $x \in \langle x^\star \rangle$ for all $x \in A$, then $\m A$ is ${}^\star$-involutive.
	\item Let $\m P$ be a ${}^\star$-involutive enhanced monoidal preorder. 
	\begin{enumerate}
		\item 	For every strictly negative $b \in P$, there is a layer of $\m P$ directly above the layer of $b$ and $b^\star$ is in that layer. 
		\item Also, for every strictly positive $a \in P$, there is a layer of $\m P$ directly below the layer of $a$ and $a^\star$ is in that layer.
	\end{enumerate}

\end{enumerate}
\end{lemma}
 
\begin{proof}
1. We prove that $\uparrow_* b$ is closed under ${}^\star$ and ${}^{\leftrightarrow}$; since it also contains $b$, the result will follow. Let $c \in {\uparrow}_* b$. Since $c^\leftrightarrow$ is in the same layer as $c$, it follows that $c^\leftrightarrow \in {\uparrow}_* b$. We will now show that also $c^\star \in {\uparrow}_* b$. If $c$ is negative, then  $c \in {\uparrow}_* b$ implies $\overline{b} \leq \overline{c}<\overline{c^\star}$, since $c \sqsubset c^\star$; hence  $c^\star \in {\uparrow}_* b$. If $c$ is positive, then by the definition of $b^\star$ we have
$b \sqsubset b^\star \sqsubseteq_{P^+} c$. So $b \sqsubset c$ and by the definition of $c^\star$ we get $b \sqsubseteq_{P^-} c^\star$; hence $b \sqsubseteq_* c^\star$ and  $c^\star \in {\uparrow}_* b$.

The proof of 2 is analogous.

3.  From $b < b^{\star \star}$ we get that $\overline{b} < \overline{b^{\star \star}} <\overline{b^\star}$. Assume that  $b <c \leq b^{\star \star}$; hence $c$ is negative. If $c$ is not central, then $c \not = c^\leftrightarrow$ and $c^\leftrightarrow$ is positive. Also, $\overline{b} < \overline{c}=\overline{c^\leftrightarrow} \leq \overline{b^{\star \star}}<\overline{b^\star}$. So, 
$b \sqsubset c^\leftrightarrow \sqsubset_{P^+} b^\star$, a contradiction to the definition of $b^\star$.

4. We will prove the contrapositive. If $\m A$  is not ${}^\star$-involutive, then  $x \not = x^{\star \star}$ for some $x$. By Lemma~\ref{c:starand*}(6), $x \leq x^{\star \star}$, so $x < x^{\star \star}$. We will show that this implies $x \not \in \langle x^\star \rangle$. We will do the proof for the case where $x$ is some negative element $b$, as the proof for positive elements is analogous. From  $b < b^{\star \star}$, (3) yields that $b^{\star \star}$ is central. By (1) we get that $\langle b^{\star \star} \rangle \subseteq {\uparrow}_* b^{\star \star}$. Also $\langle b^{\star \star} \rangle=\langle b^{\star} \rangle$, since $({}^\star, {}^\star)$ is a Galois connection and $b^{\star \star \star}=b^{\star}$, so  $\langle b^{\star} \rangle \subseteq {\uparrow}_* b^{\star \star}$.
 At the same time, from $b < b^{\star \star}$, we get  $b \not \in {\uparrow}_* b^{\star \star}$. Therefore, $b \not \in \langle b^\star \rangle$. 

5a. Recall that if $b$ is negative, then $\overline{b} < \overline {b^\star}$, since $b \sqsubset b^\star$. If there is a positive $c$ such that $\overline{b} < \overline{c} < \overline {b^\star}$, then 
$b \sqsubset c^\leftrightarrow \sqsubset_{P^+} b^\star$, a contradiction to the definition of $b^\star$. If there is a negative $c$ such that $\overline{b} < \overline{c} < \overline {b^\star}$, then 
$b \sqsubset_{P^-} c \sqsubset b^\star$, so $b \sqsubset_{P^-} c \sqsubseteq b^{\star \star}$, by the definition of $b^{\star \star}$; hence $b < b^{\star \star}$,
a contradiction to $^\star$-involutivity. Therefore, there is no level between $\overline{b}$ and $\overline {b^\star}$.
The proof of 5b is analogous.
\end{proof}

\begin{lemma}\label{l:star-involutive}\
 \item For a quasi-involutive idempotent residuated chain $\m B$ the following are equivalent:
\begin{enumerate}
	\item $\m B$ is $^\star$-involutive 
		\item For each $x \in B$, $x \in \langle x^\star \rangle_1$.
	\item For each $x \in B$ and each subalgebra $\m A$ of $\m B$, $x^\star \in A$ implies $x \in A$.
\end{enumerate}
\end{lemma}

\begin{proof}
 That (1) implies (2) follows from the fact that $x=x^{\star \star} \in  \langle x^\star \rangle_1$. For (2) implies (1), assume  $x \in \langle x^\star \rangle_1$, for all $x$. If $x=1$, then actually $1 \in \langle 1^\star \rangle$. Also, if $x \not = 1$, then from $x \in \langle x^\star \rangle_1= \langle x^\star \rangle \cup \{1\}$, we get  $x \in \langle x^\star \rangle$. Therefore,  $x \in \langle x^\star \rangle$, for all $x$. By Lemma~\ref{l: cyclic subalgebras}(4) we get (1).

Condition (3) implies (2) by taking $A= \langle x^\star \rangle_1$. To show that (2) implies (3), let  $x^\star \in A$ for some subalgebra $\m A$. Then $\langle x^\star \rangle_1 \subseteq A$, and by (2), $x \in \langle x^\star \rangle_1 \subseteq A$.
\end{proof}
 
Since our strategy for amalgamation of conic idempotent residuated lattices is to amalgamate their quasi-involutive skeletons, we need to identify the ones whose quasi-involutive skeletons are ${}^\star$-involutive. This is done in the next definition and lemma. An semiconic idempotent  residuated lattice is called \emph{rigid} if it satisfies the identities $x^r=x^{r \star \star}$ and $x^\ell=x^{\ell \star \star}$. 

 \begin{lemma}\label{l:weaklyinvolutive}\
\begin{enumerate}
	\item A  conic idempotent residuated lattice is rigid iff its quasi-involutive skeleton is rigid iff its quasi-involutive skeleton is ${}^\star$-involutive.
	\item If an idempotent residuated chain is ${}^\star$-involutive, then it is also quasi-involutive.
\end{enumerate}
\end{lemma}

\begin{proof}
 1. The rigidity equations for a conic idempotent residuated lattice $\m A$ stipulate that every element of the quasi-involutive skeleton $A^i$ is fixed under double star, and this is also what rigidity for its quasi-involutive skeleton $\m A^i$ stipulates.

2. For every element $x$ in a ${}^\star$-involutive idempotent residuated chain $\m A$, we have that $x=x^{\star \star}=(x^\ell \jn x^r)^\ell \jn (x^\ell \jn x^r)^r$. Since, $x^\ell$ and $x^r$ are comparable, we get that $x \in \{x^{\ell r}, x^{r\ell}\}\subseteq A^i$. Hence, $A=A^i$ and $\m A$ is quasi-involutive.
\end{proof}
 
As previously mentioned, we will see that one of the main obstacles to amalgamating conic idempotent residuated lattices occurs when, for some $x$, $x^\star$ is in some subalgebra of the skeleton but $x$ is not. The conic idempotent residuated lattices for which this does not occur are exactly the rigid ones, which by Lemma~\ref{l:weaklyinvolutive}(1) are precisely those whose skeleton is ${}^\star$-involutive. Hence it suffices to focus on ${}^\star$-involutive quasi-involutive idempotent residuated chains, which by Lemma~\ref{l:weaklyinvolutive}(2) are exactly the $^\star$-involutive idempotent residuated chains. 
   
 \subsection{$^\star$-involutive  idempotent residuated chains: the one-generated case}\label{s:one-generated} We will characterize the $^\star$-involutive idempotent residuated chains as nested sums of one-generated ones. First, we provide a description of the one-generated $^\star$-involutive idempotent residuated chains.

Note that one of the most obvious consequences of ${}^\star$-involutivity is that the only $x$ with $x^\star=1$ is $x=1$, since then $x^{\star \star}=1^\star=1 = 1$. In particular this implies that, for non-identity $x$, we cannot have $x^\ell=1$ or $x^r=1$, as that would imply that $x$ is negative, thus that $x^\ell$ and $x^r$ are positive, so $x^\star=x^\ell \mt x^r=1$. Therefore, in the presence of ${}^\star$-involutivity, $1$ is \emph{isolated}: It is not equal to the inverse of a non-identity element. Note that $1$ is isolated in an idempotent residuated chain iff in the associated enhanced monoidal preorder, it is not the case that there exists a strictly negative element in a layer directly below the layer of $1$. So, if there is a layer directly below $1$, it contains a positive element only (or else there is no layer covered by the layer of $1$). 

Therefore, in this setting, not only do we have $\langle x \rangle_1= \langle x \rangle \cup \{1\}$, for every $x$, but we also have  $\langle x \rangle= \langle x \rangle_1 - \{1\}$. Therefore, we characterize the 1-free subalgebras of the form $\langle x \rangle$, for $x$ in a ${}^\star$-involutive idempotent residuated chain. 

We have already seen examples of non-trivial one-generated ${}^\star$-involutive idempotent residuated chains in Example~\ref{e:noncom Sug} and Figure~\ref{f:Sugihara}, and we now also include their bounded variants and their finite variants. We introduce them by means of their enhanced monoidal preorders.

   An enhanced monoidal preorder is called a \emph{vertical crown} iff $1$ is isolated (in particular its size is not 2) and it has at most two non-identity central elements. Figure~\ref{f:v-crowns} shows all the possible shapes that vertical crowns can take. 
	
	     \begin{figure}[htbp]
 \begin{center}
\begin{tikzpicture}[scale=0.8]
\node[left] at (6.5,6) {\footnotesize  $1$};
\draw[mark=*]  plot coordinates {(6.5,6)};
\draw[dotted] (7,4)--(7,4.5);
\node[left] at (6.5,4) {\footnotesize  $a_{i}$};
\draw[mark=square*]  plot coordinates {(6.5,4)};
\node[right] at (7.5,4) {\footnotesize  $b_{i}$};
\draw[mark=square*]  plot coordinates {(7.5,4)};
\node[right] at (7.5,3) {\footnotesize  $b_j$};
\draw[mark=square*]  plot coordinates {(7.5,3)};
\node[left] at (6.5,3) {\footnotesize  $a_j$};
\draw[mark=square*]  plot coordinates {(6.5,3)};
\draw[dashed]  (6.5,4)-- (7.5,4);
\draw[dotted] (7,3) -- (7,4);
\draw[dashed]  (6.5,3)-- (7.5,3);
\draw[dotted] (7,3)--(7,2.5);
\node[right] at (7.5,2) {\phantom{$b_0$}};
\end{tikzpicture} \;
\begin{tikzpicture}[scale=0.8]
\node[left] at (6.5,6) {\footnotesize  $1$};
\draw[mark=*]  plot coordinates {(6.5,6)};
\draw[dotted] (7,4)--(7,4.5);
\node[left] at (6.5,4) {\footnotesize  $a_i$};
\draw[mark=square*]  plot coordinates {(6.5,4)};
\node[right] at (7.5,4) {\footnotesize  $b_i$};
\draw[mark=square*]  plot coordinates {(7.5,4)};
\node[right] at (7.5,3) {\footnotesize  $b_1$};
\draw[mark=square*]  plot coordinates {(7.5,3)};
\node[left] at (6.5,3) {\footnotesize  $a_1$};
\draw[mark=square*]  plot coordinates {(6.5,3)};
\draw[dashed]  (6.5,4)-- (7.5,4);
\draw[dotted] (7,3) -- (7,4);
\draw[dashed]  (6.5,3)-- (7.5,3);
\node[right] at (7.5,2) {\footnotesize  $b_0$};
\draw[mark=*]  plot coordinates {(7.5,2)};
\draw (7.5,3)--(7.5,2);
\draw[dashed]  (6.5,4)-- (7.5,4);
\draw[dotted] (7,3) -- (7,4);
\draw[dashed]  (6.5,3)-- (7.5,3);
\draw (6.5,3)--(7.5,2);
\end{tikzpicture}  \;
\begin{tikzpicture}[scale=0.8]
\node[left] at (6.5,6) {\footnotesize  $1$};
\draw[mark=*]  plot coordinates {(6.5,6)};
\draw  (6.5,6)-- (6.5,5);
\node[left] at (6.5,5) {\footnotesize  $a_0$};
\draw[mark=*]  plot coordinates {(6.5,5)};
\node[left] at (6.5,4) {\footnotesize  $a_1$};
\draw[mark=square*]  plot coordinates {(6.5,4)};
\node[right] at (7.5,4) {\footnotesize  $b_1$};
\draw[mark=square*]  plot coordinates {(7.5,4)};
\node[right] at (7.5,3) {\footnotesize  $b_i$};
\draw[mark=square*]  plot coordinates {(7.5,3)};
\node[left] at (6.5,3) {\footnotesize  $a_i$};
\draw[mark=square*]  plot coordinates {(6.5,3)};
\draw   (6.5,5) --(6.5,4) ;
\draw (6.5,5) --(7.5,4);
\draw[dashed]  (6.5,4)-- (7.5,4);
\draw[dotted] (7,3) -- (7,4);
\draw[dashed]  (6.5,3)-- (7.5,3);
\draw[dotted] (7,3)--(7,2.5);
\node[right] at (7.5,2) {\phantom{$b_0$}};
\end{tikzpicture} \;
\begin{tikzpicture}[scale=0.8]
\node[left] at (6.5,6) {\footnotesize  $1$};
\draw[mark=*]  plot coordinates {(6.5,6)};
\draw  (6.5,6)-- (6.5,5);
\node[left] at (6.5,5) {\footnotesize  $a_{n+1}$};
\draw[mark=*]  plot coordinates {(6.5,5)};
\node[left] at (6.5,4) {\footnotesize  $a_n$};
\draw[mark=square*]  plot coordinates {(6.5,4)};
\node[right] at (7.5,4) {\footnotesize  $b_n$};
\draw[mark=square*]  plot coordinates {(7.5,4)};
\node[right] at (7.5,3) {\footnotesize  $b_1$};
\draw[mark=square*]  plot coordinates {(7.5,3)};
\node[left] at (6.5,3) {\footnotesize  $a_1$};
\draw[mark=square*]  plot coordinates {(6.5,3)};
\node[right] at (7.5,2) {\footnotesize  $b_0$};
\draw[mark=*]  plot coordinates {(7.5,2)};
\draw   (6.5,5) --(6.5,4) ;
\draw (6.5,5) --(7.5,4);
\draw (7.5,3)--(7.5,2);
\draw[dashed]  (6.5,4)-- (7.5,4);
\draw[dotted] (7,3) -- (7,4);
\draw[dashed]  (6.5,3)-- (7.5,3);
\draw (6.5,3)--(7.5,2);
\end{tikzpicture}  
\qquad
\begin{tikzpicture}[scale=0.5]
\draw[dotted] (3,6) --(3,7); 
\node[left] at (3,6) {\footnotesize  $a_i$};
\draw[mark=square*]  plot coordinates {(3,6)};
\draw (3,5) --(3,6); 
\node[left] at (3,5) {\footnotesize  $a_{i+1}$};
\draw[mark=square*]  plot coordinates {(3,5)};
\draw[dotted] (3,4) --(3,5); 
\node[left] at (3,3) {\footnotesize  $1$};
\draw[mark=*]  plot coordinates {(3,3)};
\draw[dotted] (3,1) --(3,2); 
\node[right] at (3,1) {\footnotesize  $b_{i+1}$};
\draw[mark=square*]  plot coordinates {(3,1)};
\draw (3,0) --(3,1); 
\node[right] at (3,0) {\footnotesize  $b_i$};
\draw[mark=square*]  plot coordinates {(3,0)};
\draw[dotted] (3,-1) --(3,0); 
\end{tikzpicture}
\end{center}
  \caption{Vertical crowns	(dashed lines may be independently present or absent)\label{f:v-crowns}}
\end{figure}
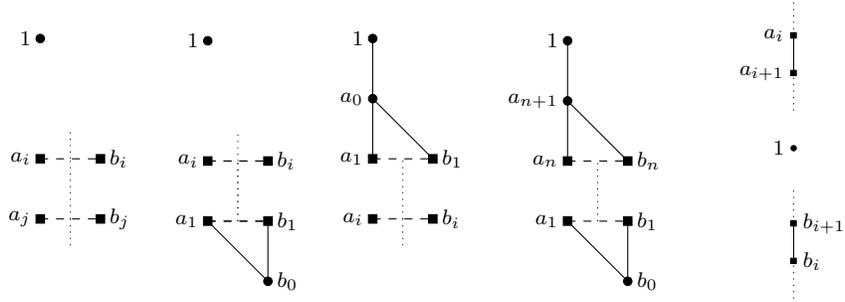
	
	To explain the structures in the picture more precisely, recall the ordinal sum $\bigoplus_{i\in I}\m P_i$ of (pairwise disjoint) (pre)orders $\m P_i=(P_i, \leq_{P_i})$ over a chain $(I, \leq_I)$: It is based on their disjoint union $\bigcup_{i\in I} P_i$ with ordering given by $x \leq y$ iff ($x\in P_i, y\in P_j$ and $i<_I j$) or  ($x, y\in P_i$ and $x \leq_{P_i} y$). We naturally extend the notion of the ordinal sum to \emph{enhanced preorders} (i.e., preorders $(P, \leq)$ enhanced with two subsets $P^+$ and $P^-$) in the natural way: For $\m P=\bigoplus_{i\in I} \m P_i$, we define $P^+:= \bigcup_{i\in I} P^+_i$ and $P^-:= \bigcup_{i\in I} P^-_i$. Examples of enhanced preorders are:
	\begin{itemize}
		\item 	$\m 1_L$: a singleton with $(1_L)^+=1_L$ and $(1_L)^-=\emptyset$, 
	\item $\m 1_R$: a singleton with $(1_R)^+=\emptyset$ and $(1_R)^-=1_R$,
	\item $\m 2_L$: a doubleton $\{a,b\}$ with $(2_L)^+=\{a\}$, $(2_L)^-=\{b\}$, $a \leq b$ and $b \leq a$,
	\item $\m 2_R$: a doubleton $\{a,b\}$ with $(2_R)^+=\{a\}$, $(2_R)^-=\{b\}$, $a \not \leq b$ and $b \not \leq a$. 
	\end{itemize}

	A \emph{vertical crown} is the enhanced monoidal preorder expanding (by ${}^\star$ and $1$, in the unique and obvious way) an enhanced preorder of one of the following types:
	\begin{enumerate}
		\item $\m P_{\mathbb{Z}, L}=(\bigoplus_{i\in I} \m P_i) \oplus \m 1_L$, where $I=\mathbb{Z}$,   
		\item $\m P_{\mathbb{N}, L}=\m 1_R \oplus (\bigoplus_{i\in I} \m P_i) \oplus \m 1_L$, where $I=\mathbb{N}$, 
		\item $\m P_{\mathbb{N^\partial}, L}= (\bigoplus_{i\in I} \m P_i) \oplus  \m 1_L\oplus \m 1_L$, where $I=\mathbb{N^\partial}$, 
				\item $\m P_{n, L}= \m 1_R \oplus (\bigoplus_{i\in I} \m P_i) \oplus  \m 1_L\oplus \m 1_L$, where $I=\{1,2, \ldots, n\}$.
	\end{enumerate}
In all cases, $L \subseteq I$ and $\m P_i= \m 2_L$ for $i \in L$, while $\m P_i= \m 2_R$ for $i \not \in L$. 
	In the case of $\m P_{n, L}$, we allow $n=0$, hence $I=\emptyset$, which yields the 3-element Sugihara monoid. Since these enhanced preorders have a unique expansion to enhanced monoidal preorders, where $1$ is their top element and $^\star$ has the usual definition, we use the same names for the corresponding enhanced monoidal preorders. 
 An idempotent residuated chain whose enhanced monoidal preorder is a vertical crown is called \emph{crownian}. A subset $Q$ of an enhanced monoidal preorder $\m P$ is called \emph{layer-convex} if $q_1, q_2 \in Q$ and $\overline{q_1} \leq \overline{p} \leq \overline{q_2}$, then $p \in Q$. 
 
\begin{lemma}\label{l: level-convex}\
\begin{enumerate}
\item  For a non-identity element $x$ in a ${}^\star$-involutive enhanced monoidal preorder, $\langle x \rangle$ forms a layer-convex subset. 
	\item A non-trivial idempotent residuated chain is one-generated and ${}^\star$-involutive iff it is crownian. 
\end{enumerate}

\end{lemma}

\begin{proof}
1. Recall that if $x$ is a non-identity element in an idempotent residuated chain, then 
 $\langle x \rangle=\{x^{c_1 \cdots c_n}: n \in \mathbb{N}, c_i \in \{\star, {\leftrightarrow}\}\}$; also by definition $(x^{\leftrightarrow})^\leftrightarrow=x$. If we assume ${}^\star$-involutivity, we also have $x^{\star \star}=x$, so the $ c_i \in \{\star, {\leftrightarrow}\}$ can be assumed to alternate. Therefore, $\langle x \rangle=\{a_n: n \in \mathbb{Z}\} \cup\{b_n: n \in \mathbb{Z}\}=\{\ldots, a_{-1}, b_{-1}, a_0, b_0, a_1, b_1, \ldots\}$, where  $a_0=x$, $b_0=a_0^\lra$, and for all $n \in \mathbb{N}$,
\begin{itemize}
	\item $a_{n+1}=b_n^\star$ and $b_{n}=a_n^\lra$,
	\item $b_{-(n+1)}=a_{-n}^\star$ and $a_{-n}=b_{-n}^\lra$.
\end{itemize}
In other words we have 

\begin{equation}
\ldots a_{-1}\stackrel{\lra}{=} b_{-1}\stackrel{\star}{=} a_0 \stackrel{\lra}{=} b_0 \stackrel{\star}{=} a_1\stackrel{\lra}{=} b_1 \ldots 
\label{eq:}\tag{$\lra\star$}
\end{equation}
 where $u \stackrel{\star}{=} v$ denotes ($u^\star=v$ and $v^\star=u$) and $u \stackrel{\lra}{=} v$ denotes ($u^\lra=v$ and $v^\lra=u$). Note that $a_i$ and $b_i$ are in the same layer and, by Lemma~\ref{l: cyclic subalgebras}(5), $b_i$ and $a_{i+1}$ are in adjacent layers. Therefore,  $\langle x \rangle$ is layer-convex.

2. Recall that every one-generated idempotent residuated chain is of the form $\langle x \rangle_1 =\langle x \rangle \cup \{1\}$; since the algebra is non-trivial, we will assume $x \not = 1$ below.

If none of the elements of $\langle x \rangle$ is central, then in particular $y \not = y^\lra$. Note that ${}^\star$ always moves to the next layer, and actually in a NW-SE diagonal (the level directly above when applied to a negative element, and the level directly below when applied to a positive element). Moreover, $^\lra$ moves to the other element of the same layer. It follows that all of these elements are distinct, they form a zigzag pattern (with horizontal and NW-SE diagonal lines), and $\overline{a_i}=\overline{b_i} \prec \overline{a_j}=\overline{b_j}$ for all $i<j$ in $\mathbb{Z}$. Therefore, the enhanced monoidal preorder is $\m P_{\mathbb{Z},L}$, for some $L \subseteq \mathbb{Z}$. 

Now assume that one of the strictly negative elements of $\langle x \rangle$ is central. By possible renaming, we may assume that this element is $b_0$. Then $a_0=b_0^\lra=b_0$. Also, by Lemma~\ref{l: cyclic subalgebras}(1), $\langle b_0 \rangle \subseteq {\uparrow}_* b_0$, and to be more precise because $a_0=b_0$ and because of Equation~(\ref{eq:}), we have $\{\ldots, a_{-1}, b_{-1}, a_0\} \subseteq \{b_0, a_1, b_1, \ldots\}$. Thus $\langle x \rangle=\{b_0, a_1, b_1, \ldots\}$. Now, if furthermore no strictly positive element of $\langle x \rangle$ is central, then in particular for every  positive element $a \in \langle x \rangle$, the element $a^\lra$ is different from $a$ (and negative). Also, $\overline{a_i}=\overline{b_i} \prec \overline{a_j}=\overline{b_j}$ for all $i<j$ in $\mathbb{N}$. Therefore, the enhanced monoidal preorder is $\m P_{\mathbb{N},L}$, for some $L \subseteq \mathbb{N}$. On the other hand, if furthermore  there is a strictly positive element of $\langle x \rangle=\{b_0, a_1, b_1, \ldots\}$ that is central, say $a_{n+1}$, then $b_{n+1}=a_{n+1}^\lra=a_{n+1}$ and because of Equation~(\ref{eq:}), we have $\{b_{n+1}, a_{n+2}, b_{n+2}, \ldots \} \subseteq \{ \ldots a_n, b_n, a_{n+1}\} \subseteq \{b_0, a_1, b_1, \ldots a_n, b_n, a_{n+1}\}$, thus $\langle x \rangle=\{b_0, a_1, b_1, \ldots a_n, b_n, a_{n+1}\}$.  Therefore, in this case the enhanced monoidal preorder is $\m P_{n,L}$, for some $L \subseteq \{1, 2, \ldots, n\}$.

Dually to one of the situations above, if there is a strictly positive element of $\langle x \rangle$ that is central, but no strictly negative element is central, then the enhanced monoidal preorder is $\m P_{\mathbb{N^\partial},L}$, for some $L \subseteq \mathbb{N}$.

Conversely, all vertical crowns are $^\star$-involutive enhanced monoidal preorders. Also they are generated by any one of their elements.
\end{proof}

\subsection{Nested sums and $^\star$-involutive  idempotent residuated chains}\label{nestedsums}
Having described the one-generated subalgebras of  $^\star$-involutive  idempotent residuated chains, we now show what arbitrary $^\star$-involutive idempotent residuated chains look like. To do so, we will make use of the nested sum construction.

Loosely speaking, given two structures $\m K$ and $\m L$ over a language with a constant $1$, this construction is performed by replacing $1_{\m K}$ in $\m K \; (=\m K[1_{\m K}])$ by $\m L$ to obtain  their nested sum $\m K[\m L]$. We will also denote $\m K[\m L]$ by $\m K \boxplus \m L$, as this notation allows iterations of this construction: $\boxplus_{i \in I} \m K_i$. We will make use of this construction where the structures are either residuated lattices or enhanced monoidal preorders.

Nested sums of residuated lattices are defined in \cite{Ga2005}. Here we will only describe how nested sums specialize to the conic idempotent case. Assume that $\m A_i$, for $i \in I$, are conic  residuated lattices,  where $I$ is a totally ordered index set. Also, assume that for all $i \in I$ except possibly for the top element of $I$, if it exists, we have $a^\ell \not =1_{\m A_i}$, $a^r \not =1_{\m A_i}$, $a \jn b \not = 1$ and $a \mt b \not = 1$, for all non-identity elements $a, b \in {\m A_i}$.  Their \emph{nested sum} is an algebra $\boxplus_{i \in I} \m A_i$ whose underlying set is the union $\{1\}\cup \bigcup_{i \in I} (A_i \setminus \{1_{\m A_i}\})$ (we identify all the identity elements). The order and operations extend the operations on the $\m A_i$'s, which become subalgebras of the nested sum, by the following additional clauses, where $a_i \in A_i$ and $a_j \in A_j$ are non-identity elements and $i<_I j$: $a_i < a_j$ iff $a_i < 1_{\m A_i}$, and $a_j < a_i$ iff $1_{\m A_i} < a_i$; also, 
 \begin{center}
 $a_i \bullet a_j=a_i \bullet_{\m A_i} 1_{\m A_i}$  and $a_j \bullet a_i=1_{\m A_i} \bullet_{\m A_i} a_i$,
 \end{center}
  where $\bullet$ ranges over multiplication and the divisions. In particular, focusing on the lattice operations, for $i <_I j$, $\m A_j$ is nested inside $\m A_i$ at the location of the identity element. It is shown in \cite{Ga2005} that the nested sum is a residuated lattice. It is obvious that conicity, linearity, and idempotency are preserved. 
	
Note that every non-trivial Sugihara monoid is a nested sum of copies of the 3-element Sugihara monoid. In particular, these copies are subalgebras.

\medskip

We now define nested sums of enhanced monoidal preorders $\m P_i$, over a chain $I$. We again assume that  for all $i \in I$ except possibly for the top element of $I$, if it exists, we have $a^\star \not =1_{\m P_i}$ for every non-identity element $a \in {\m P_i}$. The \emph{nested sum} is the enhanced monoidal preorder $\boxplus_{i \in I} \m P_i$, whose underlying set is the union $\{1\}\cup \bigcup_{i \in I} (P_i \setminus \{1_{\m P_i}\})$. The operation ${}^\star$ is defined as in each $\m P_i$. For the positive and negative elements, we defined $P^+:= \bigcup_{i \in I} P^+_i$ and $P^-:= \bigcup_{i \in I} P^-_i$. The preorder is defined to extend the preorders on the structures $\m P_i$ by setting $1$ to be the largest element, and for non-identity $a_i \in P_i$ and $a_j \in P_j$, stipulating that $a_i \sqsubset a_j$ iff $i < j$. In essence, the preorder is the ordinal sum $ \bigoplus_{i \in I} (P_i \setminus \{1_{\m P_i}\}) \oplus \{1\}$.

 Nested sums for integral residuated chains have been considered in the literature. Confusingly, they are often called ordinal sums \cite{AM2003,HNP2007}, in conflict with the use of the term for (pre)ordered sets. Also, \cite{Ga2005} refers to the generalization beyond the integral case as generalized ordinal sums. Here we try to fix the terminological clash caused by the overload of the term ordinal sum by introducing the term nested sums.

The following is an immediate consequence of the work of this section.

\begin{lemma}\label{l: nested IRC and EMP}
 Nested sums of idempotent residuated chains correspond bijectively to nested sums of the corresponding enhanced monoidal preorders. 
\end{lemma}

    \begin{lemma}\label{l:crowns}
    The $^\star$-involutive idempotent residuated chains are exactly the nested sums of crownians, i.e., those whose enhanced monoidal preorders are crowns. Also, one-generated subalgebras intersect only at $\{1\}$ and they are generated by any of their non-identity elements.
\end{lemma}

\begin{proof}  By Lemma~\ref{l: nested IRC and EMP}, it suffices to show that an enhanced monoidal preorder is $^\star$-involutive iff it is a nested sum of vertical crowns. By Lemma~\ref{l: level-convex}(1), for every non-identity element $x$ in a $^\star$-involutive enhanced monoidal preorder, $\langle x \rangle$ is layer-convex, and together with $1$ forms an enhanced monoidal preorder (the one corresponding to the subalgebra  $\langle x \rangle_1$). As a result, the enhanced monoidal preorder is an ordinal sum of the various  $\langle x \rangle$, hence the nested sum of the various $\langle x \rangle_1$.
The converse follows from the fact that vertical crowns are $^\star$-involutive and that the nested sum of $^\star$-involutive enhanced monoidal preorders remains $^\star$-involutive. 
\end{proof}

\section{Amalgamation}\label{s:amalgamation}

We now have all the ingredients for proving the strong amalgamation property. Our approach is to first prove that the class of conic algebras has the amalgamation property, and then extend this result to the semiconic case. Unfortunately, the amalgamation property fails for the the whole class of conic idempotent residuated lattices; as we have already mentioned, it even fails for the class of idempotent residuated chains, as we show in Section~\ref{s:APfailschains}. However, the analysis of the structure of conic idempotent residuated lattices---and in particular their decomposition into nested sums and their subalgebra generation---allows us to identify the two obstacles to amalgamation. By stipulating the conditions of being rigid and conjunctive (i.e., all prelattices appearing in the decomposition are lattices; see Section~\ref{s:failure for rigid conic}), we are able to prove the strong amalgamation for this class of conic idempotent residuated lattices and then extend the result to the semiconic variety this class generates.

\subsection{Failure of amalgamation for idempotent residuated chains}\label{s:APfailschains}
We consider the idempotent residuated chains $\m B$ and $\m C$ given in terms of their enhanced monoidal preorder in Figure~\ref{f:APfails}, and their common subalgebra $\m A$, where $A=\{1\}$. We will show that this witnesses the failure of the amalgamation property in conic idempotent residuated lattices even when the V-formation consists of idempotent residuated chains.

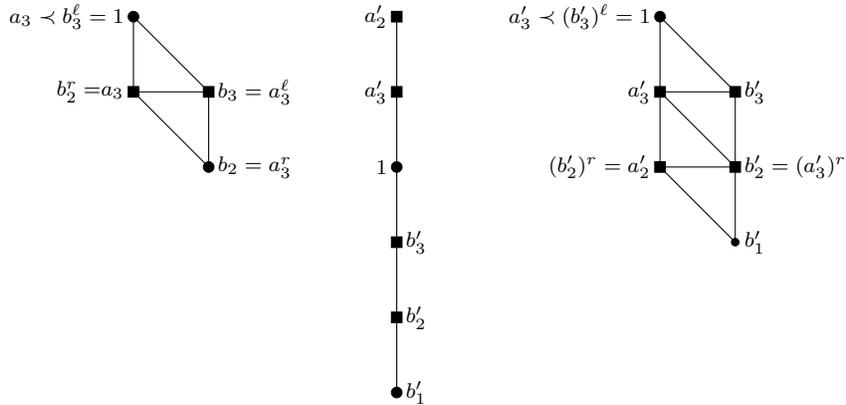
\begin{figure}[htbp]
\begin{center}
\begin{tikzpicture}[scale=1]
\node[left] at (-.5,5) {\footnotesize  $a_3 \prec b_3^\ell=1$};
\draw[mark=*]  plot coordinates {(-.5,5)};
\node[left] at (-.5,4) {\footnotesize  ${b_2^r=}a_3$};
\draw[mark=square*]  plot coordinates {(-.5,4)};
\node[right] at (.5,4) {\footnotesize  $b_3=a_3^\ell$};
\draw[mark=square*]  plot coordinates {(.5,4)};
\node[right] at (.5,3) {\footnotesize  $b_2=a_3^r$};
\draw[mark=*]  plot coordinates {(.5,3)};
\draw  (-.5,5) --(-.5,4) --(.5,3);
\draw (-.5,5) --(.5,4) -- (.5,3);
\draw  (-.5,4)-- (.5,4);
\qquad
\node[left] at (6.5,5) {\footnotesize  $a'_3 \prec (b'_3)^\ell=1$};
\draw[mark=*]  plot coordinates {(6.5,5)};
\node[left] at (6.5,4) {\footnotesize  $a'_3$};
\draw[mark=square*]  plot coordinates {(6.5,4)};
\node[right] at (7.5,4) {\footnotesize  $b'_3$};
\draw[mark=square*]  plot coordinates {(7.5,4)};
\node[right] at (7.5,3) {\footnotesize  $b'_2=(a'_3)^r$};
\draw[mark=square*]  plot coordinates {(7.5,3)};
\node[left] at (6.5,3) {\footnotesize  ${(b'_2)^r}=a'_2$};
\draw[mark=square*]  plot coordinates {(6.5,3)};
\node[right] at (7.5,2) {\footnotesize  $b'_1$};
\draw[fill] (7.5,2)  circle [radius=0.05];
\draw   (6.5,5) --(6.5,4) --(7.5,3);
\draw (6.5,5) --(7.5,4) -- (7.5,3)--(7.5,2);
\draw  (6.5,4)-- (7.5,4);
\draw  (6.5,3)-- (7.5,3);
\draw (6.5,4)--(6.5,3)--(7.5,2);
\qquad
\node[left] at (3,5) {\footnotesize  $a'_2$};
\draw[mark=square*]  plot coordinates {(3,5)};
\node[left] at (3,4) {\footnotesize  $a'_3$};
\draw[mark=square*]  plot coordinates {(3,4)};
\node[left] at (3,3) {\footnotesize  $1$};
\draw[mark=*]  plot coordinates {(3,3)};
\node[right] at (3,2) {\footnotesize  $b'_3$};
\draw[mark=square*]  plot coordinates {(3,2)};
\node[right] at (3,1) {\footnotesize  $b'_2$};
\draw[mark=square*]  plot coordinates {(3,1)};
\node[right] at (3,0) {\footnotesize  $b'_1$};
\draw[mark=*]  plot coordinates {(3,0)};
\draw (3,0)--(3,1)--(3,2)--(3,3)--(3,4)--(3,5);
\end{tikzpicture}
\end{center}
\caption{Failure of amalgamation for conic; $\m B$ is on the left, $\m C$ is on the middle and the right and $A=\{1\}$. All algebras are chains and connected; they are not rigid. \label{f:APfails}}
\end{figure}

If there would be an idempotent residuated chain $\m D$ that would serve as an amalgam for this V-formation, then $\m B$ and $\m C$ could be identified with subalgebras of $\m D$, where the intersection of  $B$ and $C$ could be larger than $A$. We will argue in the algebra $\m D$. Since $b_3^\ell=1$  in $\m B$ and thus also in $\m D$, the element $a_3$ has to be an upper cover of $1$ in $\m D$, by Corollary~\ref{c:RandL}. Likewise, since  $(b'_3)^\ell=1$, the element $a'_3$ is an upper cover of $1$ in $\m D$; hence $a_3=a'_3$ in $\m D$. Therefore,
$b_2=a_3^r=(a'_3)^r=b'_2$ in $\m D$, hence $a'_2=(b'_2)^r=b_2^r=a_3=a'_3$ in $\m D$, which is a contradiction to the fact that $\m C$ is a subalgebra of $\m D$ and in $\m C$ the elements $a'_2$ and $a'_3$ are different.

More generally, if an idempotent residuated chain $\m B$ has a subalgebra $\m A$ and some element outside of $A$ (in the above example this was $b_3$) whose ${}^\star$  lies inside $\m A$ (this was $b_3^\star=b_3^\ell=1$), then a similar failure of amalgamation can be produced.
The condition that ensures precisely that this issue does not come up is rigidity, as we proved in Section~\ref{s:weaklyinvolutive}. Therefore, we will restrict our attention to rigid conic idempotent residuated lattices.

\subsection{Failure of amalgamation for semilinear idempotent and semiconic idempotent residuated lattices}\label{s:APfailssemilinear} It is conceivable that a V-formation of idempotent residuated chains does not have a totally ordered amalgam, but it does have an amalgam in the variety of semilinear idempotent residuated lattices.
We now show that the amalgamation property fails for the variety of semilinear idempotent residuated lattices. 

The \emph{one sided amalgamation property} (1AP) \cite{FM2022} is defined in the same way as the amalgamation property, but by removing the demand that $g_{\m B}: \m B \ra \m D$ is an embedding and only stipulating that it is a homomorphism, while still insisting that $g_{\m C}$ is an embedding; we refer to $\m D$ as a \emph{1-amalgam} of the associated V-formation. We prove that in the setting of residuated lattices 1-amalgamation and amalgamation are more closely related than in arbitrary algebras.

\begin{lemma}\label{l:AP1APmonolith}
If $({\m A},{\m B},{\m C},\fb,\fc)$ is a V-formation of residuated lattices and $\m B$ is subdirectly irreducible with monolith $\m A$, then every residuated lattice that is a 1-amalgam of the V-formation is actually an amalgam. 
\end{lemma}

\begin{proof}
Assume that the residuated lattice $\m D$ is a 1-amalgam for the V-formation. By the commutativity of the diagram we obtain that $g_{\m B}$ has to be injective on $\m A$. 

If $\m A$ is the monolith of $\m B$, then $A$ is the smallest non-trivial convex normal subalgebra of $\m B$ and any other such must contain $A$. In particular, since the kernel of $g_{\m B}$ is a convex normal subalgebra of $\m B$, it must contain $A$ or be trivial. If it contains $A$, then all elements of $A$ are mapped to $1$ under $g_{\m B}$, a contradiction to the injectivity of $g_{\m B}$ on $\m A$. Therefore, the kernel of $g_{\m B}$ is trivial, hence $g_{\m B}$ is injective. Consequently, $\m D$ is an amalgam of the V-formation.
\end{proof}

\begin{theorem} \label{t:1APfails}
The 1AP (hence also the AP) fails for the class of idempotent residuated chains. Also, it fails for the class of finitely subdirectly irreducible conic idempotent residuated lattices.
\end{theorem}

\begin{proof}
We consider the idempotent residuated chains $\m B$ and $\m C$ given in terms of their enhanced monoidal preorder in Figure~\ref{f:APfailsVar}, and their common subalgebra $\m A$, where $A=\{b, a, 1\}$. 
Assume there is a residuated lattice $\m D$ that would serve as a 1-amalgam for this V-formation. Note that by the characterization given in Lemma~\ref{l:congruencefilters2}, $A$ is the smallest non-trivial convex normal subalgebra of $\m B$ and any other such must contain $A$. 
By Lemma~\ref{l:AP1APmonolith}, $\m D$ is actually an amalgam of the V-formation and we can identify $\m B$ and $\m C$ with a subalgebras of $\m D$, and $g_{\m B}$ and $g_{\m C}$ can be taken to be the inlcusion maps. We will argue in the algebra $\m D$. 

First assume that $\m D$ is an idempotent residuated chain. Note that $b_B^\star = a$ and $b_C^\star = a$, so both $a_B$ and $a_C$ are a covers of $a$ in $\m D$. Therefore, $a_B=a_C$, a contradiction on how $a_B$,$b_B$ and $a_C$,$b_C$ multiply. Thus the 1AP fails for the class of idempotent residuated chains.

By the characterization in Lemma~\ref{l:FSI}(4), the algebras $\m A$, $\m B$ and $\m C$ are finitely subdirectly irreducible conic idempotent residuated lattices. We now assume that $\m D$ is also a  finitely subdirectly irreducible conic idempotent residuated lattice. Also, note the algebras $\m A$, $\m B$ and $\m C$ are quasi-involutive, since every element in them is the (left or right) inverse of some element. Therefore, the images of these algebras in $\m D$ (via  $g_{\m B}$ and $g_{\m C}$) are also quasi-involutive and thus are contained in the quasi-involutive skeleton $\m S$ of $\m D$. Thus, $\m S$ is also an amalgam of the V-formation, which happens to be an idempotent residuated chain. This contradicts the fact established above that this particular V-formation does not have an amalgam among idempotent residuated chains. This establishes that the 1AP fails for the class of finitely subdirectly irreducible conic idempotent residuated lattices. 
\end{proof}

\begin{corollary}
The amalgamation property fails for the variety of semilinear idempotent residuated lattices. Also, it fails for the variety of semiconic idempotent residuated lattices.
\end{corollary}

\begin{proof}
In \cite{FM2022} it is shown that if a variety $\mathcal{V}$ has the CEP and the class $\mathcal{V}_{FSI}$ of its finitely subdirectly irreducibles is closed under subalgebras, then $\mathcal{V}$ has the AP iff $\mathcal{V}_{FSI}$ has the 1AP. Taking $\mathcal{V}$ to be the variety of semilinear idempotent residuated lattices, Lemma~\ref{c:CEP}(2) shows that $\mathcal{V}$ has the CEP, while  Lemma~\ref{c:FSI}(5) shows that $\mathcal{V}_{FSI}$ is exactly the class of idempotent residuated chains, which is clearly closed under subalgebras. Since  $\mathcal{V}_{FSI}$ fails the 1AP by Theorem~\ref{t:1APfails}, we obtain the failure of the AP for  $\mathcal{V}$.

Likewise, if we take  $\mathcal{V}$ to be the variety of  semiconic idempotent residuated lattices, Lemma~\ref{c:CEP}(2) shows that $\mathcal{V}$ has the CEP, while  Lemma~\ref{l:FSI}(4) ensures that $\mathcal{V}_{FSI}$ is closed under subalgebras. Since  $\mathcal{V}_{FSI}$ fails the 1AP by Theorem~\ref{t:1APfails}, we obtain the failure of the AP for  $\mathcal{V}$.
\end{proof}

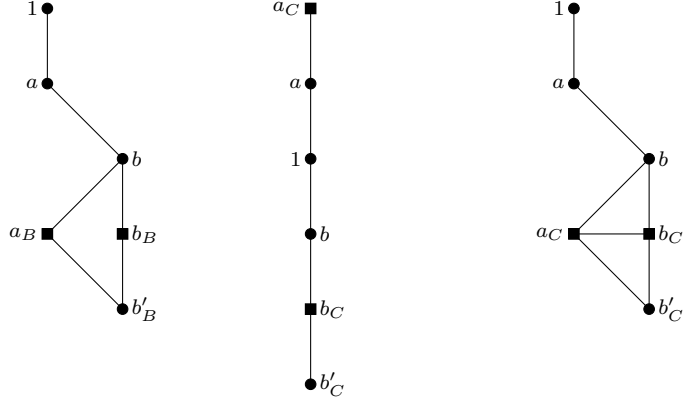
\begin{figure}[htbp]
\begin{center}
\begin{tikzpicture}[scale=1]
\node[left] at (-.5,5) {\footnotesize  $1$};
\draw[mark=*]  plot coordinates {(-.5,5)};
\node[left] at (-.5,4) {\footnotesize  $a$};
\draw[mark=*]  plot coordinates {(-.5,4)};

\node[right] at (.5,3) {\footnotesize  $b$};
\draw[mark=*]  plot coordinates {(.5,3)};
\node[right] at (.5,2) {\footnotesize  $b_B$};
\draw[mark=square*]  plot coordinates {(.5,2)};
\node[left] at (-.5,2) {\footnotesize  $a_B$};
\draw[mark=square*]  plot coordinates {(-.5,2)};
\node[right] at (.5,1) {\footnotesize  $b'_B$};
\draw[mark=*]  plot coordinates {(.5,1)};

\draw   (-.5,5) --(-.5,4) --(.5,3);
\draw (-.5,2)--(.5,3) -- (.5,2)--(.5,1);

\draw (-.5,2)--(.5,1);
\qquad
\node[left] at (6.5,5) {\footnotesize  $1$};
\draw[mark=*]  plot coordinates {(6.5,5)};
\node[left] at (6.5,4) {\footnotesize  $a$};
\draw[mark=*]  plot coordinates {(6.5,4)};

\node[right] at (7.5,3) {\footnotesize  $b$};
\draw[mark=*]  plot coordinates {(7.5,3)};
\node[right] at (7.5,2) {\footnotesize  $b_C$};
\draw[mark=square*]  plot coordinates {(7.5,2)};
\node[left] at (6.5,2) {\footnotesize  $a_C$};
\draw[mark=square*]  plot coordinates {(6.5,2)};
\node[right] at (7.5,1) {\footnotesize  $b'_C$};
\draw[mark=*]  plot coordinates {(7.5,1)};

\draw   (6.5,5) --(6.5,4) --(7.5,3);
\draw (6.5,2)--(7.5,3) -- (7.5,2)--(7.5,1);
\draw  (6.5,2)-- (7.5,2);
\draw (6.5,2)--(7.5,1);
\qquad
\node[left] at (3,5) {\footnotesize  $a_C$};
\draw[mark=square*]  plot coordinates {(3,5)};
\node[left] at (3,4) {\footnotesize  $a$};
\draw[mark=*]  plot coordinates {(3,4)};
\node[left] at (3,3) {\footnotesize  $1$};
\draw[mark=*]  plot coordinates {(3,3)};
\node[right] at (3,2) {\footnotesize  $b$};
\draw[mark=*]  plot coordinates {(3,2)};
\node[right] at (3,1) {\footnotesize  $b_C$};
\draw[mark=square*]  plot coordinates {(3,1)};
\node[right] at (3,0) {\footnotesize  $b'_C$};
\draw[mark=*]  plot coordinates {(3,0)};
\draw (3,0)--(3,1)--(3,2)--(3,3)--(3,4)--(3,5);
\end{tikzpicture}
\end{center}
\caption{Failure of 1AP for idempotent residuated chains; $\m B$ is on the left, $\m C$ is in the middle and on the right and $A=\{1, a, b\}$. \label{f:APfailsVar}}
\end{figure}

\subsection{Failure of amalgamation for rigid conic idempotent residuated lattices}\label{s:failure for rigid conic}
We show that even with the above restriction of rigidity, amalgamation still fails for conic idempotent residuated lattices, even in the commutative case. Let $\m A$, $\m B$ and $\m C$ be the commutative conic idempotent residuated lattices given in Figure~\ref{f:APfails2}, where the decomposition systems are based on $\{a^*, b^*, 1, b, a\}$ and $\{a^*, c^*, d^*, 1, d, c, a\}$.

\begin{figure}[htbp]
\begin{center}
\begin{tikzpicture}[scale=1]
\node[left] at (0,6) {\footnotesize  $a$};
\draw[mark=*]  plot coordinates {(0,6)};
\draw  (0,6)-- (0,3);
\node[left] at (0,3) {\footnotesize  $1$};
\draw[mark=*]  plot coordinates {(0,3)};
\draw  (0,0)-- (0,3);
\node[right] at (0,0) {\footnotesize  $a^*$};
\draw[mark=*]  plot coordinates {(0,0)};
\node[left] at (3,6) {\footnotesize  $a$};
\draw[mark=*]  plot coordinates {(3,6)};
\draw  (2.5,5.25)--(3,6)-- (3.5,5.25);
\node[left] at (2.5,5.25) {\footnotesize  $b_1$};
\draw[mark=*]  plot coordinates {(2.5,5.25)};
\node[right] at (3.5,5.25) {\footnotesize  $b_2$};
\draw[mark=*]  plot coordinates {(3.5,5.25)};
\draw  (2.5,5.25)--(3,4.5)-- (3.5,5.25);
\node[left] at (3,4.5) {\footnotesize  $b$};
\draw[mark=*]  plot coordinates {(3,4.5)};
\draw  (2.5,3.75)--(3,4.5)-- (3.5,3.75);
\node[left] at (2.5,3.75) {\footnotesize  $b'_1$};
\draw[mark=*]  plot coordinates {(2.5,3.75)};
\node[right] at (3.5,3.75) {\footnotesize  $b'_2$};
\draw[mark=*]  plot coordinates {(3.5,3.75)};
\draw  (2.5,3.75)--(3,3)-- (3.5,3.75);
\node[left] at (3,3) {\footnotesize  $1$};
\draw[mark=*]  plot coordinates {(3,3)};
\draw  (3,3)-- (3,1.5);
\node[right] at (3,1.5) {\footnotesize  $b^*$};
\draw[mark=*]  plot coordinates {(3,1.5)};
\draw  (3,0)-- (3,1.5);
\node[right] at (3,0) {\footnotesize  $a^*$};
\draw[mark=*]  plot coordinates {(3,0)};
\node[left] at (6,6) {\footnotesize  $a$};
\draw[mark=*]  plot coordinates {(6,6)};
\draw  (5.5,5.5)--(6,6)-- (6.5,5.5);
\node[left] at (5.5,5.5) {\footnotesize  $c_1$};
\draw[mark=*]  plot coordinates {(5.5,5.5)};
\node[right] at (6.5,5.5) {\footnotesize  $c_2$};
\draw[mark=*]  plot coordinates {(6.5,5.5)};
\draw  (5.5,5.5)--(6,5)-- (6.5,5.5);
\node[left] at (6,5) {\footnotesize  $c$};
\draw[mark=*]  plot coordinates {(6,5)};
\draw (6,4)--(6,5);
\node[left] at (6,4) {\footnotesize  $d$};
\draw[mark=*]  plot coordinates {(6,4)};
\draw  (5.5,3.5)--(6,4)-- (6.5,3.5);
\node[left] at (5.5,3.5) {\footnotesize  $d_1$};
\draw[mark=*]  plot coordinates {(5.5,3.5)};
\node[right] at (6.5,3.5) {\footnotesize  $d_2$};
\draw[mark=*]  plot coordinates {(6.5,3.5)};
\draw  (5.5,3.5)--(6,3)-- (6.5,3.5);
\node[left] at (6,3) {\footnotesize  $1$};
\draw[mark=*]  plot coordinates {(6,3)};
\draw  (6,2)-- (6,3);
\node[right] at (6,2) {\footnotesize  $d^*$};
\draw[mark=*]  plot coordinates {(6,2)};
\draw  (6,2)-- (6,1);
\node[right] at (6,1) {\footnotesize  $c^*$};
\draw[mark=*]  plot coordinates {(6,1)};
\draw  (6,0)-- (6,1);
\node[right] at (6,0) {\footnotesize  $a^*$};
\draw[mark=*]  plot coordinates {(6,0)};
\end{tikzpicture}
\end{center}
\caption{Failure of amalgamation for commutative rigid conic idempotent residuated lattices: the algebras $\m A$, $\m B$, and $\m C$.\label{f:APfails2}}
\end{figure}
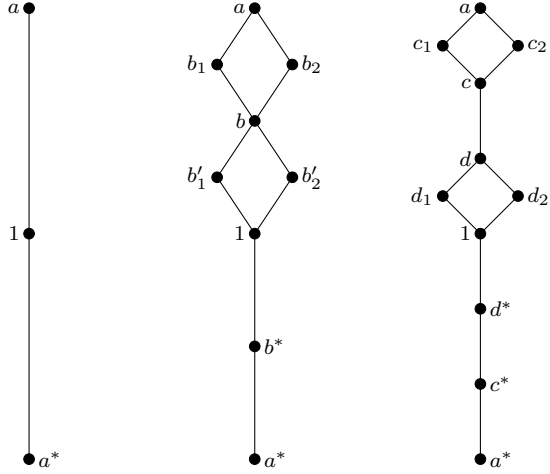

Assume that there is a conic idempotent residuated lattice $\m D$ that serves as an amalgam, where elements of $B$ and $C$ may be identified. 
In the amalgam $\m D$ the element $b$ is an inverse and so is $c$, hence they are conical and thus comparable. If $b<c$ in $\m D$, then since $b_1 \mt b_2 =b$ and since $c$ is conical, we would get $b_1, b_2 \leq c$, hence $a=b_1 \jn b_2\leq  c$, a contradiction. A similar argument with $c_1$ and $c_2$ shows that we cannot have $c<b$ in $\m D$, so $c=b$ in $\m D$. A dual argument involving $b'_1, b'_2$ and $d_1, d_2$ shows that in $\m D$ we must also have $b=d$, implying the contradiction that $c=d$.

More generally, the issue is that the positive blocks are not lattices and the meet of elements of the same block is the top element of the next block below. The condition that captures precisely the conic idempotent residuated lattices for which every block is a lattice is being conjunctive. 
A semiconic  idempotent residuated lattice is called \emph{conjunctive} if it satisfies $\gamma(x \mt y)=\gamma(x) \mt \gamma(y)$.

\begin{lemma}\label{l:DeMorganarelatt}
A conic idempotent residuated lattice is conjunctive iff all of its (positive) blocks are lattices. 
\end{lemma}

\begin{proof}
Assume that all blocks are lattices. Condition  $\gamma(x \mt y)=\gamma(x) \mt \gamma(y)$ holds if $x$ and $y$ are comparable, so we only need to verify it for the case where $x$ and $y$ are in the same block, i.e. $\gamma(x) = \gamma(y)$. Since all blocks are lattices, we also have  $\gamma(x \mt y)=\gamma(x) =\gamma(x) \mt \gamma(y)$.

 Conversely, assume the equation is satisfied. Each block is closed under join, so we only need to check closure under meet. If $x$ and $y$ are in the same block, then $\gamma(x) = \gamma(y)$. Since $\gamma(x \mt y)=\gamma(x) \mt \gamma(y)$, it follows that $x\mt y$ is also in the same block. 
\end{proof}

Therefore, we will restrict our attention to conjunctive rigid conic idempotent residuated lattices.

\subsection{Failure of amalgamation for rigid semiconic idempotent residuated lattices}\label{s:failure for rigid semiconic} As an aside we mention that amalgamation fails even for the variety of rigid semiconic idempotent residuated lattices, by taking a more special V-formation.

\begin{theorem}\label{thm:failures rigid variety}
The amalgamation property fails for the variety of rigid semiconic idempotent residuated lattices, as well as for its commutative subvariety.
\end{theorem}

\begin{proof}
By \cite{FM2022}, if a variety $\mathcal{V}$ has the CEP and the class $\mathcal{V}_{FSI}$ of its finitely subdirectly irreducibles is closed under subalgebras, then $\mathcal{V}$ has the AP iff $\mathcal{V}_{FSI}$ has the 1AP. Taking $\mathcal{V}$ to be the variety of (commutative)  rigid semiconic idempotent residuated lattices, Lemma~\ref{t:CEP}(1) shows that $\mathcal{V}$ has the CEP, while Lemma~\ref{c:FSI}(4) shows that $\mathcal{V}_{FSI}$ is exactly the class of (commutative) rigid conic residuated lattices where $1$ is join irreducible, which is clearly closed under subalgebras. We will show that $\mathcal{V}_{FSI}$ fails the 1AP, so  we will obtain the failure of the AP for $\mathcal{V}$.

Let $\m A$, $\m B$ and $\m C$ be the commutative rigid conic idempotent residuated lattices given in Figure~\ref{f:APfails3}, where the decomposition systems are based on $\{a^*, b^*, 1, b, a\}$ and $\{a^*, c^*, d^*, 1, d, c, a\}$, and note that 1 is join irreducible in them, so they are in $\mathcal{V}_{FSI}$. 
Assume that there is a $\m D \in \mathcal{V}_{FSI}$ that serves as a 1-amalgam; we view $\m C$ as a subalgebra of $\m D$.
In the amalgam $\m D$ the element $g_{\m B}(b)$ is an inverse and so is $c$, hence they are conical and thus comparable.  Since the lattice $[b,a]$ is isomorphic to $\m M_3$, its image under $g_{\m B}$ is either a single point or isomorphic to $\m M_3$. If $g_{\m B}(b)<c$ in $\m D$, then $g_{\m B}(b)<c<a=g_{\m B}(a)$, so $g_{\m B}([b,a])$ cannot be a single point (as $g_{\m B}(b)<g_{\m B}(a)$) nor can it be isomorphic to $\m M_3$, as that would contradict the conicity of $c$ (for example $c\leq g_{\m B}(b_1), g_{\m B}(b_2)$ would imply $c\leq g_{\m B}(b_1) \mt g_{\m B}(b_2)=g_{\m B}(b_1 \mt b_2)= g_{\m B}(b)$). So,  by the conicity of $c$, we get  $c \leq g_{\m B}(b)$. A similar argument using the interval $[1,b]$ shows that $g_{\m B}(b)\leq d$ in $\m D$, so $c\leq d$ in $\m D$, a contradiction to the injectivity of $g_{\m C}$. 
\end{proof}

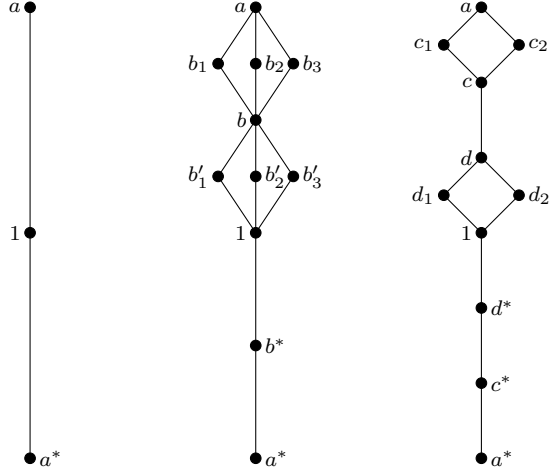
\begin{figure}[htbp]
\begin{center}
\begin{tikzpicture}[scale=1]
\node[left] at (0,6) {\footnotesize  $a$};
\draw[mark=*]  plot coordinates {(0,6)};
\draw  (0,6)-- (0,3);
\node[left] at (0,3) {\footnotesize  $1$};
\draw[mark=*]  plot coordinates {(0,3)};
\draw  (0,0)-- (0,3);
\node[right] at (0,0) {\footnotesize  $a^*$};
\draw[mark=*]  plot coordinates {(0,0)};
\node[left] at (3,6) {\footnotesize  $a$};
\draw[mark=*]  plot coordinates {(3,6)};
\draw  (2.5,5.25)--(3,6)-- (3.5,5.25);
\node[left] at (2.5,5.25) {\footnotesize  $b_1$};
\draw[mark=*]  plot coordinates {(2.5,5.25)};
\node[right] at (3.5,5.25) {\footnotesize  $b_3$};
\draw[mark=*]  plot coordinates {(3.5,5.25)};
\node[right] at (3,5.25) {\footnotesize  $b_2$};
\draw[mark=*]  plot coordinates {(3,5.25)};
\draw  (2.5,5.25)--(3,4.5)-- (3.5,5.25);
\draw  (3,4.5)-- (3,5.25)--(3,6);
\node[left] at (3,4.5) {\footnotesize  $b$};
\draw[mark=*]  plot coordinates {(3,4.5)};
\draw  (2.5,3.75)--(3,4.5)-- (3.5,3.75);
\node[left] at (2.5,3.75) {\footnotesize  $b'_1$};
\draw[mark=*]  plot coordinates {(2.5,3.75)};
\node[right] at (3.5,3.75) {\footnotesize  $b'_3$};
\draw[mark=*]  plot coordinates {(3.5,3.75)};
\node[right] at (3,3.75) {\footnotesize  $b'_2$};
\draw[mark=*]  plot coordinates {(3,3.75)};
\draw  (2.5,3.75)--(3,3)-- (3.5,3.75);
\draw  (3,3)-- (3,3.75)--(3,4.5);
\node[left] at (3,3) {\footnotesize  $1$};
\draw[mark=*]  plot coordinates {(3,3)};
\draw  (3,3)-- (3,1.5);
\node[right] at (3,1.5) {\footnotesize  $b^*$};
\draw[mark=*]  plot coordinates {(3,1.5)};
\draw  (3,0)-- (3,1.5);
\node[right] at (3,0) {\footnotesize  $a^*$};
\draw[mark=*]  plot coordinates {(3,0)};
\node[left] at (6,6) {\footnotesize  $a$};
\draw[mark=*]  plot coordinates {(6,6)};
\draw  (5.5,5.5)--(6,6)-- (6.5,5.5);
\node[left] at (5.5,5.5) {\footnotesize  $c_1$};
\draw[mark=*]  plot coordinates {(5.5,5.5)};
\node[right] at (6.5,5.5) {\footnotesize  $c_2$};
\draw[mark=*]  plot coordinates {(6.5,5.5)};
\draw  (5.5,5.5)--(6,5)-- (6.5,5.5);
\node[left] at (6,5) {\footnotesize  $c$};
\draw[mark=*]  plot coordinates {(6,5)};
\draw (6,4)--(6,5);
\node[left] at (6,4) {\footnotesize  $d$};
\draw[mark=*]  plot coordinates {(6,4)};
\draw  (5.5,3.5)--(6,4)-- (6.5,3.5);
\node[left] at (5.5,3.5) {\footnotesize  $d_1$};
\draw[mark=*]  plot coordinates {(5.5,3.5)};
\node[right] at (6.5,3.5) {\footnotesize  $d_2$};
\draw[mark=*]  plot coordinates {(6.5,3.5)};
\draw  (5.5,3.5)--(6,3)-- (6.5,3.5);
\node[left] at (6,3) {\footnotesize  $1$};
\draw[mark=*]  plot coordinates {(6,3)};
\draw  (6,2)-- (6,3);
\node[right] at (6,2) {\footnotesize  $d^*$};
\draw[mark=*]  plot coordinates {(6,2)};
\draw  (6,2)-- (6,1);
\node[right] at (6,1) {\footnotesize  $c^*$};
\draw[mark=*]  plot coordinates {(6,1)};
\draw  (6,0)-- (6,1);
\node[right] at (6,0) {\footnotesize  $a^*$};
\draw[mark=*]  plot coordinates {(6,0)};
\end{tikzpicture}
\end{center}
\caption{Failure of the 1AP for (commutative) rigid conic idempotent FSI residuated lattices: the algebras $\m A$, $\m B$, and $\m C$. \label{f:APfails3}}
\end{figure}

\subsection{Amalgamation for conjunctive rigid conic idempotent residuated lattices}

The following fact is well known and can also be established easily using residuated frames \cite{GJ2013}.

\begin{lemma}\label{l:sAPlattices}\label{l:sAPBr}
The varieties of lattices and of Brouwerian algebras each have the strong amalgamation property.
\end{lemma}

\begin{theorem}\label{t:sAPqichains}
The class of $^\star$-involutive idempotent residuated chains has the strong amalgamation property.
\end{theorem}

\begin{proof}
Let $\m B$ and $\m C$ be $^\star$-involutive idempotent residuated chains that intersect at a common subalgebra $\m A$. Lemma~\ref{l:crowns} entails that  $\m A$ is a union of one-generated subalgebras of $\m B$ and also a union of one-generated subalgebras of $\m C$. In particular, if a non-trivial one-generated subalgebra of $\m A$ is part of a one-generated subalgebra of $\m B$ (or of $\m C$), then the two one-generated subalgebras are identical.

Therefore, the nested sum decomposition of the corresponding enhanced monoidal preorders afforded by Lemma~\ref{l:crowns} is such that for the corresponding index sets/chains we have that $I_B \cap I_C=I_A$. We can now amalgamate the index sets, as chains to get $I_B \cup I_C$, where the ordering is any linear ordering extending the union of the two orderings. The enhanced monoidal preorder of the algebra $\m D$ is obtained as the nested sum over that index set, where the one-generated subalgebras are exactly the ones of $\m B$ and the ones from $\m C$; if an index is in both  $I_B$ and $I_C$, then it is in $I_A$ and the corresponding one-generated subalgebra in $\m A$,   $\m B$ and $\m C$ are identical.
\end{proof}

\begin{theorem}\label{t:sAPconic}
The class of rigid conjunctive conic idempotent residuated lattices has the strong amalgamation property.
\end{theorem}

\begin{proof}
Let $\m B$ and $\m C$ be a rigid conjunctive conic idempotent residuated lattices and assume that their intersection is a common subalgebra $\m A$. We will produce a rigid conjunctive conic idempotent residuated lattice $\m D$ that will have $\m B$ and $\m C$ as subalgebras. Let $\m S_{\m A}$, $\m S_{\m B}$ and $\m S_{\m C}$ be the residuated chains of quasi-involutive elements of the three algebras. By rigidity, we get that these skeletons are $^\star$-involutive, so by Lemma~\ref{t:sAPqichains} we have a $^\star$-involutive  residuated chain $\m S_{\m D}$ that has $\m S_{\m B}$ and $\m S_{\m C}$ as subalgebras. 

 Also, we denote the corresponding blocks by $\m A_s$, $s \in S_A$,  $\m B_s$, $s \in S_B$, and  $\m C_s$, $s \in S_C$, which are all lattices by Lemma~\ref{l:DeMorganarelatt}. For a negative $s \in S_D$ we have that  $\m B_s$ and  $\m C_s$ are Brouwerian algebras and $\m A_s$ is their common subalgebra. By Lemma~\ref{l:sAPBr} there exists a Brouwerian algebra $\m D_s$ that has  $\m B_s$ and  $\m C_s$  as subalgebras. Also, for a positive $s \in S_D$ we have that  $\m B_s$ and  $\m C_s$ are lattices and $\m A_s$ is their common sublattice. By Lemma~\ref{l:sAPlattices} there exists a lattice $\m D_s$ that has  $\m B_s$ and  $\m C_s$  as sublattices.  Therefore, $(\m S_{\m D}, \{D_s: s \in S_D\})$ is a decomposition system, so by Theorem~\ref{t:decomposition} it corresponds to a conic idempotent residuated lattice $\m D$, which is conjunctive by Lemma~\ref{l:DeMorganarelatt} and rigid since its skeleton is $^\star$-involutive. Also, by Lemma~\ref{l:subsystems}, this system has  $(\m S_{\m B}, \{B_s: s \in S_B\})$ and $(\m S_{\m C}, \{C_s: s \in S_C\})$ as subsystems, so  $\m D$ has $\m B$ and $\m C$ as subalgebras.
\end{proof}

\subsection{Strong amalgamation for rigid conjunctive semiconic idempotent residuated lattices}\label{s:sAP}

In order to lift the amalgamation result from the conic to the semiconic case, we require a variant of the following theorem.
\begin{theorem}[{\cite[Theorem 9]{MMT2014}}]\label{thm:MMT2014}
Let $\mathcal{S}$ be a subclass of a variety $\mathcal{V}$ satisfying the following conditions:
\begin{enumerate}
\item $\mathcal{S}$ contains all subdirectly irreducible members of $\mathcal{V}$;
\item $\mathcal{S}$ is closed under isomorphisms and subalgebras;
\item For any algebra ${\m B}\in\mathcal{V}$ and subalgebra ${\m A}$ of ${\m B}$, if $\Theta$ is a congruence of ${\m A}$ and ${\m A}/\Theta\in\mathcal{S}$, then there exists a congruence $\Psi$ of ${\m B}$ such that $\Psi\cap A^2=\Theta$ and ${\m B}/\Psi\in\mathcal{S}$;
\item $\mathcal{S}$ has the amalgamation property in $\mathcal{V}$.
\end{enumerate}
Then $\mathcal{V}$ has the amalgamation property.
\end{theorem}

To obtain a version of the above for the strong amalgamation property, we exploit the link between strong amalgamation and epimorphism surjectivity. A class $\mathcal{K}$ of similar algebras is said to have \emph{epimorphism surjectivity} or \emph{the ES property} if every epic homomorphism between members of $\mathcal{K}$ is surjective. A subalgebra ${\m A}$ of ${\m B}\in\mathcal{K}$ is said to be \emph{$\mathcal{K}$-epic} if the inclusion homomorphism ${\m A}\hookrightarrow {\m B}$ is an epimorphism, i.e., if for any ${\m C}\in\mathcal{K}$ and any homomorphisms $h,k\colon {\m B}\to {\m C}$ we have that $h\restriction_A=k\restriction_A$ implies $h=k$. Observe that if ${\m A},{\m B}\in\mathcal{K}$ and $h\colon {\m A}\to {\m B}$ is an ${\m K}$-epic homomorphism, then $h({\m A})$ is a ${\m K}$-epic subalgebra of ${\m B}$. Consequently, if $\mathcal{K}$ is closed under taking subalgebras, then $\mathcal{K}$ has the ES property if and only if no ${\m B}\in\mathcal{K}$ has a proper $\mathcal{K}$-epic subalgebra.

 The following is an immediate consequence of \cite[Corollary 2.5.20]{H2001}

\begin{lemma}\label{lem:Hoogland1}
Suppose that $\mathcal{K}$ is a class of similar algebras. If $\mathcal{K}$ has the strong amalgamation property, then $\mathcal{K}$ has the ES property.
\end{lemma}

We will also make use of the following results.

\begin{lemma}[{\cite[Corollary 2.5.23]{H2001}}]\label{lem:Hoogland2}
Suppose that $\mathcal{K}$ is a class of similar algebras that is closed under subalgebras and direct products.
 Then $\mathcal{K}$ has the strong amalgamation property if and only if $\mathcal{K}$ has the amalgamation property and the ES property.
\end{lemma}

\begin{lemma}[{\cite[Theorem 22]{C2018}}]\label{lem:Campercholi}
Let $\mathcal{V}$ be an arithmetical variety such that its class $\mathcal{V}_{FSI}$ of finitely subdirectly irreducible members is a universal class. Then $\mathcal{V}$ has the ES property if and only if $\mathcal{V}_{FSI}$ has the ES property.
\end{lemma}

The next result puts the previous facts together to obtain a variant of Theorem~\ref{thm:MMT2014} for strong amalgamation.

\begin{theorem}\label{thm:strongMMT2014}
Let $\mathcal{V}$ be an arithmetical variety, and denote by $\mathcal{V}_{FSI}$ its class of finitely subdirectly irreducible members. Suppose that:
\begin{enumerate}
\item $\mathcal{V}_{FSI}$ forms a universal class;
\item For any algebra ${\m B}\in\mathcal{V}$ and subalgebra ${\m A}$ of ${\m B}$, if $\Theta$ is a congruence of ${\m A}$ and ${\m A}/\Theta\in\mathcal{V}_{FSI}$, then there exists a congruence $\Psi$ of ${\m B}$ such that $\Psi\cap A^2=\Theta$ and ${\m B}/\Psi\in\mathcal{V}_{FSI}$;
\item $\mathcal{V}_{FSI}$ has the strong amalgamation property in $\mathcal{V}$.
\end{enumerate}
Then $\mathcal{V}$ has the strong amalgamation property.
\end{theorem}
\begin{proof}
By Lemma \ref{lem:Hoogland2}, it suffices to show that $\mathcal{V}$ has both the amalgamation property and the ES property.

We show first that $\mathcal{V}$ has the amalgamation property by using Theorem~\ref{thm:MMT2014} for $\mathcal{S}=\mathcal{V}_{FSI}$. Because each subdirectly irreducible is finitely subdirectly irreducible, condition (1) of Theorem~\ref{thm:MMT2014} is satisfied. Condition (2) of Theorem~\ref{thm:MMT2014} is satisfied because $\mathcal{V}_{FSI}$ forms a universal class by the hypotheses. Conditions (3) and (4) of Theorem~\ref{thm:MMT2014} are satisfied directly from the assumptions, noting that $\mathcal{V}_{FSI}$ having the strong amalgamation property in $\mathcal{V}$ implies that in particular $\mathcal{V}_{FSI}$ has the amalgamation property in $\mathcal{V}$. It follows that $\mathcal{V}$ has the amalgamation property.

To conclude the proof, we show that $\mathcal{V}$ has the ES property. To do so, we show that $\mathcal{V}_{FSI}$ has the ES property and apply Lemma~\ref{lem:Campercholi}. Let ${\m B}\in\mathcal{V}_{FSI}$ and suppose that ${\m A}$ is a proper subalgebra of ${\m B}$. As $\mathcal{V}_{FSI}$ is a universal class, we have ${\m A}\in{\mathcal{V}_{FSI}}$. Renaming elements as necessary, let ${\m C}$ be an isomorphic copy of ${\m B}$ with $B\cap C = A$ and let $h\colon {\m B}\to {\m C}$ be an isomorphism with $h[A]=A$. Then ${\m C}\in\mathcal{V}_{FSI}$. By assumption, the reduced V-formation $({\m A},{\m B},{\m C})$ has a strong amalgam ${\m D}\in\mathcal{V}$. Let $s\colon {\m D}\to \prod_{i\in I} {\m D}_i$ be a subdirect representation of ${\m D}$ where each ${\m D}_i$ is subdirectly irreducible (and hence in $\mathcal{V}_{FSI}$). Because ${\m A}$ is a proper subalgebra of ${\m B}$, there exists $x\in B\setminus A$. Since $h(A)=A$ we have $h(x)\in C\setminus A$. Since $B\cap C = A$, we have $h(x)\neq x$. Thus $s(h(x))\neq s(x)$, and consequently there exists $i\in I$ such that $s(h(x))(i)\neq s(x)(i)$. Let $\pi\colon {\m D}\to {\m D}_i$ be the canonical projection map, and let $j\colon {\m B}\to {\m D}$ be the inclusion map. By construction, the maps $\pi\circ s\circ h,\pi\circ s\circ j\colon {\m B}\to {\m D}_i$ are equal on $A$, but $(\pi\circ s\circ h)(x)\neq(\pi\circ s\circ j)(x)$. Thus ${\m A}$ is not a $\mathcal{V}_{FSI}$-epic subalgebra of ${\m B}$. It follows that ${\m B}$ has no proper $\mathcal{V}_{FSI}$-epic subalgebras, whence $\mathcal{V}_{FSI}$ has the ES property as desired.
\end{proof}

\begin{theorem}\label{t:SAPsemiconic}
Let $\mathcal{V}$ be a variety of semiconic idempotent residuated lattices, and suppose that $\mathcal{V}_{FSI}$ has the strong amalgamation property in $\mathcal{V}$. Then $\mathcal{V}$ has the strong amalgamation property, and hence the ES property.
\end{theorem}

\begin{proof}
We deploy Theorem~\ref{thm:strongMMT2014}. The variety $\mathcal{V}$ is arithmetical by \cite[p.~94]{GJKO2007}. By Lemma~\ref{l:FSI}(4), $\mathcal{V}_{FSI}$ consists precisely of the conic algebras in $\mathcal{V}$ for which $1$ is join-irreducible (i.e., satisfying the universal sentence $(\forall x,y)(x \jn y =1 \Rightarrow x=1$ or $y=1)$), so it is a universal class and thus condition (1) of Theorem~\ref{thm:strongMMT2014} is satisfied. Condition (3) of Theorem~\ref{thm:strongMMT2014} holds by assumption. For condition (2), suppose that ${\m A}$ and ${\m B}$ are semiconic idempotent residuated lattices, that ${\m A}$ is a subalgebra of ${\m B}$, and that $\Theta$ is a congruence of ${\m A}$ such that ${\m A}/\Theta$ is finitely subdirectly irreducible, hence also conic by Lemma~\ref{l:FSI}(4).

 Let $F$ be the congruence filter of ${\m A}$ corresponding to $\Theta$. Since $F_B:=\upset_{\m B} F$ is a filter of $\m B$ containing $1$, it follows from Lemma~\ref{l:congruencefilters2}(3) that $\langle F_B \rangle=\upset_{\m B} \{t_n(y): n \in \mathbb{N}, y \in F_B\}$. Since $t_n$ is monotone this is equal to $\upset_{\m B} \{t_n(y): n \in \mathbb{N}, y \in F\}$ and since $F$ is closed under $t_n$ it is equal to $\upset_{\m B} F=F_B$. Thus $F_B$ is a congruence filter of $\m B$. Also note that $F_B \cap A=F$, because if $y \in F_B \cap A$, then $y \in A$ and $x \leq y$ for some $x \in F$, so $y \in F$. Therefore, $F_B$ is an element of the poset
\[P=\{G\subseteq B : G\text{ is a congruence filter of }{\m B}\text{ and }F=G\cap A\},\]
which is, therefore, non-empty.

By Zorn's Lemma, $P$ has a maximal element $G$. We will show that $G$ is prime. Toward a contradiction, let $x,y \in B$ such that $x\join y \in G$, $x\notin G$, and $y\notin G$. Since $1\in G$ and ${\m B}$ satisfies distributivity at $1$, we may assume without loss of generality that $x$ and $y$ are negative (meeting with $1$ if necessary). By the maximality of $G$ in $P$, each of the congruence filters $\langle G\cup\{x\}\rangle$ and $\langle G\cup\{y\}\rangle$ is not in $P$. Thus each of $H_x=\langle G\cup \{x\}\rangle \cap A$ and $H_y=\langle G\cup \{y\}\rangle \cap A$ properly contains $F$, whence there exist elements $a\in H_x\setminus F$ and $b\in H_y\setminus F$. Without loss of generality we may assume that $a$ and $b$ are both negative (considering the result of meeting with $1$ if necessary).
 By Lemma~\ref{l:congruencefilters2}(4), there exists $g,h\in G^-$ and $m,n \in \mathbb{N}$ such that $g \mt s_m(x)\leq a$ and $h \mt s_n(y)\leq b$. Now since $x\join y\in G$, we have that $[1]_G\leq [x]_G\join [y]_G$, and since we chose $x$ and $y$ to be negative we have $[x]_G\join [y]_G=[1]_G$. By Lemma~\ref{l:adding_conjugates}(3), it follows that $[s_m(x)]_G\join [s_n(y)]_G=[1]_G$, so:
$$[1]_G = [s_m(x)]_G\join [s_n(y)]_G = [g\meet s_m(x)]_G\join [h\meet s_n(y)]_G\leq [a]_G\join [b]_G.$$
From $[1]_G\leq [a]_G\join [b]_G$ we obtain $a \jn b \in G$; since $a \jn b \in A$, we get $a \jn b \in G \cap A=F$, so  $[1]_F\leq [a]_F\join [b]_F$, and since $a$ and $b$ are negative we get $[a]_F\join [b]_F=[1]_F$. Because ${\m A}/F$ is finitely subdirectly irreducible by Lemma~\ref{l:FSI}(4), this yields $[a]_F=[1]_F$ or $[b]_F=[1]_F$, so $a\in F$ or $b\in F$. This is a contradiction to the choice of $a$ and $b$, so it follows that $G$ is prime.

To complete the proof, observe that $G$ being prime implies that $[1]_G$ is join-irreducible in the quotient ${\m B}/G$, so by Lemma~\ref{l:FSI}(4) we have ${\m B}/G\in\mathcal{V}_{FSI}$.
\end{proof}

The following is immediate from combining Theorem~\ref{t:SAPsemiconic} with Theorem~\ref{t:sAPconic}.

\begin{theorem}\label{t:DeMSemSAP}
The variety of rigid conjunctive semiconic idempotent residuated lattices has the strong amalgamation property, hence the ES property.
\end{theorem}

\subsection{Subvarieties.}\label{s:amalg subvarieties} The method of obtaining Theorem~\ref{t:DeMSemSAP} applies, \emph{mutatis mutandis}, to a host of subvarieties. We do not catalog these exhaustively, and only discuss a few prominent examples. Most notably, since $^\star$-involutive idempotent residuated chains are finitely subdirectly irreducible, rigid (by Lemma~\ref{l:weaklyinvolutive}), and conjunctive, we immediately obtain strong amalgamation for the corresponding semilinear variety by combining Theorems~\ref{t:sAPqichains} and \ref{t:SAPsemiconic}:

\begin{theorem}\label{t:SAPsemilinearStarInv}
The variety of $^\star$-involutive semilinear idempotent residuated lattices has the strong amalgamation property, and hence the ES property.
\end{theorem}

The commutative $^\star$-involutive idempotent residuated chains are exactly the odd Sugihara monoids, so in the commutative setting Theorems~\ref{t:sAPqichains} and \ref{t:SAPsemilinearStarInv} are well-known (see e.g. \cite{GR2012}). The fact that odd Sugihara monoid chains admit strong amalgamation allows us to prove the follow result:

\begin{theorem}\label{t:SAPcommsemiconic}
The variety of commutative conjunctive semiconic idempotent residuated lattices has the strong amalgamation property, hence the ES property.
\end{theorem}

\begin{proof}
Observe that every commutative semiconic idempotent residuated lattice is rigid. Moreover, it follows from Lemma~\ref{c:simple mult}(4) that a conic idempotent residuated lattice is commutative iff its quasi-involutive skeleton is commutative (i.e., an odd Sugihara monoid). Inspection of the proofs of Theorems~\ref{t:sAPqichains} and Theorem~\ref{t:sAPconic} shows that commutativity is preserved in taking the amalgams in each case, so the result follows by Theorem~\ref{t:SAPsemiconic}.
\end{proof}

Distributive lattices do not have the strong amalgamation property \cite{FrGr1990}, so we cannot apply the same methodology to distributive subvarieties of rigid conjunctive semiconic idempotent residuated lattices. However, we may obtain the amalgamation property:

\begin{theorem}\label{c:APsemiconic}
Each of the following subvarieties of rigid conjunctive semiconic idempotent residuated lattices has the amalgamation property.
\begin{enumerate}
\item The variety of distributive, rigid, and conjunctive semiconic idempotent residuated lattices.
\item The variety of distributive, commutative, and conjunctive semiconic idempotent residuated lattices.
\end{enumerate}
\end{theorem}

\begin{proof}
1. Let ${\m A}$ be a rigid conjunctive conic idempotent residuated lattice, and let $({\m S}_A, \{{\m A}_s : s\in S_A\})$ be its decomposition system. Then ${\m A}$ is distributive iff each ${\m A}_s$, $s>1$, is a distributive lattice with designated top element. The variety of distributive lattices with top element has the amalgamation property, so by repeating the proof of Theorem~\ref{t:sAPconic} (and taking amalgams rather than strong amalgams of positive blocks) we obtain that the class of distributive, rigid, and conjunctive conic idempotent residuated lattices has the amalgamation property. Lifting the amalgamation property to the variety of distributive, rigid, and conjunctive semiconic idempotent residuated lattices is a routine application of Theorem~\ref{thm:MMT2014}, using the proof of Theorem~\ref{t:SAPsemiconic} to establish Condition 3 of Theorem~\ref{thm:MMT2014}.

2. This follows immediately from the proofs of 1 and Theorem~\ref{t:SAPcommsemiconic}.
\end{proof}

Although topped distributive lattices only have the amalgamation property, the trivial subvariety has strong amalgamation thanks to the paucity of V-formations. The recent paper \cite{C2020} studies the subvariety of commutative semiconic idempotent residuated lattices that satisfy the identity $(x\join 1)^{**} = x\join 1,$
called \emph{semiconic generalized Sugihara monoids}. Although it is not discussed in \cite{C2020}, the conic members of this subvariety can be specified by stipulating that strictly positive blocks in the decomposition system are trivial:

\begin{proposition}\label{p:generalizedSugihara}
Let ${\m A}$ be a conic idempotent residuated lattice, and let $({\m S}, \{{\m A}_s : s\in S\})$ be its decomposition system. Then ${\m A}$ is a semiconic generalized Sugihara monoid if and only if ${\m S}$ is an odd Sugihara monoid and ${\m A}_s = \{s\}$ for each strictly positive $s\in S$.
\end{proposition}

\begin{proof}
Suppose that ${\m A}$ is a semiconic generalized Sugihara monoid. The skeleton ${\m S}$ of ${\m A}$ is an odd Sugihara monoid by commutativity. Moreover, if $s\in S$ with $s>1$ and $a\in A$ with $a^{**}=s$, then $a=a\join 1$ and hence $s=a^{**}=(a\join 1)^{**}=a\join 1 = a$, so $A_s=\{s\}$.

Conversely, suppose ${\m S}$ is an odd Sugihara monoid and ${\m A}_s=\{s\}$ for every strictly positive $s\in S$. That ${\m A}$ is commutative follows from the fact that its skeleton is commutative. If $a\in A$, then either $a\join 1=1$ or $a\join 1>1$. If $a\join 1=1$, then $(a\join 1)^{**}=1=a\join 1$ is immediate. If $a\join 1>1$, then $s=(a\join 1)^{**}$ being a strictly positive element of ${\m S}$ implies $A_s=\{s\}$, so $a\join 1=s=(a\join 1)^{**}$.
\end{proof}

The conic members of the subvariety of \emph{central} semiconic generalized Sugihara monoids, axiomatized relative to semiconic generalized Sugihara monoids by $1\leq x^{**}\join (x^{**}\to x),$ may further be characterized by stipulating that the only non-trivial block is the block of the identity element:
\begin{proposition}
Let ${\m A}$ be a conic idempotent residuated lattice, and let $({\m S}, \{{\m A}_s : s\in S\})$ be its decomposition system. Then ${\m A}$ is a central semiconic generalized Sugihara monoid if and only if ${\m S}$ is an odd Sugihara monoid and ${\m A}_s = \{s\}$ for each $s\in S\setminus\{1\}$.
\end{proposition}

\begin{proof}
Observe that if ${\m A}$ is a central semiconic generalized Sugihara monoid and $a^{**}=s<1$, where $s\in S$, then $1\leq a^{**}\join (a^{**}\to a)$ implies $1\leq a^{**}\to a$ by the conicity of $a^{**}$. Hence $s=a^{**}=a$, so $A_s=\{s\}$. Conversely, if $A_s=\{s\}$ for all $s<1$, then an easy calculation shows that $1\leq a^{**}\join (a^{**}\to a)$ for all $a\in A$. The rest follows from Proposition~\ref{p:generalizedSugihara}.
\end{proof}

By using a categorical equivalence generalizing \cite{GR2015}, \cite{C2020} obtains the strong amalgamation property for the variety of central semiconic generalized Sugihara monoids. However, it does not discuss amalgamation of semiconic generalized Sugihara monoids beyond the central case. Because the conic members in each variety can be described by stipulating that certain blocks are trivial, it follows from the results of this section that both varieties have the strong amalgamation property:

\begin{corollary}\label{c:SGSM}
Each of the following varieties has the strong amalgamation property, hence the ES property.
\begin{enumerate}
\item The variety of semiconic generalized Sugihara monoids.
\item The variety of central semiconic generalized Sugihara monoids.
\end{enumerate}
\end{corollary}

\begin{proof}
Both results follow from applying Theorem~\ref{t:SAPsemiconic}, and by noting that in each case the amalgam obtained in the proof of Theorem~\ref{t:sAPconic} may be taken to preserve trivial blocks.
\end{proof}

For brevity, we now assign names to several of the varieties of this study:
\begin{itemize}
\item $\mathcal{R}$: rigid, conjunctive semiconic idempotent residuated lattices.
\item $\mathcal{ISL^\star}$: $^\star$-involutive semilinear idempotent residuated lattices.
\item $\mathcal{DR}$: distributive, rigid, and conjunctive semiconic idempotent residuated lattices.
\item $\mathcal{CR}$: commutative, conjunctive semiconic idempotent residuated lattices.
\item $\mathcal{DCR}$: distributive, commutative, and conjunctive semiconic idempotent residuated lattices.
\item $\mathcal{SGSM}$: semiconic generalized Sugihara monoids.
\item $\mathcal{CSGSM}$: central semiconic generalized Sugihara monoids.
\end{itemize}
For each of the listed varieties $\mathcal{V}$, we denote by $\vdash_\mathcal{V}$ the logic corresponding to $\mathcal{V}$. Each of these is an axiomatic extension of $\vdash_\CIL$.
By applying Theorem~\ref{t:deductive interpolation} and the previous theorem, we obtain from Theorems~\ref{t:DeMSemSAP}, \ref{t:SAPsemilinearStarInv}, \ref{t:SAPcommsemiconic}, \ref{c:APsemiconic}, and Corollary~\ref{c:SGSM} the following result:

\begin{theorem}\label{t:application to interpolation}
Each of $\vdash_\mathcal{R}$, $\vdash_\mathcal{ISL^\star}$, $\vdash_\mathcal{DR}$, $\vdash_\mathcal{CR}$, $\vdash_\mathcal{DCR}$, $\vdash_\mathcal{SGSM}$, and $\vdash_\mathcal{CSGSM}$ has the deductive interpolation property.
\end{theorem}

We may apply Theorem~\ref{t:Beth def} to all but two of the logics mentioned in the previous theorem, arriving at our final result:

\begin{theorem}\label{r:application to Beth}
Each of $\vdash_\mathcal{R}$, $\vdash_\mathcal{ISL^\star}$, $\vdash_\mathcal{CR}$, $\vdash_\mathcal{SGSM}$, and $\vdash_\mathcal{CSGSM}$ has the projective Beth definability property.
\end{theorem}





\bibliographystyle{elsarticle-harv} 
\bibliography{conic} 





\end{document}